\documentclass[letterpaper]{amsart}
\usepackage[margin=1in]{geometry} 
\usepackage[utf8]{inputenc}
\usepackage{amsmath,amsthm,amssymb,amsfonts,mathrsfs}
\usepackage{setspace}
\usepackage{dsfont}
\usepackage{tikz-cd}
\usepackage{xfrac}
\usepackage{faktor}
\usepackage{tikz}
\usepackage{wrapfig}
\usepackage{graphicx}
\usepackage{verbatim} 
\usepackage[all]{xy}
\usepackage{cite}
\usepackage{bookmark}

\newcommand{\RomanNumeralCaps}[1]
    {\MakeUppercase{\romannumeral #1}}

\theoremstyle{definition}
\newtheorem{theorem}{Theorem}[section]
\newtheorem{lemma}[theorem]{Lemma}
\newtheorem{proposition}[theorem]{Proposition}
\newtheorem{corollary}[theorem]{Corollary}

\theoremstyle{remark}
\newtheorem{remark}[theorem]{Remark}

\numberwithin{equation}{section}

\title[Eigenfunction Restriction Estimates]{Eigenfunctions restriction estimates for curves with nonvanishing geodesic curvatures in compact Riemannian surfaces with nonpositive curvature}
\author[C. Park]{Chamsol Park}
\address{Department of Mathematics and Statistics, University of New Mexico, Albuquerque, NM 87131, USA}
\email{parkcs@unm.edu}
\date{}

\usepackage{pgfplots}
\pgfplotsset{compat=1.16}

\begin{document}

\begin{abstract}
    For $2\leq p<4$, we study the $L^p$ norms of restrictions of eigenfunctions of the Laplace-Beltrami operator on smooth compact $2$-dimensional Riemannian manifolds. Burq, G\'erard, and Tzvetkov \cite{BurqGerardTzvetkov2007restrictions}, and Hu \cite{Hu2009lp} found the eigenfunction estimates restricted to a curve with nonvanishing geodesic curvatures. We will explain how the proof of the known estimates helps us to consider the case where the given smooth compact Riemannian manifold has nonpositive sectional curvatures. For $p=4$, we will also obtain a logarithmic analogous estimate, by using arguments in Xi and Zhang \cite{XiZhang2017improved}, Sogge \cite{Sogge2017ImprovedCritical}, and Bourgain \cite{Bourgain1991Besicovitch}.
\end{abstract}

\maketitle

\section{Introduction}

Let $(M, g)$ be a smooth compact $n$-dimensional Riemannian manifold without boundary and $\Sigma$ a $k$-dimensional embedded submanifold. We denote by $\Delta_g$ the associated negative Laplace-Beltrami operator on $M$ so that the spectrum of $-\Delta_g$ is discrete. If $e_\lambda$ is any $L^2$ normalized eigenfunction, then we write
\begin{align*}
    \Delta_g e_\lambda = -\lambda^2 e_\lambda,\quad \|e_\lambda\|_{L^2(M)}=1, \quad \lambda\geq 0.
\end{align*}
Here $L^p (M)$ is the space of $L^p$ functions with respect to the Riemannian measure. There have been many ways of measuring possible concentrations of the eigenfunctions of the Laplace-Beltrami operator on a manifold so far. One of the ways of measuring the possible concentrations of $e_\lambda$ on a manifold is to study the possible growth of the $L^p$ norm of the restrictions of $e_\lambda$ to submanifolds of $M$. This article deals with the concentrations of the restrictions of $e_\lambda$ to a curve with nonvanishing geodesic curvatures of $2$-dimensional manifold $M$.

We first review the previous results. We consider the operator $\mathds{1}_{[\lambda, \lambda+h(\lambda)]}(\sqrt{-\Delta_g})$, which projects a function onto all eigenspaces of $\sqrt{-\Delta_g}$ whose corresponding eigenvalue lies in $[\lambda, \lambda+h(\lambda)]$, which are approximations to eigenfunctions, or quasimodes. Recall that the exact eigenfunctions can also be considered as quasimodes in that
\begin{align*}
    \mathds{1}_{[\lambda, \lambda+h(\lambda)]} (\sqrt{-\Delta_g}) e_\lambda=e_\lambda.
\end{align*}
For $h(\lambda)\equiv 1$ case, there are well-known estimates of Sogge \cite{Sogge1988concerning} which state that, for a uniform constant $C>0$ depending only on $M$,
\begin{align}\label{Sog88}
    \| \mathds{1}_{[\lambda, \lambda+1]} (\sqrt{-\Delta_g}) \|_{L^2 (M) \to L^p (M)} \leq C \lambda^{\delta(p, n)},\quad \lambda\geq 1,
\end{align}
where
\[
\delta(p, n)=
\begin{cases}
\frac{n-1}{2}-\frac{n}{p}, & \text{if}\quad p_c\leq p\leq \infty, \\
\frac{n-1}{2}\left(\frac{1}{2}-\frac{1}{p} \right), & \text{if}\quad 2\leq p\leq p_c,
\end{cases}\quad\quad
p_c=\frac{2(n+1)}{n-1}.
\]
It follows immediately that
\begin{align}\label{Sog88 exact eigfcns}
    \| e_\lambda \|_{L^p (M)}\leq C\lambda^{\delta (p, n)}.
\end{align}
The exponent $p_c$ is a so-called ``critical'' exponent. The work of Sogge \cite{Sogge1988concerning} (see also \cite[pp.142-145]{Sogge1993fourier}) also showed that the estimates \eqref{Sog88} are sharp in that there exist a function $f$, or a quasimode, such that
\begin{align*}
    \|\mathds{1}_{[\lambda, \lambda+1]} (\sqrt{-\Delta_g}) f \|_{L^p (M)}\geq c \lambda^{\delta(p, n)} \|f\|_{L^2 (M)},\quad \text{for some uniform } c>0.
\end{align*}
Sogge \cite{Sogge1986oscillatory} showed that the estimates \eqref{Sog88 exact eigfcns} are sharp for an infinite family of exact eigenfunctions $e_\lambda$ in that
\begin{align*}
    \|e_\lambda \|_{L^p (\mathbb{S}^n)} \geq c \lambda^{\delta(p, n)},\quad \text{for some uniform } c>0,
\end{align*}
where $M$ is the round sphere. Specifically, the $p_c\leq p\leq \infty$ case is saturated by a sequence of the zonal harmonics on the sphere, whereas $2\leq p\leq p_c$ case is sharp due to the highest weight spherical harmonics on the sphere. The estimates \eqref{Sog88} or \eqref{Sog88 exact eigfcns} are sometimes called ``universal estimates'' since they are satisfied on any smooth compact Riemannian manifold. If one assumes nonpositive curvatures or no conjugate points on $M$, the phenomenas are a bit different. For example, the geodesic flow in negatively curved manifolds behave chaotically, and so, there may be smaller concentration of the restrictions of eigenfunctions of the Laplace-Beltrami operator to geodesics in the negatively curved manifolds.

If $(M, g)$ has nonpositive sectional curvatures, we have some estimates of the case $h(\lambda)=(\log \lambda)^{-1}$
\begin{align}\label{LogM}
    \| \mathds{1}_{[\lambda,\lambda+(\log \lambda)^{-1}]} (\sqrt{-\Delta_g}) \|_{L^2(M)\to L^p (M)} \leq C_p \frac{\lambda^{\delta(p, n)}}{(\log \lambda)^{\sigma(p, n)}},
\end{align}
for some constant $\sigma(p, n)>0$. By using methods of B\'erard \cite{Berard1977onthewaveequation}, Hassell and Tacy showed in \cite{HassellTacy2015improvement} that the estimates \eqref{LogM} hold for $\sigma(p, n)=\frac{1}{2}$ with $p_c < p\leq \infty$. This case was also recently investigated by Canzani and Galkowski \cite{CanzaniGalkowski2020Growth} under more general hypotheses. The case $2<p\leq p_c$ was investigated by Blair and Sogge \cite{BlairSogge2017refined, BlairSogge2018concerning, BlairSogge2019logarithmic}, Sogge \cite{Sogge2011KakeyaNikodym}, and Sogge and Zelditch \cite{SoggeZelditch2014eigenfunction}.

There are analogues of \eqref{Sog88} and \eqref{LogM} when we replace $\mathds{1}_{[\lambda, \lambda+h(\lambda)]}$ by $\mathcal{R}_{\Sigma}\circ\mathds{1}_{[\lambda, \lambda+h(\lambda)]}$, where $\mathcal{R}_\Sigma$ denotes the restriction map as $\mathcal{R}_\Sigma f=\left. f\right|_\Sigma$. The metric $g$ endows $\Sigma$ with induced measures, and thus, we can also consider the Lebesgue spaces $L^p (\Sigma)$. Burq, G\'erard, and Tzvetkov \cite{BurqGerardTzvetkov2007restrictions}, and Hu \cite{Hu2009lp} studied estimates of the form
\begin{align}\label{BGT}
    \| \mathcal{R}_\Sigma \circ \mathds{1}_{[\lambda, \lambda+1]}(\sqrt{-\Delta_g}) \|_{L^2 (M) \to L^p(\Sigma)} \leq C \lambda^{\rho_k(p, n)},\quad \lambda\geq 1,
\end{align}
where
\[
\rho_k (p, n)=
\begin{cases}
\frac{n-1}{4}-\frac{n-2}{2p}, & \text{if } k=n-1 \text{ and } 2\leq p\leq \frac{2n}{n-1},\\
\frac{n-1}{2}-\frac{n-1}{p}, & \text{if } k=n-1, \text{ and } \frac{2n}{n-1}\leq p\leq \infty,
\end{cases}
\]
which in turn implies that
\begin{align}\label{BGT exact eigfcns}
    \| e_\lambda \|_{L^p (\Sigma)}\leq C\lambda^{\rho_k(p, n)}.
\end{align}
These estimates are also called universal estimates since they hold on any smooth compact Riemannian manifold. The exponent $\frac{2n}{n-1}$ is the critial exponent in this case. They also considered other cases $k\leq n-2$, but we focus on $k=n-1$ here and below, since we will talk about $(n, k)=(2, 1)$ mainly in this paper. Using methods in semiclassical analysis, Tacy \cite{Tacy2010Semiclassical} considered the same estimates as special cases of estimates for quasimodes. In \cite{BurqGerardTzvetkov2007restrictions}, Burq, G\'erard, and Tzvetkov also showed the estimates \eqref{BGT} are sharp by showing that, for all $\lambda \geq 1$, there exists a function $f=f_\lambda$ such that
\begin{align*}
    \| \mathcal{R}_\Sigma \circ \mathds{1}_{[\lambda, \lambda+1]} (\sqrt{-\Delta_g}) f\|_{L^p (\Sigma)} \geq c \lambda^{\rho_k (p, n)} \| f\|_{L^2 (M)},\quad \text{for some uniform } c>0,
\end{align*}
on any compact Riemannian manifold, and the estimates \eqref{BGT exact eigfcns} are sharp by showing that
\begin{align*}
    \| e_\lambda \|_{L^p (\Sigma)} \geq c \lambda^{\rho_k (p, n)}, \quad \text{for some } c>0,
\end{align*}
if the $e_\lambda$ are the zonal harmonics or the highest weight spherical harmonics on the round sphere $M=\mathbb{S}^n$.

Focusing on the case $(n, k, p)=(2, 1, 2)$ in \eqref{BGT}, they showed
\begin{align*}
    \| \mathcal{R}_\Sigma \circ \mathds{1}_{[\lambda, \lambda+1]}(\sqrt{-\Delta_g}) \|_{L^2 (M) \to L^2(\Sigma)} \leq C \lambda^{\frac{1}{4}},\quad \lambda\geq 1.
\end{align*}
Burq, G\'erard, and Tzvetkov \cite{BurqGerardTzvetkov2007restrictions}, and Hu \cite{Hu2009lp} showed that if $\Sigma$ is a curve $\gamma$ with nonvanishing geodesic curvatures, then $\lambda^{1/4}$ can be replaced by $\lambda^{1/6}$.
\begin{theorem}[Theorem 2 in \cite{BurqGerardTzvetkov2007restrictions}, Theorem 2 in \cite{Hu2009lp}]\label{Theorem Universal Estimates}
Suppose $\dim M=2$ and the curve $\gamma$ is a unit-length curve having nonvanishing geodesic curvatures, that is, $g(D_t \gamma', D_t \gamma')\not=0$, where $D_t$ is the covariant derivatives along the curve $\gamma$. We then have that, for a uniform constant $C$,
\begin{align}\label{Universal Estimates for Nonvanishing Curvature}
    \| \mathcal{R}_\gamma \circ \mathds{1}_{[\lambda, \lambda+1]}(\sqrt{-\Delta_g}) \|_{L^2 (M) \to L^2(\gamma)} \leq C \lambda^{\frac{1}{6}},\quad \lambda\gg 1.
\end{align}
If follows immediately from this that
\begin{align}\label{Universal Estimates exact eigfcns}
    \| e_\lambda \|_{L^2(\gamma)} \leq C \lambda^{\frac{1}{6}},\quad \lambda\gg 1.
\end{align}
\end{theorem}
This estimate was generalized to a higher dimensional analogue in \cite[Theorem 1.4]{Hu2009lp}. Again, using semiclassical analytic methods, Hassell and Tacy \cite{HassellTacy2012CurvedHypersurfaces} obtained estimates generalized to quasimodes.

Again, Burq, G\'erard, and Tzvetkov \cite[Section 5.2 and Remark 5.4]{BurqGerardTzvetkov2007restrictions} showed that the estimate \eqref{Universal Estimates for Nonvanishing Curvature} is sharp by finding a function $f$ as above, and the estimate \eqref{Universal Estimates exact eigfcns} is also sharp when $M$ is the standard sphere $\mathbb{S}^2$, and $\gamma$ is any curve with nonvanishing geodesic curvatures. See also Tacy \cite{Tacy2018Constructing} for constructing sharp examples for exact eigenfunctions on $\mathbb{S}^n$ or quasimodes. We will prove Theorem \ref{Theorem Universal Estimates} again in this article in a different point of view, since we need estimates in our proof to prove Theorem \ref{Theorem Log Improvement}, which will be illustrated below.

Similarly, when $(M, g)$ has nonpositive (or constant negative) curvatures, it has been studied that
\begin{align}\label{LogSub}
    \| \mathcal{R}_\Sigma \circ \mathds{1}_{[\lambda, \lambda+(\log \lambda)^{-1}]} (\sqrt{-\Delta_g}) \|_{L^2(M) \to L^p(\Sigma)} \leq C \frac{\lambda^{\rho_k (p, n)}}{(\log \lambda)^{\sigma_k(p, n)}},\quad \lambda\geq 1,
\end{align}
for some constant $\sigma_k(p, n)>0$ with the same constant $\rho_k (p, n)$ as in \eqref{BGT}. In \cite{Chen2015improvement}, Chen obtained $\sigma_k(p, n)=\frac{1}{2}$ in \eqref{LogSub} for the cases $k=n-1$ with $p>\frac{2n}{n-1}$.

For $k=1$, there also have been studies of critical or subcritical exponent. For subcritical cases, Sogge and Zelditch \cite{SoggeZelditch2014eigenfunction} showed that for any $\epsilon>0$ there exists a $\lambda (\epsilon)<\infty$ such that
\begin{align}\label{SoggeZelditch Result}
    \sup_{\gamma \in \Pi} \left(\int_\gamma |e_\lambda|^2\:ds \right)^{1/2} \leq \epsilon \lambda^{\frac{1}{4}},\quad \lambda >\lambda(\epsilon),\; \dim M=2,
\end{align}
where $\Pi$ is the space of all unit-length geodesics in $M$, and $ds$ is the arc-length measure on $\gamma$. By using the methods in \cite{SoggeZelditch2014eigenfunction} with Toponogov's comparison theorem, Blair and Sogge \cite{BlairSogge2018concerning} obtained $\sigma_1 (2, 2)=\frac{1}{4}$, which is an improvement of $\epsilon$ in \eqref{SoggeZelditch Result}. The works of Blair \cite{Blair2018logarithmic}, and Xi and Zhang \cite{XiZhang2017improved} obtain $\sigma_1(4, 2)=\frac{1}{4}$ for (unit-length) geodesics, which is a critical exponent in that $p=\frac{2n}{n-1}$.

As in the universal estimates, for the case $(n, k, p)=(2, 1, 2)$ in \eqref{LogSub},  we can expect that $\lambda^{1/4}$ may be replaced by $\lambda^{1/6}$ if $\gamma$ has nonvanishing geodesic curvatures, analogous to \eqref{Universal Estimates for Nonvanishing Curvature}. Moreover, by \cite[Theorem 2]{BurqGerardTzvetkov2007restrictions} and \cite[Theorem 1.2]{Hu2009lp}, we know that
\begin{align*}
    \| \mathcal{R}_\gamma \circ \mathds{1}_{[\lambda, \lambda+(\log \lambda)^{-1}]}(\sqrt{-\Delta_g}) \|_{L^2 (M) \to L^p(\gamma)} \leq C \lambda^{\frac{1}{3}-\frac{1}{3p}},\quad\lambda\geq 1, \quad 2\leq p\leq 4.
\end{align*}
We want to show the analogue of this for $2\leq p<4$ in the presence of nonpositive sectional curvatures.

\begin{theorem}\label{Theorem Log Improvement}
Let $(M, g)$ be a compact $2$-dimensional smooth Riemannian manifold (without boundary) with nonpositive sectional curvatures pinched between $-1$ and $0$. Also suppose that $\gamma$ is a fixed unit-length curve with $g(D_t \gamma', D_t \gamma')\not=0$. Then, for a uniform constant $C_p>0$ and $\lambda\geq 1$,
\begin{align}\label{Our aim in this paper}
    \| \mathcal{R}_\gamma \circ \mathds{1}_{[\lambda, \lambda+(\log \lambda)^{-1}]}(\sqrt{-\Delta_g}) \|_{L^2 (M) \to L^p (\gamma)} \leq C_p \frac{\lambda^{\frac{1}{3}-\frac{1}{3p}}}{(\log \lambda)^{\frac{1}{2}}},\quad 2\leq p<4,
\end{align}
where $C_p \to \infty$ as $p\to 4$.

It follows from this that
\begin{align*}
    \| e_\lambda \|_{L^p (\gamma)} \leq C_p \frac{\lambda^{\frac{1}{3}-\frac{1}{3p}}}{(\log \lambda)^{\frac{1}{2}}},\quad \lambda\geq 1,\quad 2\leq p<4.
\end{align*}
\end{theorem}
We remark that, by scaling the metric, the above theorem deals with manifolds with nonpositive sectional curvatures. Using Theorem \ref{Theorem Log Improvement}, we can show the following estimate at the critical exponent $p=4$.

\begin{corollary}\label{Cor: at p=4}
Let $(M, g)$ be a compact $2$-dimensional smooth Riemannian manifold (without boundary) with nonpositive sectional curvatures pinched between $-1$ and $0$. Also suppose that $\gamma$ is a fixed unit-length curve with $g(D_t \gamma', D_t \gamma')\not=0$. Then, for a uniform constant $C>0$ and $\lambda\gg 1$,
\begin{align*}
    \|\mathcal{R}_\gamma \circ \mathds{1}_{[\lambda, \lambda+(\log \lambda)^{-1}]} (\sqrt{-\Delta_g}) \|_{L^2 (M)\to L^4 (\gamma)}\leq C \frac{\lambda^{\frac{1}{4}}}{(\log \lambda)^{\frac{1}{8}}}.
\end{align*}
It then follows that
\begin{align*}
    \|e_\lambda \|_{L^4 (\gamma)}\leq C\frac{\lambda^{\frac{1}{4}}}{(\log \lambda)^{\frac{1}{8}} },\quad \lambda\gg 1.
\end{align*}
\end{corollary}

This corollary is a curved curve analogue to Blair \cite[Theorem 1.1]{Blair2018logarithmic}, and Xi and Zhang \cite[Theorem 1, Theorem 2]{XiZhang2017improved}.

\subsection*{Outline of the work}
Even though Theorem \ref{Theorem Universal Estimates} is already proved in \cite[Theorem 2]{BurqGerardTzvetkov2007restrictions} and \cite[Theorem 2]{Hu2009lp}, we go through a variation of the proof of Theorem \ref{Theorem Universal Estimates} in \S \ref{S:Proof of universal estimates}, since we need some results from the proof to show Theorem \ref{Theorem Log Improvement}.

In \S \ref{S:Reduction to Universal Estimates}, we introduce some tools to prove Theorem \ref{Theorem Universal Estimates}. We will use pseudo-differential cutoffs $Q_j$ as in \cite{BlairSogge2018concerning} to reduce our problem to Proposition \ref{Prop: I-Qj estimates}, \ref{Prop: QJ estimates}, and \ref{Prop: q plus minus estimates}. The support properties of the $Q_j$ in $\xi$ are similar to a partition of unity in \cite[Section 6]{BurqGerardTzvetkov2007restrictions}.

We will prove Theorem \ref{Theorem Universal Estimates} by showing Proposition \ref{Prop: I-Qj estimates}, \ref{Prop: QJ estimates}, and \ref{Prop: q plus minus estimates} in \S \ref{S:Proof of universal estimates}. Stationary phase arguments, Young's inequality, and Egorov's theorem (cf. \cite{Sogge2014hangzhou}, \cite{Zworski2012Semiclassical}) will be the key points in the section.

By using Proposition \ref{Prop: I-Qj estimates} and \ref{Prop: q plus minus estimates}, we reduce Theorem \ref{Theorem Log Improvement} to a simpler version in \S \ref{S:Prop for large j}. To show the reduced estimates, we lift the remaining problem to the universal cover of the given manifold by the Cartan-Hadamard theorem. We will use the Hadamard parametrix there to compute the remaining part. We will need Proposition \ref{Prop: SP results in Log Improvement} to convert our problem to oscillatory integral operator problems. To finish the proof of Theorem \ref{Theorem Log Improvement}, we may need support properties of the oscillatory integral operators. We will use the Hessian comparison there (cf. \cite[Theorem 11.7]{Lee2018secondEd}) to figure out the support properties.

Using Theorem \ref{Theorem Log Improvement}, and the strategies in Xi and Zhang \cite{XiZhang2017improved}, Sogge \cite{Sogge2017ImprovedCritical}, and Bourgain \cite{Bourgain1991Besicovitch}, we will show Corollary \ref{Cor: at p=4} in the last section.

\subsection*{Notation}
\begin{enumerate}
    \item For nonnegative numbers $A$ and $B$, $A\lesssim B$ means $A\leq CB$ for some uniform constant $C>0$ which depends only on the manifold under consideration.
    \item $A\approx B$ means $cB\leq A \leq CB$ for some uniform constants $c>0$ and $C>0$.
    \item The constant $C>0$ may be assumed to be a uniform constant, if there is no further notice. The uniform constant $C>0$ can also be different from each other at any different lines.
    \item For geometric terminologies, the notation draws largely from Lee \cite{Lee2018secondEd}.
    \item For terminologies of the pseudodifferential operator theory and Egorov's theorem, the notation draws largely from Sogge \cite{Sogge1993fourier}, \cite{Sogge2014hangzhou}, and Zworski \cite{Zworski2012Semiclassical}.
    \item We use $\rho (x, y)$ for the Riemannian distance between $x$ and $y$.
    \item Certain variables may be redefined in different places when the arguments there are independent of each other. For example,
        \begin{itemize}
            \item $\nabla$ may represent the gradient of functions in some places, and may be the Levi-Civita connection in other places.
            \item $\alpha$ may represent a multiindice in some places, and may be deck transformations in other places, defined in the context of the universal cover of the base manifold $M$.
            \item $\partial$ usually represents partial derivative, but $\partial^2 \phi$ represents the Hessian of $\phi$.
            \item $N$ may represent a unit normal vector to a given curve $\gamma$ in some places, but may represent integers $N=1, 2, 3, \cdots$ in other places.
            \item Tildes over letters usually denote the corresponding letters in the universal cover of the base manifold, but sometimes, we also use letters with tildes (or bars) when changing variables if needed.
            \item $\epsilon>0$ appears in many places, and the meanings of $\epsilon>0$ there may be slightly different, but all of them are sufficiently small but fixed at the end of the computations in each section,
        \end{itemize}
and so on. However, the context in which we are using the notations will be clear.
\end{enumerate}

\subsection*{Acknowledgements}
The author would like to thank his Ph.D. advisor Matthew D. Blair for suggesting the problem, for helpful insights, for unlimited patience, and for his guide with numerous details in the course of this work. The author was supported in part by the National Science Foundation grants DMS-1301717 and DMS-1565436.

\section{Some Tools and Reductions for Theorem \ref{Theorem Universal Estimates}}\label{S:Reduction to Universal Estimates}

Let $P=\sqrt{-\Delta_g}$. For some $\epsilon_0>0$ sufficiently small, let $\chi\in \mathcal{S}(\mathbb{R})$ be an even function such that
\begin{align}\label{epsilon0 and chi}
    \chi(0)=1,\quad \chi(t)>0 \text{ for } |t|\leq 1,\quad  \mathrm{supp}(\widehat{\chi})\subset \{t: \epsilon_0/2\leq |t|\leq \epsilon_0\}, \quad \mathrm{supp}(\widehat{\chi^2})\subset [-2\epsilon_0, 2\epsilon_0],
\end{align}
so that
\begin{align*}
    \chi(\lambda-P)e_\lambda =e_\lambda.
\end{align*}
Assume that $\gamma$ has a unit-speed (parametrized by arc-length). With this in mind, to prove \eqref{Universal Estimates for Nonvanishing Curvature}, we now want to show
\begin{align}\label{Universal Estimate Chi Reduction}
    \|\chi(\lambda-P) f \|_{L^2(\gamma)}\lesssim \lambda^{\frac{1}{6}}\|f\|_{L^2 (M)},
\end{align}
that is, we can replace the spectral projector $\mathds{1}_{[\lambda, \lambda+1]} (P)$ by $\chi (\lambda-P)$. Indeed, the operator $\chi (\lambda-P)$ is invertible on the range of the spectral projector $\mathds{1}_{[\lambda, \lambda+1]} (P)$ and
\begin{align*}
    \| \chi(\lambda-P)^{-1} \circ \mathds{1}_{[\lambda, \lambda+1]} (P) \|_{L^2 (M)\to L^2 (M)} \lesssim 1,
\end{align*}
and so, it suffices to show \eqref{Universal Estimate Chi Reduction}.

Fix $\chi_0\in C_0^\infty (\mathbb{R})$ satisfying $\chi_0(t)=1$ for $|t|\leq 1$ and $\chi_0(t)=0$ for $|t|\geq 2$. We also fix $\Tilde{\chi}_0 \in C_0^\infty (\mathbb{R})$ that satisfies $\Tilde{\chi}_0 (t)=1$ for $|t|\leq 3$ and $\Tilde{\chi}_0 (t)=0$ for $|t|\geq 4$. Choose a Littlewood-Paley bump function $\chi_1 \in C_0^\infty (\mathbb{R})$ that satisfies $\chi_1(t)=0$ if $t\not\in (1/2, 2)$ so that we write
\begin{align*}
    \sum_{j=-\infty}^\infty \chi_1 ( 2^j t )=1,\quad \text{for } t\not=0.
\end{align*}

We will use Fermi coordinates frequently in the rest of this article. We recall basic properties of Fermi coordinates briefly here. Let $\gamma$ and $M$ be as above, let $N$ be an element of the normal bundle $N\gamma$, let $\mathcal{E}\subset TM$ be the domain of the exponential map of $M$, let $\mathcal{E}_p=\mathcal{E}\cap N\gamma$, let $E:\mathcal{E}_p \to M$ be the restriction of $\mathrm{exp}$ (the exponential map of $M$) to $\mathcal{E}_p$, and let $U\subset M$ be a normal neighborhood of $\gamma$ with $U=E(V)$ for an appropriate open subset $V\subset N\gamma$. If $(W_0, \psi)$ is a smooth coordinate chart for $\gamma$, we define $B:\psi(W_0)\times \mathbb{R}\to N\gamma|_{W_0}$ by
\begin{align*}
    B(x_1, v_1)=(q, v_1 N|_q),
\end{align*}
by shrinking $W_0$ if necessary. Setting $V_0=V\cap N\gamma|_{W_0}\subset N\gamma$ and $U_0=E(V_0)\subset M$, we define a smooth coordinate map $\varphi: U_0\to \mathbb{R}^2$ by $\varphi=B^{-1}\circ (E|_{V_0})^{-1}$,
\begin{align*}
    \varphi: E(q, v_1 N_q)\mapsto (x_1 (q), v_1).
\end{align*}
Coordinates of this form are called Fermi coordinates. We list here properties of Fermi coordinates from \cite[Proposition 5.26]{Lee2018secondEd}.
\begin{enumerate}
    \item $\gamma\cap U_0$ is the set of points where $v_1=0$.
    \item At each point $q\in \gamma\cap U_0$, the metric components satisfy that
    \[
        g_{ij}=g_{ji}=\begin{cases}
        0, & i=1 \quad \text{and} \quad j=2, \\
        1, & i=j=2.
        \end{cases}
    \]
    \item For every $q\in \gamma\cap U_0$ and $v=v_1 E_1|_q\in N_q \gamma$, the geodesic $\gamma_v$ starting at $q$ with initial velocity $v$ is the curve with coordinate expression $\gamma_v (t)=(x_1 (q), tv_1)$.
\end{enumerate}
For detail, see \cite[Chapter 2]{Gray2004Tubes}, \cite[Chapter 5]{Lee2018secondEd}, etc. If we identify a covector $\xi$ with a vector, then, in Fermi coordinates, we have
\begin{align*}
    |\xi|_{g(x)}=g^{11}(x)\xi_1^2+\xi_2^2,\quad \text{for } x\in \gamma,
\end{align*}
where $(g^{ij})=(g_{ij})^{-1}$. Also, we observe that $g^{11}(x_1, 0)=1$ for $x=(x_1, 0)\in \gamma$ in Fermi coordinates, by the arc-length parametrization.

Suppose $\xi$ is a covector and $N$ is a unit vector field normal to $\gamma$. Here, $\xi(N)$ means $\langle \xi^\#, N \rangle_g$, where $\xi^\#$ is the sharp of $\xi$ as a musical isomorphism. In Fermi coordinates, $N=\frac{\partial}{\partial x_2}$. Set $J=\lfloor \log_2 \lambda^{\frac{1}{3}} \rfloor$. We write
\begin{align*}
    1=\sum_{j=-\infty}^{J-1} \chi_1 \left(2^j \frac{|\xi(N)|}{|\xi|_g} \right)+\Tilde{\chi}_J \left(\lambda^{\frac{1}{3}} \frac{|\xi(N)|}{|\xi|_g} \right),
\end{align*}
where
\begin{align*}
    \Tilde{\chi}_J (t)=1-\sum_{j=-\infty}^{J-1} \chi_1 (t),\quad \Tilde{\chi}_J \in C_0^\infty (\mathbb{R}),\quad \mathrm{supp} (\Tilde{\chi}_J) \subset \{t: |t|\lesssim 1\}.
\end{align*}
Here, if $j\ll 0$, the term $\chi_1 (2^j |\xi(N)|/|\xi_g|)$ is zero, and thus, the sum is a finite sum, since $|\xi(N)|\lesssim |\xi|_g$.

We will consider decomposition using pseudodifferential cutoffs in Smith and Sogge \cite{SmithSogge2007Boundary}, and Blair \cite{Blair2013LowRegularity}. In Fermi coordinates, if $j\leq J-1$, we define the compound symbols
\begin{align}\label{Construction of the compound symbol qj}
    q_j (x, y, \xi)=\chi_0 (\epsilon_0^{-1} \rho (x, \gamma)) \Tilde{\chi}_0 (\epsilon_0^{-1} \rho (y, \gamma)) \chi_1 (2^j |\xi_2|/|\xi|_g) \Upsilon (|\xi|_g/\lambda),
\end{align}
where $d_g=\rho$, and $\Upsilon \in C_0^\infty (\mathbb{R})$ satisfies
\begin{align*}
    \Upsilon (t)=1,\; \text{for } t\in [c_1, c_1^{-1}],\quad \Upsilon (t)=0,\; \text{for } t\not\in \left[\frac{c_1}{2}, 2c_1^{-1}\right],
\end{align*}
with a small fixed number $c_1>0$. Invariantly, we can also define the compound symbols by
\begin{align*}
    q_j (x, y, \xi)=\chi_0 (\epsilon_0^{-1} \rho (x, \gamma)) \Tilde{\chi}_0 (\epsilon_0^{-1} \rho (y, \gamma)) \chi_1 \left(2^j\frac{|\xi(N)|}{|\xi|_g}\right)\Upsilon(|\xi|_g/\lambda).
\end{align*}
If $j=J$, we define, in Fermi coordinates,
\begin{align*}
    q_J (x, y, \xi)=\chi_0 (\epsilon_0^{-1} \rho (x, \gamma)) \Tilde{\chi}_0 (\epsilon_0^{-1} \rho (y, \gamma)) \Tilde{\chi}_J (\lambda^{\frac{1}{3}}|\xi_2|/|\xi|_g) \Upsilon(|\xi|_g/\lambda),\quad 0<\epsilon_0\ll 1,
\end{align*}
or invariantly,
\begin{align*}
    q_J (x, y, \xi)=\chi_0 (\epsilon_0^{-1} \rho (x, \gamma)) \Tilde{\chi}_0 (\epsilon_0^{-1} \rho (y, \gamma)) \Tilde{\chi}_J \left(\lambda^{\frac{1}{3}}\frac{|\xi(N)|}{|\xi|_g}\right) \Upsilon(|\xi|_g/\lambda), \quad 0<\epsilon_0\ll 1.
\end{align*}

Let $Q_j$ be the pseudodifferential operator with compound symbol $q_j$ whose kernel $Q_j (x, w)$ is defined by
\begin{align}\label{Definition of Qj kernel}
    Q_j(x, w)=\frac{1}{(2\pi)^2} \int e^{i(x-w)\cdot \eta} q_j (x, w, \eta) \:d\eta.
\end{align}
As in \cite{BlairSogge2018concerning}, in Fermi coordinates, we know from the homogeneity in $\xi$ and $|\xi|\approx \lambda$ that
\begin{align}\label{Symbol Q properties}
    \begin{split}
        & |D_{x, w}^\beta D_{\xi_1}^{\alpha_1} D_{\xi_2}^{\alpha_2} q_j (x, w, \xi)|\leq C_{\alpha_1, \alpha_2, \beta} 2^{j|\alpha_2|} \lambda^{-|\alpha_1|-|\alpha_2|},\quad \text{for all } \alpha_1, \alpha_2,\\
        & |\partial_{x, w}^\beta Q_j (x, w)|\leq C_N 2^{-j} \lambda^{2+|\beta|} (1+\lambda |x_1-y_1|+\lambda 2^{-j}|x_2-y_2|)^{-N},\quad \text{for } N=1, 2, 3, \cdots, \\
        & \sup_x \int |Q_j (x, w)|\:dw,\quad \sup_w \int |Q_j (x, w)|\:dx \lesssim 1.
    \end{split}
\end{align}
Now, for \eqref{Universal Estimate Chi Reduction}, we are reduced to showing that
\begin{align}\label{Qj estimates}
    \|\sum_{j\leq J} Q_j \circ \chi (\lambda-P) f \|_{L^2(\gamma)} \lesssim \lambda^{\frac{1}{6}} \|f\|_{L^2 (M)},
\end{align}
and
\begin{align}\label{I-Qj estimates}
    \|(I-\sum_{j\leq J} Q_j)\circ \chi(\lambda-P) f\|_{L^2(\gamma)} \leq C_N \lambda^{-N} \|f \|_{L^2(M)},\quad N=1, 2, 3, \cdots.
\end{align}
The estimate \eqref{I-Qj estimates} follows from Young's inequality and the analysis of its kernel.

\begin{proposition}\label{Prop: I-Qj estimates}
The kernel $(I-\sum_{j\leq J} Q_j)\circ \chi(\lambda-P) (x, y)$ satisfies
\begin{align*}
    (I-\sum_{j\leq J} Q_j)\circ \chi(\lambda-P) (x, y)=O(\lambda^{-N}),
\end{align*}
for any $N\geq 1$.
\end{proposition}

We will talk about this later in \S \ref{SS: I-Qj estimates}. To see \eqref{Qj estimates}, we consider $j=J$ separately.

\begin{proposition}\label{Prop: QJ estimates}
We have
\begin{align*}
    \| Q_J \circ \chi(\lambda-P) f\|_{L^2 (\gamma)}\lesssim \lambda^{\frac{1}{6}} \|f\|_{L^2 (M)}.
\end{align*}
\end{proposition}
We will talk about this proposition in the next section. Assuming this proposition is true, we would have \eqref{Qj estimates} if we could show that
\begin{align*}
    \sum_{j\leq J-1} \| Q_j \circ \chi(\lambda-P) f\|_{L^2 (\gamma)} \lesssim \lambda^{\frac{1}{6}} \|f\|_{L^2 (M)},
\end{align*}
which follows from
\begin{align*}
    \| Q_j \circ \chi(\lambda-P) f\|_{L^2 (\gamma)}\lesssim 2^{\frac{j}{2}} \|f\|_{L^2 (M)},\quad j\leq J-1.
\end{align*}
To see this, we split $Q_j$ into two operators $Q_{j, \pm}$ 
\begin{align*}
    Q_j =Q_{j, +}+Q_{j, -},
\end{align*}
where the compound symbols $q_{j, \pm}$ of the $Q_{j, \pm}$ are
\begin{align*}
    q_{j, \pm} (x, y, \xi)=\chi_0 (x_2) \Tilde{\chi}_0 (y_2) \chi_1 (\pm 2^j \xi_2/|\xi|_g) \Upsilon (|\xi|_g/\lambda),
\end{align*}
in Fermi coordinates. We would have \eqref{Qj estimates} if we could show the following.

\begin{proposition}\label{Prop: q plus minus estimates}
If $j\leq J-1$, we have
\begin{align*}
    \| Q_{j, +} \circ \chi(\lambda-P) f\|_{L^2 (\gamma)}\lesssim 2^{\frac{j}{2}} \|f\|_{L^2 (M)}, \quad \text{and} \quad \| Q_{j, -} \circ \chi(\lambda-P) f\|_{L^2 (\gamma)} \lesssim 2^{\frac{j}{2}} \|f\|_{L^2 (M)}.
\end{align*}
\end{proposition}
We will also prove this proposition in the next section.

\subsection{Notation for symbols of pseudodifferential operators}\label{SS: Notation for PDOs}

The pseudodifferential operator $Q_j$ above is defined by using the compound symbols $q_j (x, y, \xi)$, but sometimes we will identify the compound symbol $q_j (x, y, \xi)$ with the usual symbol $q_j (x, \xi)$, modulo smoothing errors, especially when we apply Egorov's theorem and the theorem is stated with usual symbols of pseudodifferential operators. Indeed, recall from the pseudodifferential operator theory (cf. \cite[p.97]{Sogge1993fourier} or \cite[pp.92-pp.93]{Sogge2014hangzhou}) that there exists a symbol $\Tilde{q}_j (x, \xi)$ such that
\begin{align*}
    \int e^{i(x-y)\cdot \xi} q_j (x, y, \xi)\:d\xi-\int e^{i(x-y)\cdot \xi} \Tilde{q}_j (x, \xi)\:d\xi
\end{align*}
is smoothing of any order.

The same principle is applied to any other symbols of pseudodifferential operators unless otherwise specified.

\section{Proof of Theorem \ref{Theorem Universal Estimates}}\label{S:Proof of universal estimates}
Note that
\begin{align*}
    \chi^2 (\lambda-P) f(x)=\frac{1}{2\pi} \int e^{i(\lambda -P)t} \widehat{\chi^2}(t)f(x)\:dt =\frac{1}{2\pi} \int e^{it\lambda} \widehat{\chi^2}(t) e^{-itP} f(x)\:dt.
\end{align*}
We first recall from \cite{Sogge1993fourier} or \cite{Zworski2012Semiclassical} that, by the Lax parametrix, there exist $\varphi$ and $a$ such that, up to smoothing errors,
\begin{align*}
    e^{-itP} (x, w)=\int e^{i\varphi (t, x, \xi)-iw\cdot \xi} a(t, x, \xi)\:d\xi.
\end{align*}
Here, the phase $\varphi=\varphi(t, x, \xi)$ satisfies, for small enough $t$,
\begin{align}\label{varphi construction}
    \begin{split}
        & \kappa_t (d_\xi \varphi(t, x, \xi), \xi)=(x, d_x \varphi(t, x, \xi)), \quad (\text{or, } \kappa_t (y, \xi(0))=(x, \xi(t))) \\
        & \partial_t \varphi+p(x, d_x \varphi)=0,\quad \varphi(0, x, \xi)=\langle x, \xi \rangle,
    \end{split}
\end{align}
where $\kappa_t: \mathbb{R}^{4}\to \mathbb{R}^4$ is the Hamiltonian flow of $p(x, \xi)=|\xi|_{g(x)}$, and homogeneous in $\xi$. Also, the amplitude $a$ satisfies
\begin{align*}
    |\partial_t^j \partial_x^\alpha \partial_\xi^\beta a(t, x, \xi)|\leq C_{j, \alpha, \beta} (1+|\xi|)^{-|\beta|},
\end{align*}
and so,
\begin{align}\label{Size estimates for a}
    |\partial_t^j \partial_x^\alpha \partial_\xi^\beta a(t, x, \lambda \xi)|\leq C_{j, \alpha, \beta} \lambda^{|\beta|} (1+\lambda |\xi|)^{-|\beta|}.
\end{align}
Note that the right hand side is dominated by $C_{j, \alpha, \beta}\lambda^{|\beta|}(1+\lambda |\xi|)^{-|\beta|}\lesssim C_{j, \alpha, \beta}$ if $|\xi|\approx 1$.

In this section, we prove Proposition \ref{Prop: q plus minus estimates}, Proposition \ref{Prop: QJ estimates}, and \eqref{I-Qj estimates} in order.

\subsection{Proof of Proposition \ref{Prop: q plus minus estimates}}\label{SS: Proof of Proposition q+-}

By the $TT^*$ argument, Proposition \ref{Prop: q plus minus estimates} follows from
\begin{align}\label{Qj+ estimates}
    \|Q_{j, \pm} \circ \chi^2 (\lambda-P) \circ Q_{j, +}^* f \|_{L^2(\gamma)} \lesssim 2^j \|f\|_{L^2(\gamma)},\quad j\leq J-1.
\end{align}
We focus on the operator $Q_{j, +} \circ \chi^2 (\lambda-P) \circ Q_{j, +}^*$. The argument for $Q_{j, -} \circ \chi^2 (\lambda-P) \circ Q_{j, -}^*$ is similar. We write
\begin{align*}
    Q_{j, +} \circ \chi^2 (\lambda-P) \circ Q_{j, +}^* f(x)&=\left[ Q_{j, +} \circ \left(\frac{1}{2\pi} \int e^{it\lambda} \widehat{\chi^2}(t) e^{-itP}\:dt \right) \circ Q_{j, +}^* \right]f(x) \\
    &=\frac{1}{2\pi} \int K_{j, +}(x, y)f(y)\:dy,
\end{align*}
where
\begin{align}\label{Kj+ kernel}
    K_{j, +} (x, y)=\int e^{it\lambda} \widehat{\chi^2}(t) \left( Q_{j, +} \circ e^{-itP} \circ Q_{j, +}^* \right)(x, y)\:dt.
\end{align}
By Egorov's theorem (cf. \cite[Theorem 11.1]{Zworski2012Semiclassical}, and/or \cite[Chapter 4]{Sogge2014hangzhou}), we have
\begin{align*}
    Q_{j, +} \circ e^{-itP}=e^{-itP}\circ B_{t, j, +},
\end{align*}
where $B_{t, j, +}$ has a symbol
\begin{align*}
    b_{t, j, +}=\kappa_t^* q_{j, +} +b'.
\end{align*}
Here, $\kappa_t^* q_{j, +}$ is homogeneous of degree $1$ in $\xi$, and $|\partial^\alpha b'|\leq C_\alpha' \lambda^{-1} 2^{2j} 2^{j|\alpha|}\leq C_\alpha \lambda^{-\frac{1}{3}} 2^{j|\alpha|}$ for all $\alpha$. We will ignore the remainder $b'$. Indeed, let $B'$ be the operator whose symbol is $b'=O_S (\lambda^{-1}2^{2j})$ such that
\begin{align*}
    B'(x, y)=\frac{1}{(2\pi)^2}\int e^{i(x-y)\cdot \xi} b'(x, y, \xi)\:d\xi.
\end{align*}
The size estimates $|\partial^\alpha b'|\leq C_\alpha \lambda^{-\frac{1}{3}} 2^{j|\alpha|}$ are better than the size estimates of $\kappa_t^* q_{j, +}$ and $q_{j, +}$. Also, the symbol $b'$ is compactly supported with $\mathrm{supp}(b')\subset \mathrm{supp}(\kappa_{-t}(\mathrm{supp}(q_j)))$. Thus, the arguments below will work when we replace $\kappa_t^* q_{j, +}$ by $b'$ but with better estimates. Hence, for simplicity, we write $b_{t, j, +}=\kappa_t^* q_{j, +}$.

Without loss of generality, we can assume that $a(t, x, \lambda\xi)$ is compactly supported in $\xi$, independent of $\lambda$. Indeed, first let $h_t (z, y, \eta)$ be the symbol of $B_{t, j, +}\circ Q_j^*$. By construction, $h_t (z, y, \eta)$ is supported near $|\eta|\approx \lambda^{-1}$. Let $\beta\in C_0^\infty$ be a bump function with $\beta \equiv 1$ in a neighborhood of $\mathrm{supp}(h_t (z, y, \lambda(\cdot)))$. We then have that
\begin{align*}
    (e^{-itP}\circ B_{t, j, +}\circ Q_j^*)(x, y)=\frac{\lambda^4}{(2\pi)^2}\iiint e^{i\lambda [\varphi(t, x, \xi)-z\cdot \xi+(z-y)\cdot \eta]} a(t, x, \lambda \xi) h_t(z, y, \lambda \eta)\:dz\:d\xi\:d\eta.
\end{align*}
Integrating by parts in $z$, we have, for $N=1, 2, 3, \cdots$,
\begin{align*}
    & \left|\iiint e^{i\lambda [\varphi(t, x, \xi)-z\cdot \xi+(z-y)\cdot \eta]} (1-\beta(\xi)) a(t, x, \lambda \xi) h_t(z, y, \lambda \eta)\:dz\:d\xi\:d\eta \right| \\
    &\lesssim \iiint_{|\xi|\not\approx 1,\; |\eta|\approx 1} (1+\lambda|\xi-\eta|)^{-N} |(1-\beta(\xi)) a(t, x, \lambda \xi) h_t (z, y, \lambda\eta)|\:dz\:d\xi\:d\eta \\
    &\lesssim \iint_{|\xi-\eta|\gtrsim 1,\; |\eta|\approx 1} (1+\lambda |\xi-\eta|)^{-N}\:d\xi\:d\eta \lesssim \int_{|\xi|\gtrsim 1} (\lambda|\xi|)^{-N}\:d\xi \lesssim \lambda^{-N}.
\end{align*}
Since we can ignore this contribution, we can assume that $a(t, x, \lambda\xi)$ is compactly supported in $\xi$.

We write
\begin{align*}
    (e^{-itP}\circ B_{t, j, +}\circ Q_{j, +}^*)(x, y)=\frac{\lambda^6}{(2\pi)^4}\int e^{i\lambda(\varphi(t, x, \xi)-y\cdot \xi)} V_{j, +}(t, x, y, \xi)\:d\xi,
\end{align*}
where
\begin{align}\label{Notation for semiclassical SP}
    \begin{split}
        & V_{j, +}(t, x, y, \xi)=\iiiint e^{i\lambda \Phi(w, \eta, z, \zeta)} v_{j, +}(t, w, \eta, z, \zeta)\:dw\:d\eta\:dz\:d\zeta, \\
        & \Phi(w, \eta, z, \zeta)=(y-w)\cdot \xi+(w-z)\cdot \eta+(z-y)\cdot \zeta, \\
        & v_{j, +}(t, w, \eta, z, \zeta)=v_{j, +}(x, y; t, w, \eta, z, \zeta)=a(t, x, \lambda\xi) b_{t, j, +}(w, z, \lambda\eta) q_{j, +}(y, z, \lambda\zeta).
    \end{split}
\end{align}

We first consider the kernel $e^{-itP} \circ B_{t, j, +}\circ Q_{j, +}^*$.

\begin{lemma}\label{Lemma Semiclassical version of SP q+}
Let $\Phi, v$ be as in \eqref{Notation for semiclassical SP}. We have
\begin{align}\label{Result of SP for q+}
    \begin{split}
        (e^{-itP}\circ B_{t, j, +}\circ Q_{j, +}^*)(x, y)&=\lambda^2 \int e^{i\lambda(\varphi(t, x, \xi)-y\cdot \xi)} \Tilde{a}_j (t, x, y, \xi)\:d\xi \\
        &\hspace{50pt}+\frac{\lambda^6}{(2\pi)^4}R_N(t, y) \int e^{i\lambda (\varphi(t, x, \xi)-y\cdot \xi)} a(t, x, \lambda \xi)\:d\xi,
    \end{split}
\end{align}
where
\begin{align}\label{b lambda set up}
    \begin{split}
        \Tilde{a}_j (t, x, y, \xi)=\sum_{l=0}^{N-1} \lambda^{-l} L_l \Big(v_{j, +}(x, y; t, y, \xi, y, \xi)\Big), \quad \text{and} \quad |\partial_t^\alpha R_N|\leq C_{N, \alpha} \lambda^{-\frac{N}{3}},
    \end{split}
\end{align}
and the $L_l$ are the differential operators of order at most $2l$ with respect to the variables $w, \eta, z$, and $\zeta$, acting on $v_{j, +}$ at the critical point of $\Phi (w, \eta, z, \zeta)$.
\end{lemma}

\begin{proof}
Since
\begin{align*}
    \nabla_{w, \eta, z, \zeta} \Phi=(\Phi_w', \Phi_\eta', \Phi_z', \Phi_\zeta')=(-\xi+\eta, w-z, \zeta-\eta, z-y),
\end{align*}
the critical point is $(w, \eta, z, \zeta)=(y, \xi, y, \xi)$. We consider the stationary phase argument. The Hessian of $\Phi$, denoted by $\partial^2 \Phi$, is
\begin{align*}
    \partial^2 \Phi=
    \begin{pmatrix}
    \Phi_{ww}'' & \Phi_{w\eta}'' & \Phi_{wz}'' & \Phi_{w\zeta}'' \\
    \Phi_{\eta w}'' & \Phi_{\eta\eta}'' & \Phi_{\eta z}'' & \Phi_{\eta\zeta}'' \\
    \Phi_{zw}'' & \Phi_{z\eta}'' &\Phi_{zz}'' & \Phi_{z\zeta}'' \\
    \Phi_{\zeta w}'' & \Phi_{\zeta \eta}'' & \Phi_{\zeta z}'' & \Phi_{\zeta \zeta}''
    \end{pmatrix}=
    \begin{pmatrix}
    O & I & O & O \\
    I & O & -I & O \\
    O & -I & O & I \\
    O & O & I & O
    \end{pmatrix}.
\end{align*}
By standard properties of the determinant, we have $|\det(\partial^2\Phi)|=1$. We begin by computing the signum of $\Phi$. Let $e$ be an eigenvalue of $\partial^2\Phi$, that is, $\det(\partial^2\Phi-eI)=0$. If $e=0$, then $|\det(\partial^2\Phi-eI)|=1\not=0$, which is a contradiction, and so, $e\not=0$. With this in mind, using the properties of block matrices (cf. \cite{Powell2011DetOfBlockMatrices}, \cite{Silvester2000Determinants}, etc.), we have
\begin{align*}
    \det(\partial^2\Phi-eI)&=e^4\left(\left(e-\frac{1}{e} \right)^2-1 \right)^2=\left((e^2-1)^2-e^2 \right)^2=(e^2-e-1)^2 (e^2+e-1)^2 \\
    &=\left(e-\frac{1+\sqrt{5}}{2} \right)^2 \left(e-\frac{1-\sqrt{5}}{2} \right)^2 \left( e-\frac{-1+\sqrt{5}}{2} \right)^2 \left(e-\frac{-1-\sqrt{5}}{2} \right)^2.
\end{align*}
This gives us $\mathrm{sgn}(\partial^2 \Phi)=0$.

By construction and homogeneity, we have the size estimates for radial and generic derivatives
\begin{align*}
    \left|\partial_{w, z}^\alpha (\eta\cdot \nabla_\eta)^{l_1} (\zeta\cdot \nabla_\zeta)^{l_2} \partial_{\eta, \zeta}^\beta \Big(b_{t, j, +}(w, z, \lambda\eta) q_j(y, z, \lambda\zeta)\Big)\right|\leq C_{\alpha, k, l_1, l_2, \beta} 2^{j|\beta|},
\end{align*}
which in turn implies
\begin{align*}
    |\partial_{w, \eta, z, \zeta}^\alpha v_{j, +}(t, w, \eta, z, \zeta)|\leq C_\alpha 2^{j|\alpha|}.
\end{align*}
Here, we used the homogeneity of $b_{t, j, +}=\kappa_t^* q_j$, and the fact that the size estimates of $b_{t, j, +}=\kappa_t^* q_j$ are comparable to those of $q_j$ by \cite[Lemma 11.11]{Zworski2012Semiclassical} with small $t$.

By the method of stationary phase (cf. Theorem 7.7.5 and (3.4.6) in \cite{Hormander2003LinearVol1}), we have that
\begin{align*}
    V_{j, +}(t, x, y, \xi)&=\iiiint e^{i\lambda\Phi(w, \eta, z, \zeta)} v_{j, +}(t, w, \eta, z, \zeta) \:dw\:d\eta\:dz\:d\zeta \\
    &=e^{i\lambda \Phi(y, \xi, y, \xi)} \left(\frac{\lambda}{2\pi} \right)^{-\frac{8}{2}} |\det \partial^2 \Phi|^{-\frac{1}{2}} e^{\frac{\pi i}{4} \mathrm{sgn}(\partial^2 \Phi)} \sum_{l<N} \lambda^{-l} L_l v_{j, +}(t, y, \xi, y, \xi)+ R_N(t, y) a(t, x, \lambda \xi),
\end{align*}
for $N=1, 2, 3, \cdots$, where, at the critical point $(y, \xi, y, \xi)$,
\begin{align*}
    |R_N|\leq C_N \lambda^{-N} \sup_{|\alpha|\leq 2N} |\partial^\alpha v_{j, +}| \leq \Tilde{C}_N \lambda^{-N}(2^j)^{2N} \lesssim \Tilde{C}_N \lambda^{-\frac{N}{3}}.
\end{align*}
Here, the $L_l$ are differential operators of order at most $2l$ acting on $v_{j, +}$ at the critical point $(y, \xi, y, \xi)$, and $2^{-j} \gtrsim \lambda^{-\frac{1}{3}}$. It follows that
\begin{align*}
    V_{j, +}(t, x, y, \xi)=(2\pi)^4 \lambda^{-4} \left(\sum_{k=0}^{N-1} \lambda^{-k} L_k v_{j, +} (t, y, \xi, y, \xi) \right)+R_N (t, y) a(t, x, \lambda \xi),
\end{align*}
which in turn implies that
\begin{align*}
    (e^{-itP}\circ B_{t, j, +}\circ Q_{j, +}^*)(x, y)&=\lambda^2 \int e^{i\lambda(\varphi(t, x, \xi)-y\cdot \xi)} \Tilde{a}_j (t, x, y, \xi)\:d\xi \\
    &\hspace{50pt}+\frac{\lambda^6}{(2\pi)^4}R_N (t, y) \int e^{i\lambda (\varphi(t, x, \xi)-y\cdot \xi)} a(t, x, \lambda \xi)\:d\xi,
\end{align*}
where
\begin{align*}
    \Tilde{a}_j (t, x, y, \xi)=\sum_{l=0}^{N-1} \lambda^{-l} L_l v_{j, +}(x, y; t, y, \xi, y, \xi).
\end{align*}

This completes the proof of Lemma \ref{Lemma Semiclassical version of SP q+}.
\end{proof}

\begin{remark}
Note that the proof of this lemma also works for $j=J$, since we only used the fact that $2^{-j}\gtrsim \lambda^{-\frac{1}{3}}$ for $j$, which is also satisfied for $j=J$. We will use this later to prove Proposition \ref{Prop: QJ estimates}.
\end{remark}

We first want to show that we can ignore the contribution of the second term in the right hand side of \eqref{Result of SP for q+}. If we replace $(Q_j \circ e^{-itP}\circ Q_j^*)$ by the second term modulo smoothing errors, by \eqref{Kj+ kernel}, the contribution of the second term in $K_{j, +}$ is
\begin{align*}
    \frac{\lambda^6}{(2\pi)^4} \iint e^{i\lambda[t+\varphi(t, x, \xi)-y\cdot \xi]} \widehat{\chi^2}(t) R_N (t, y) a(t, x, \lambda \xi)\:d\xi\:dt.
\end{align*}
We can ignore this contribution.
\begin{lemma}\label{Lemma 2nd contribution small}
We have
\begin{align*}
    \left|\frac{\lambda^6}{(2\pi)^4} \iint e^{i\lambda[t+\varphi(t, x, \xi)-y\cdot \xi]} \widehat{\chi^2}(t) R_N (t, y) a(t, x, \lambda \xi)\:d\xi\:dt \right|\leq C_N \lambda^{6-\frac{N}{3}}=O(1),\quad \text{when } N\geq 18.
\end{align*}
\end{lemma}

\begin{proof}
Recall that we may assume that $a(t, x, \xi)$ is compactly supported in $\xi$. The function $\widehat{\chi^2}(t)$ is also compactly supported in $t$. It then follows that
\begin{align*}
    \left|\frac{\lambda^6}{(2\pi)^4} \iint e^{i\lambda[t+\varphi(t, x, \xi)-y\cdot \xi]} \widehat{\chi^2}(t) R_N (t, y) a(t, x, \lambda \xi)\:d\xi\:dt \right|\lesssim \lambda^6 \sup_{t, y} |R_N| \lesssim \lambda^{6-\frac{N}{3}}=O(1),
\end{align*}
when $N\geq 18$.
\end{proof}
By Lemma \ref{Lemma 2nd contribution small}, the contribution of the second term of the right hand side of \eqref{Result of SP for q+} is $O(1)$ by the generalized Young's inequality, which is better than what we need to show.

We thus focus on the first term in the right hand side of \eqref{Result of SP for q+}, that is, modulo $O(1)$ errors,
\begin{align*}
    K_{j, +} (x, y)=\lambda^2 \iint e^{i\lambda(t+\varphi(t, x, \xi)-y\cdot \xi)} \widehat{\chi^2 }(t) \Tilde{a}_j (t, x, y, \xi)\:d\xi\:dt.
\end{align*}
Recall that
\begin{align*}
    \Tilde{a}_j (t, x, y, \xi)=\sum_{l=0}^{N-1} \lambda^{-l} L_l \Big(v_{j, +}(x, y; t, y, \xi, y, \xi)\Big),
\end{align*}
and, by the discussion in \S \ref{SS: Notation for PDOs}, we write
\begin{align*}
    v_{j, +}(t, y, \xi, y, \xi)=a(t, x, \lambda\xi) b_{t, j, +}(y, y, \lambda \xi) q_{j, +} (y, y, \lambda \xi) =a(t, x, \lambda\xi) q_{j, +} (\kappa_t (y, \lambda \xi)) q_{j, +} (y, \lambda \xi).
\end{align*}
We will consider the contribution of the first term of \eqref{Result of SP for q+} in Fermi coordinates. To do so, we consider the behavior of $\xi(t)$ in the Hamiltonian flow of $|\xi|_{g(x)}$. By \eqref{varphi construction}, for Hamiltonian $p(x, \xi)=|\xi|_{g(x)}$, we have Hamilton's equation
\begin{align*}
    \dot{x}=d_\xi p,\quad \dot{\xi}=-d_x p,
\end{align*}
that is,
\begin{align}\label{Hamilton's eqn's}
    \begin{split}
        & \dot{x}_1 (t)=\frac{g^{11}(x) \xi_1}{|\xi|_{g(x)}},\quad \dot{\xi}_1 (t)=-\frac{\partial_{x_1} g^{11}(x)\xi_1^2}{|\xi|_{g(x)}}, \\
        & \dot{x}_2 (t)=\frac{\xi_2}{|\xi|_{g(x)}},\quad \dot{\xi}_2 (t)=-\frac{\partial_{x_2}g^{11}(x)\xi_1^2}{|\xi|_{g(x)}}.
    \end{split}
\end{align}
Here, we used the fact that $|\xi|_g^2=g^{11}(x)\xi_1^2+\xi_2^2$. Recall that we focus only on small $t$ by \eqref{epsilon0 and chi}. We show that $|\partial_{\xi_2} \varphi(t, x, \xi)|=|t|$, if $\xi_1 (s)=0$ for some small $s$.

\begin{lemma}\label{Lemma xi2 with vanishing xi1}
If $\xi_1 (s)=0$ for some small $s$, then we have that $|\partial_{\xi_2} \varphi(s, x, \xi)|=|s|$.
\end{lemma}

\begin{proof}
Suppose $\xi_1 (s_0)=0$ for some small $s_0$. Since we focus on small $t$, we have a unique solution of \eqref{Hamilton's eqn's} by the uniqueness of the solutions to the ODEs. Since $\xi_1 \equiv 0$ satisfies both the equations \eqref{Hamilton's eqn's} and $\xi_1 (s_0)=0$, by the uniqueness of the solution, we should have $\xi_1 \equiv 0$. This implies that
\begin{align*}
    |\xi|_g =\sqrt{g^{11}(x)\xi_1^2+\xi_2^2}=|\xi_2|.
\end{align*}
If $(z(s), \zeta(s))$ is the curve with $z(t)=x$, $\zeta(0)=\xi$, $z_2 (t)=0$, and $\zeta(0)=(0, \xi_2)$, then
\begin{align*}
    \dot{z}_2 (s)=\frac{\xi_2 (s)}{|\xi (s)|_{g(z(s))} }=\pm 1.
\end{align*}
By the mean value theorem, we have $0=z_2 (t)=z_2 (0)\pm t$, and thus, $z_2 (0)=\mp t$. We then have that
\begin{align*}
    |\partial_{\xi_2} \varphi(t, x, \xi)|=|z_2 (0)|=|t|,
\end{align*}
as required.
\end{proof}

We next consider the case of $\xi_1 (s)\not=0$ for any small $s$.

\begin{lemma}\label{Lemma xi2 dot positive}
For $|s|\ll 1$, suppose $\xi(s)\in \mathrm{supp}(\Tilde{a}_j (s, x, y, \cdot))$. Let $\gamma$ be as above. If $\xi_1 (s)\not=0$ for any small $s$, then, for $x, y\in \gamma$, in Fermi coordinates, we have either $\dot{\xi}_2 (s)>0$ or $\dot{\xi}_2 (s)<0$.
\end{lemma}

\begin{proof}
We know from the curvature assumption of $\gamma$ that
\begin{align*}
    |\nabla_{\partial_1} \partial_1 |_g \not=0,\quad \text{where }\partial_1=\frac{\partial}{\partial x_1}\text{ and } \partial_2=\frac{\partial}{\partial x_2},
\end{align*}
where $\nabla$ denotes the Levi-Civita connection. Note that
\begin{align*}
    0=\frac{\partial}{\partial x_1} \langle \partial_1, \partial_2 \rangle_g=\langle \nabla_{\partial_1} \partial_1, \partial_2 \rangle_g+\langle \partial_1, \nabla_{\partial_1} \partial_2 \rangle_g.
\end{align*}
But since $[\partial_1, \partial_2]=0$ and the Levi-Civita connection is symmetric, we have that
\begin{align*}
    \langle \partial_1, \nabla_{\partial_1} \partial_2 \rangle_g =\langle \partial_1, \nabla_{\partial_2} \partial_1 \rangle_g +\langle \partial_1, [\partial_1, \partial_2] \rangle_g =\frac{1}{2} \frac{\partial}{\partial x_2} |\partial_1|_g^2=\frac{1}{2}\frac{\partial}{\partial x_2} g_{11} (x_1, 0).
\end{align*}
Combining these two, we have that
\begin{align}\label{g11 x2 derivative}
    \frac{\partial}{\partial x_2} g_{11} (x_1, 0)=-2\langle \nabla_{\partial_1} \partial_1, \partial_2 \rangle_g.
\end{align}
Since $|\partial_1|_g=1$ along $x_2=0$ by the arc-length parametrization, we have that
\begin{align*}
    \langle \nabla_{\partial_1} \partial_1, \partial_1 \rangle_g=\frac{1}{2}\frac{\partial}{\partial x_1} |\partial_1|_g^2=0,
\end{align*}
and thus,
\begin{align*}
    \nabla_{\partial_1} \partial_1=\langle \nabla_{\partial_1} \partial_1, \partial_1 \rangle_g \partial_1+\langle \nabla_{\partial_1} \partial_1, \partial_2 \rangle_g \partial_2=c\partial_2,
\end{align*}
for some $c\not=0$ due to the assumption $|\nabla_{\partial_1} \partial_1|_g\not=0$. By \eqref{g11 x2 derivative} with this, we have that
\begin{align*}
    -\frac{\partial}{\partial x_2} g^{11}(x)\not=0, \quad \text{on } x_2=0,
\end{align*}
since $g^{11}=g_{11}^{-1}$ (cf. \cite[Proposition 5.26]{Lee2018secondEd}), and this also holds on a neighborhood of $x_2=0$. Since we are assuming $\xi_1\not=0$, by the above Hamilton's equation, we have that
\begin{align*}
    \dot{\xi}_2 (s)=-\partial_2 g^{11}(x)\xi_1^2\not=0,\quad \text{along } x_2=0.
\end{align*}

This completes the proof.
\end{proof}

We next consider the $\xi_2$ derivative of $\varphi$.

\begin{lemma}\label{Lemma xi2 deriv of varphi}
Suppose $\xi \in \mathrm{supp}(q_{j, +}(x, y, \lambda(\cdot)))$, $\xi_1 \not=0$, and $d_x \varphi (t, x, \xi)\in \mathrm{supp}(q_{j, +}(x, y, \lambda(\cdot)))$ for some $x\in \gamma$, i.e., $x_2=0$ in Fermi coordinates. Then there exists a uniform constant $C>0$ such that $|\partial_{\xi_2} \varphi(t, x, \xi)|\geq C2^{-j} |t|$.
\end{lemma}

\begin{proof}
We use Fermi coordinates to prove this lemma. Suppose $(z(s), \zeta(s))$ is the curve such that $z(t)=x$ and $\zeta(0)=\xi$. Without loss of generality, by homogeneity we assume $|\xi|_g=1$ since $|\xi|_g\approx 1$ for $\xi\in \mathrm{supp}(q_{j, +}(x, y, \lambda(\cdot)))$. It follows from \eqref{varphi construction} that $d_x \varphi(t, x, \xi)=\zeta(t)$ and $d_\xi \varphi(t, x, \xi)=z(0)$, and thus,
\begin{align*}
    \chi_1 (2^j \zeta_2 (0))\not= 0 \quad \text{and} \quad \chi_1 (2^j \zeta_2 (t))\not=0, \quad \text{i.e.,} \quad \zeta_2 (0)\approx2^{-j} \quad \text{and} \quad \zeta_2 (t)\approx 2^{-j}.
\end{align*}
Thus, if $0\leq s\leq t$, we have $\zeta_2 (s)\approx 2^{-j}$ since the map $s\mapsto \zeta_2 (s)$ is monotonic in $s$, due to the fact that either $\dot{\zeta}_2 (s)>0$ or $\dot{\zeta}_2 (s)<0$ by Lemma \ref{Lemma xi2 dot positive}. Similarly, the map $s\mapsto \zeta_2 (s)$ is monotonic in $s$ when $t\leq s\leq 0$. In any cases, we have $|\zeta_2 (s)|\geq C 2^{-j}$ when $s$ is between $0$ and $t$.

Now recall that $\dot{z}_2 (s)=\zeta_2 (s)$ since we are in Fermi coordinates. Thus, since $z(t)=x\in \gamma$, we have $z_2 (t)=0$, and so, the mean value theorem gives
\begin{align*}
    0=z_2 (0)+t\dot{z}_2 (\Tilde{c}),
\end{align*}
for some $\Tilde{c}$ between $0$ and $t$. This gives
\begin{align*}
    |\partial_{\xi_2} \varphi (t, x, \xi)|=|z_2 (0)|=|t\dot{z}_2 (\Tilde{c})|=|t\zeta_2 (\Tilde{c})| \geq C|t|2^{-j},
\end{align*}
for some uniform constant $C>0$.
\end{proof}

We now return to the kernels $K_{j, +} (x, y)$. By Lemma \ref{Lemma Semiclassical version of SP q+}, we write
\begin{align*}
    K_{j, +} (x, y)=\lambda^2 \iint e^{i\lambda(t+\varphi(t, x, \xi)-y\cdot \xi)} \Tilde{a}_j (t, x, y, \xi) \widehat{\chi^2}(t)\:d\xi\:dt,
\end{align*}
where $\Tilde{a}_j (t, x, y, \xi)=0$ unless
\begin{align*}
    \chi_1 \left(2^j \frac{\partial_{x_2} \varphi(t, x, \xi)}{|d_x \varphi(t, x, \xi)|_g} \right)\not=0, \quad \chi_1 \left(2^j \frac{\xi_2}{|\xi|_g} \right)\not=0,\quad \text{and} \quad |\xi|_g \in [\frac{c_1}{2}, 2c_1^{-1}],
\end{align*}
for some small constant $c_1 >0$. Moreover, we have
\begin{align}\label{Size estimates for a tilde}
    |\partial_t^k \partial_{\xi_1}^l \partial_{\xi_2}^m \Tilde{a}_j |\leq C_{k, l, m} 2^{jm}.
\end{align}
Here, we used \eqref{Size estimates for a} and size estimates of $q_j$ and  $\kappa_t^* q_j$, since $|\xi|_g \approx 1$ by the support properties of $\Upsilon$. Also, note that $y_2=0$ in Fermi coordinates if $y=(y_1, y_2)\in \gamma$. If we set
\begin{align*}
    L_\xi=\frac{1-i\lambda(\partial_{\xi_1} \varphi(t, x, \xi)-y_1)\partial_{\xi_1}-i\lambda 2^{-2j} \partial_{\xi_2} \varphi(t, x, \xi) \partial_{\xi_2} }{1+\lambda^2 |\partial_{\xi_1} \varphi(t, x, \xi)-y_1|^2+\lambda^2 2^{-2j} |\partial_{\xi_2} \varphi(t, x, \xi)|^2 }
\end{align*}
then we have
\begin{align*}
    L_\xi (e^{i\lambda(t+\varphi(t, x, \xi)-y_1 \xi_1)})=e^{i\lambda(t+\varphi(t, x, \xi)-y_1 \xi_1)}.
\end{align*}
Integration by parts gives us that
\begin{align*}
    \left|\int e^{i\lambda(t+\varphi(t, x, \xi)-y_1 \xi_1)} \Tilde{a}_j (t, x, y, \xi)\:d\xi \right|=\left|\int e^{i\lambda(t+\varphi(t, x, \xi)-y_1 \xi_1)} (L_\xi^T)^N (\Tilde{a}_j (t, x, y, \xi))\:d\xi\right|,
\end{align*}
where $L_\xi^T$ is the transpose of $L_\xi$. For simplicity, we set
\begin{align*}
    & w_1=\lambda(\partial_{\xi_1} \varphi-y_1),\quad w_2=\lambda 2^{-j} \partial_{\xi_2} \varphi, \\
    & L=L_{\xi}=\frac{1-iw_1 \partial_{\xi_1}-iw_2 2^{-j} \partial_{\xi_2} }{1+w_1^2+w_2^2}.
\end{align*}
Here, $w_1$ and $w_2$ are functions of $\lambda, t, x, y, \xi$ and we suppress the arguments for convenience if necessary. We then write (up to signs)
\begin{align*}
    L^T \Tilde{a}_j=A_0+A_1+A_2+A_3+A_4,
\end{align*}
where
\begin{align*}
    & A_0=\frac{1}{1+w_1^2+w_2^2} \Tilde{a}_j,\quad  A_1=\frac{i w_1}{1+w_1^2+w_2^2} \partial_{\xi_1} \Tilde{a}_j, \quad A_2=\frac{i w_2}{1+w_1^2+w_2^2} 2^{-j} \partial_{\xi_2} \Tilde{a}_j, \\
    & A_3=\Tilde{a}_j \partial_{\xi_1} \left( \frac{i w_1}{1+w_1^2+w_2^2} \right),\quad A_4=\Tilde{a}_j 2^{-j} \partial_{\xi_2} \left(\frac{i w_2}{1+w_1^2+w_2^2} \right).
\end{align*}
By \eqref{Size estimates for a tilde}, we have
\begin{align}\label{A1A2 approx}
    |A_0|,\; |A_1|,\; |A_2|\leq \frac{1}{(1+w_1^2+w_2^2)^{\frac{1}{2}}}.
\end{align}
We note that $A_3$ is
\begin{align*}
    \Tilde{a}_j\partial_{\xi_1} \left(\frac{i w_1}{1+w_1^2+w_2^2} \right)=\Tilde{a}_j\partial_{w_1} \left(\frac{i w_1}{1+w_1^2+w_2^2} \right)\partial_{\xi_1} w_1+\Tilde{a}_j\partial_{w_2} \left(\frac{iw_1}{1+w_1^2+w_2^2} \right)\partial_{\xi_1} w_2,
\end{align*}
where $\partial_{\xi_1} w_1=\lambda \partial_{\xi_1}^2 \varphi$ and $\partial_{\xi_1} w_2=\lambda 2^{-j} \partial_{\xi_1 \xi_2}^2 \varphi$. Similarly, $A_4$ is
\begin{align*}
    \Tilde{a}_j 2^{-j} \partial_{w_1}\left(\frac{iw_2}{1+w_1^2+w_2^2} \right) \partial_{\xi_2} w_1+\Tilde{a}_j 2^{-j}\partial_{w_2} \left(\frac{iw_2}{1+w_1^2+w_2^2} \right) \partial_{\xi_2} w_2,
\end{align*}
where $\partial_{\xi_2} w_1=\lambda \partial_{\xi_1 \xi_2}^2 \varphi$ and $\partial_{\xi_2} w_2=\lambda \partial_{\xi_2}^2 \varphi$. Both $A_3$ and $A_4$ contain terms of the form $\partial^\alpha \varphi$ for $|\alpha|\geq 2$, and we want to approximate these first. Recall that we are assuming $|t|\lesssim 1$, by the support properties of $\chi$.

\begin{lemma}\label{Lemma varphi derivative}
If $|\alpha|\geq 2$ for $\alpha=(\alpha_1, \alpha_2)$, then
\[
|\partial_\xi^\alpha \varphi|\lesssim \begin{cases}
|\xi_2|^2 |t|, & \text{if } \alpha_2=0, \\
|\xi_2| |t|, & \text{if } \alpha_2=1, \\
|t|, & \text{for any } |\alpha| \geq 2.
\end{cases}
\]
\end{lemma}

\begin{proof}
It follows from \eqref{varphi construction} that
\begin{align*}
    \varphi(t, x, \xi)=x\cdot \xi-\int_0^t p(x, \nabla_x \varphi(s, x, \xi))\:ds,
\end{align*}
where $p(x, \xi)=|\xi|_{g(x)}$. For any $|\alpha|\geq 2$, we obtain
\begin{align}\label{varphi deriv for any alpha}
    |\partial_{\xi}^\alpha \varphi|=\left|-\int_0^t \partial_\xi^\alpha (p(x, \nabla_x \varphi(s, x, \xi)))\:ds \right| \leq |t| \sup_\xi \big[\partial^\alpha (p(x, \nabla_x \varphi(s, x, \xi)))\big] \leq C_\alpha |t|.
\end{align}
We now focus on $\alpha_2=0$ or $\alpha_2=1$.

On the other hand, if $\Phi (\xi)$ is homogeneous of degree $-k$, then, by Euler's homogeneous theorem, we have
\begin{align}\label{Homogeneous result}
    \xi_1 \partial_{\xi_1} \Phi +\xi_2 \partial_{\xi_2} \Phi=-k \Phi.
\end{align}
Since $|\xi| \approx 1$, we have either $|\xi_1|\approx 1$ or $|\xi_2| \approx 1$. The case of $|\xi_1|\leq |\xi_2|$ is simpler. Indeed, if $|\xi_1|\leq |\xi_2|$, then $|\xi_2|\approx 1$, and so, by \eqref{varphi deriv for any alpha}, we have $|\partial_\xi^\alpha \varphi|\leq C_\alpha  |\xi_2|^l |t|$ for any nonnegative integer $l$.

Thus, we may assume that $|\xi_2|\leq |\xi_1|$, and so, $|\xi_1| \approx 1$. Taking $\Phi=\partial_{\xi_2} \varphi$ with $k=0$ in \eqref{Homogeneous result}, it follows from \eqref{varphi deriv for any alpha} that
\begin{align}\label{xi1 xi2 deriv of varphi}
    |\partial_{\xi_1 \xi_2}^2 \varphi|=\left|\frac{\xi_2}{\xi_1} \partial_{\xi_2}^2 \varphi \right|\lesssim |\xi_2| |t|.
\end{align}
Using this, if we take $\Phi=\partial_{\xi_1} \varphi$ with $k=0$ in \eqref{Homogeneous result}, then we have that
\begin{align*}
    |\partial_{\xi_1}^2 \varphi|=\left|\frac{\xi_2}{\xi_1} \partial_{\xi_2 \xi_1}^2 \varphi \right|\lesssim |\xi_2| (|\xi_2| |t|)=|\xi_2|^2 |t|.
\end{align*}
We can also compute $\partial_{\xi_1 \xi_1 \xi_2}^3 \varphi$ taking $\Phi=\partial_{\xi_1 \xi_2}^2 \varphi$ with $k=-1$
\begin{align*}
    \partial_{\xi_1 \xi_1 \xi_2}^3 \varphi=-\frac{1}{\xi_1} \partial_{\xi_1 \xi_2}^2 \varphi-\frac{\xi_2}{\xi_1} \partial_{\xi_1 \xi_2 \xi_2}^3 \varphi.
\end{align*}
By \eqref{varphi deriv for any alpha} and \eqref{xi1 xi2 deriv of varphi}, we have $|\partial_{\xi_1 \xi_1 \xi_2}^3 \varphi|\lesssim |\xi_2| |t|$. Similarly, we can find the estimate for $\partial_{\xi_1 \xi_1 \xi_1}^3 \varphi$. The higher order derivatives of $\varphi$ are bounded by induction and repeated use of \eqref{xi1 xi2 deriv of varphi}.
\end{proof}
By Lemma \ref{Lemma xi2 deriv of varphi}, we have $|\partial_{\xi_2}\varphi|\gtrsim |\xi_2| |t|$. By this and Lemma \ref{Lemma varphi derivative}, we have that
\begin{align*}
    & |\partial_{\xi_1} w_1|=|\lambda \partial_{\xi_1}^2 \varphi|\lesssim \lambda|\xi_2|^2 |t| \lesssim \lambda |\partial_{\xi_2} \varphi| |\xi_2| \lesssim \lambda |\partial_{\xi_2} \varphi| 2^{-j} \lesssim |w_2|, \\
    & |2^{-j} \partial_{\xi_2} w_1|=|\lambda 2^{-j} \partial_{\xi_1 \xi_2}^2 \varphi|\lesssim \lambda 2^{-j} |\xi_2| |t| \lesssim \lambda 2^{-j} |\partial_{\xi_2} \varphi| \lesssim |w_2|, \\
    & |\partial_{\xi_1} w_2|=|\lambda 2^{-j} \partial_{\xi_1 \xi_2}^2 \varphi|\lesssim \lambda 2^{-j} |\xi_2| |t| \lesssim \lambda 2^{-j} |\partial_{\xi_2} \varphi|= |w_2|, \\
    & |2^{-j} \partial_{\xi_2} w_2|=|\lambda (2^{-j})^2 \partial_{\xi_2}^2 \varphi|\lesssim \lambda 2^{-j} |\xi_2| |t| \lesssim \lambda 2^{-j}|\partial_{\xi_2} \varphi|= |w_2|.
\end{align*}
We also have that
\begin{align*}
    \partial_{w_l} \left(\frac{w_k}{1+w_1^2+w_2^2} \right) \lesssim \frac{1}{1+w_1^2+w_2^2},\quad l, k\in \{1, 2\}.
\end{align*}
Combining these together, we have that
\begin{align*}
    |A_3|,\; |A_4| \lesssim \frac{|w_2|}{1+w_1^2+w_2^2}\leq \frac{(1+w_1^2+w_2^2)^{\frac{1}{2}}}{1+w_1^2+w_2^2}=\frac{1}{(1+w_1^2+w_2^2)^{\frac{1}{2}}}.
\end{align*}
By this and \eqref{A1A2 approx}, we have
\begin{align}\label{L transpose estimate}
    |L^T \Tilde{a}_j|\lesssim \frac{1}{(1+w_1^2+w_2^2)^{\frac{1}{2}}}.
\end{align}

Inductively, we can obtain
\begin{align*}
    |(L^T)^N \Tilde{a}_j|\lesssim (1+w_1^2+w_2^2)^{-\frac{N}{2}} \lesssim (1+|w_1|+|w_2|)^{-N}.
\end{align*}
Hence, integration by parts gives, for $x, y\in \gamma$,
\begin{align*}
    \left|\int e^{i\lambda(t+\varphi(t, x, \xi)-y_1 \xi_1)} \Tilde{a}_j (t, x, y, \xi)\:d\xi \right| &=\left|\int (L_\xi)^{N} (e^{i\lambda (t+\varphi(t, x, \xi)-y_1 \xi_1)}) \Tilde{a}_j (t, x, y, \xi)\:d\xi\right| \\
    &=\left|\int e^{i\lambda(t+\varphi(t, x, \xi)-y_1 \xi_1)} (L_\xi^T)^{N} (\Tilde{a}_j (t, x, y, \xi))\:d\xi \right| \\
    &\lesssim \int (1+\lambda|\partial_{\xi_1} \varphi(t, x, \xi)-y_1|+\lambda 2^{-j} |\partial_{\xi_2} \varphi(t, x, \xi)|)^{-N}\:d\xi,
\end{align*}
and thus, we have that
\begin{align*}
    |K_{j, +} (x_1, 0, y_1, 0)|\leq C_N \lambda^2 \iint_{\mathrm{supp}(q_j) }|\widehat{\chi^2}(t)| \left( 1+\lambda|\partial_{\xi_1 }\varphi(t, x, \xi)-y_1|+\lambda 2^{-j}|\partial_{\xi_2} \varphi(t, x, \xi)| \right)^{-N}\:d\xi\:dt.
\end{align*}
In Fermi coordinates, we can write $\gamma=\{(x_1, 0): |x_1|\leq \epsilon\}$ for some small $\epsilon>0$, and so, we may write $x=(x_1, 0)$ and $y=(y_1, 0)$. To show \eqref{Qj+ estimates}, we now want to show that
\begin{align*}
    \int |K_{j, +} (x_1, 0, y_1, 0)|\:dx_1 \lesssim 2^j,\quad \text{and} \quad \int |K_j (x_1, 0, y_1, 0)|\:dy_1 \lesssim 2^j.
\end{align*}
To see these, first note that, by Lemma \ref{Lemma xi2 with vanishing xi1} and Lemma \ref{Lemma xi2 deriv of varphi}, we have $|\partial_{\xi_2} \varphi(t, x, \xi)|\gtrsim 2^{-j}|t|$ in both cases $\xi_1\not=0$ and $\xi_1=0$, and so, we have that
\begin{align*}
    |K_{j, +} (x_1, 0, y_1, 0)|\leq C_N \lambda^2 \iint_{\xi_2 \approx 2^{-j},|\xi|\approx 1} |\widehat{\chi^2}(t)| (1+\lambda |\partial_{\xi_1} \varphi (t, (x_1, 0), \xi)-y_1|+\lambda 2^{-2j}|t|)^{-N}\:dt\:d\xi,
\end{align*}
and thus, the second inequality follows from
\begin{align*}
    & \int |K_{j, +} (x_1, 0, y_1, 0)|\:dy_1 \\
    & \leq C_N \lambda^2 \int_{\xi_2 \approx 2^{-j}, |\xi| \approx 1 } \bigg(\iint |\widehat{\chi^2}(t)| (1+\lambda |\partial_{\xi_1} \varphi(t, (x_1, 0), \xi)-y_1|+\lambda 2^{-2j}|t| )^{-N}\:dt\:dy_1 \bigg) d\xi \\
    & \lesssim \lambda^2 (\lambda 2^{-2j})^{-1} \lambda^{-1} \mathrm{Vol}(\{\xi_2 \approx 2^{-j},\; |\xi| \approx 1\}) \\
    & \lesssim 2^{2j} 2^{-j}=2^j.
\end{align*}
Here, we gained $\lambda^{-1}$ from $y_1$ integration, $\lambda 2^{-2j}$ from $t$ integration, and $\mathrm{Vol}(\{|\xi_2|\approx 2^{-j}, |\xi|\approx 1\})$ from $\xi_2$ integration.

The proof that
\begin{align*}
    \int |K_{j, +} (x_1, 0, y_1, 0)|\:dx_1 \lesssim 2^j
\end{align*}
is similar, but it uses that $|\partial_{x_1 \xi_1}^2 \varphi (t, x, \xi)|\geq c>0$ for some small $c>0$, for $|\xi|_g \approx 1$ and $\xi_2 \approx 2^{-j}$, i.e., $|\xi_1|\approx 1$.

To see $|\partial_{x_1 \xi_1}^2 \varphi(t, x, \xi)|\gtrsim 1$, we recall that $\varphi$ satisfies $\varphi(0, x, \xi)=\langle x, \xi \rangle$ (cf. \cite[Lemma 10.5 (ii)]{Zworski2012Semiclassical}). By this, we have $|\partial_{x_1 \xi_1}^2 \varphi(t, x, \xi)|=1$ at $t=0$, and so, $|\partial_{x_1 \xi_1}^2 \varphi (t, x, \xi)|\gtrsim 1$ for small $t$ by continuity, but we can focus only on small $t$ by taking $\epsilon_0>0$ to be sufficiently small in \eqref{epsilon0 and chi}, and hence $|\partial_{x_1 \xi_1}^2 \varphi (t, x, \xi)|\gtrsim 1$ in the support of $K_{j, +}$.

This completes the proof of Proposition \ref{Prop: q plus minus estimates}.

\subsection{Proof of Proposition \ref{Prop: QJ estimates}}
In this subsection, by the $TT^*$ argument, we want to show that
\begin{align}\label{QJ TT* claim}
    \|Q_J \circ \chi^2 (\lambda-P) \circ Q_J^* f\|_{L^2 (\gamma)} \lesssim \lambda^{\frac{1}{3}} \|f\|_{L^2 (\gamma)},\quad J=\lfloor \log_2 \lambda^{\frac{1}{3}} \rfloor.
\end{align}
We obtain $K_J$, $\Tilde{a}_J$, $v_J$, etc., by replacing $j$ by $J$ in the settings of the previous section. We also ignore the contribution of the remainder after using Egorov's theorem.

Using the proof of Lemma \ref{Lemma Semiclassical version of SP q+}, we have the following lemma.

\begin{lemma}\label{Lemma Semiclassical SP QJ}
We have
\begin{align}\label{Semiclassical SP QJ result}
    \begin{split}
        (e^{-itP} \circ B_{t, J} \circ Q_J^*)(x, y)&=\lambda^2 \int e^{i\lambda (\varphi(t, x, \xi)-y\cdot \xi)} \Tilde{a}_J (t, x, y, \xi)\:d\xi \\
        &\hspace{50pt}+\frac{\lambda^6}{(2\pi)^4} \int e^{i\lambda (\varphi(t, x, \xi)-y\cdot \xi)} R_N (t, y) a(t, x, \lambda \xi)\:d\xi,
    \end{split}
\end{align}
where
\begin{align*}
    \Tilde{a}_J (t, x, y, \xi)=\sum_{l=0}^{N-1} \lambda^{-l} L_l v_J (x, y; t, y, \xi, y, \xi), \quad |\partial_t^\alpha R_N|\leq C_{N, \alpha} \lambda^{-\frac{N}{3}},
\end{align*}
and the $L_l$ are the differential operators with respect to $(w, \eta, z, \zeta)$ of order at most $2l$ acting on $v_J$ at the point $(w, \eta, z, \zeta)=(y, \xi, y, \xi)$.
\end{lemma}

By Lemma \ref{Lemma 2nd contribution small} and the generalized Young's inequality again, the contribution of the second term of the right hand side of \eqref{Semiclassical SP QJ result} is $O(1)$ with $N$ large, and so, we focus on the first term in \eqref{Semiclassical SP QJ result}.

Using the proof of Lemma \ref{Lemma xi2 dot positive}, we can also show that $\dot{\xi}_2$ is nonvanishing.

\begin{lemma}\label{Lemma xi2 dot positive qJ}
For $|s|\ll 1$, suppose $\xi(s)\in \mathrm{supp}(\Tilde{a}_J (s, x, y, \cdot))$. Let $\gamma$ be as above. If $\xi_1 (s)\not=0$ for any small $s$, then, for $x, y\in \gamma$, in Fermi coordinates, we have either $\dot{\xi}_2 (s)>0$ or $\dot{\xi}_2 (s)<0$.
\end{lemma}

With this in mind, we figure out the support properties of $\Tilde{a}_J$.

\begin{lemma}\label{Lemma xi2 derivative of varphi qJ}
Suppose $\xi\in \mathrm{supp}(q_J (x, y, \lambda(\cdot)))$, and $d_x \varphi(t, x, \xi)\in \mathrm{supp}(q_J (x, y, \lambda (\cdot)))$ for some $x\in \gamma$, i.e., $x_2=0$ in Fermi coordinates. If $|t|\gg \lambda^{-\frac{1}{3}}$, then $\Tilde{a}_J (t, x, y, \xi)=0$, and thus, $\Tilde{a}_J$ is supported where $|t|\lesssim \lambda^{-\frac{1}{3}}$.
\end{lemma}

\begin{proof}
Suppose $(z(s), \xi(s))$ is the curve such that $z(t)=x, \xi(0)=\xi$. It follows that
\begin{align*}
    d_x \varphi (t, x, \xi)=\xi(t)=(\xi_1 (t), \xi_2 (t)),\quad d_\xi \varphi(t, x, \xi)=z(0)=(z_1 (0), z_2 (0)).
\end{align*}
By construction, we have $\Tilde{a}_J(t, x, y, \xi)=0$ in Fermi coordinates unless
\begin{align*}
    \Tilde{\chi}_J\left( \lambda^{\frac{1}{3}} \frac{|\xi_2 (t)|}{|\xi (t)|_g} \right)\not=0,\quad \text{and} \quad \Tilde{\chi}_J \left(\lambda^{\frac{1}{3}} \frac{|\xi_2|}{|\xi|_g} \right)\not=0.
\end{align*}
By the support properties of $\Upsilon$, we have $|\xi|_g \approx 1$ and $|\xi(t)|_g\approx 1$, and so, we have $\Tilde{a}_J (t, x, y, \xi)=0$ unless
\begin{align*}
    |\xi_2 (t)|\lesssim \lambda^{-\frac{1}{3}},\quad |\xi_2 (0)|\lesssim \lambda^{-\frac{1}{3}}.
\end{align*}
We want to show that we cannot have $|\xi_2 (t)|\lesssim \lambda^{-\frac{1}{3}}$ when $|t|\gg \lambda^{-\frac{1}{3}}$. We note that $\xi_1 (s)\not=0$ for any small $s$. Indeed, if $|\xi_2|\lesssim \lambda^{-\frac{1}{3}}$ and $|\xi|\approx 1$, then $|\xi_1|\gtrsim 1$.

By the mean value theorem, we have
\begin{align}\label{MVT result of xi2}
    \xi_2 (t)=\xi_2 (0)+\dot{\xi}_2 (c_t) t,
\end{align}
where $c_t$ is between $0$ and $t$. Since $\widehat{\chi^2}$ is compactly supported in $[-2\epsilon_0, 2\epsilon_0]$ for small $\epsilon_0>0$ by \eqref{epsilon0 and chi}, by the proof of Lemma \ref{Lemma xi2 dot positive qJ}, there exists a $\Tilde{c}>0$ such that $|\dot{\xi}_2 (s)|\geq \Tilde{c}$. If $|\xi_2 (0)|\gg \lambda^{-\frac{1}{3}}$, then we have $\Tilde{a}_J$ vanishes automatically. If $|\xi_2 (0)|\lesssim \lambda^{-\frac{1}{3}}$ and $|t|\gg \lambda^{-\frac{1}{3}}$, then, by \eqref{MVT result of xi2} and $|\dot{\xi}_2 (s)|\geq \Tilde{c}$, we have
\begin{align*}
    |\xi_2 (t)|\geq |\dot{\xi}_2 (c_t)||t|-|\xi_2 (0)|\gg \lambda^{-\frac{1}{3}}.
\end{align*}

Hence, the amplitude $\Tilde{a}_J$ is supported where $|t|\lesssim \lambda^{-\frac{1}{3}}$.
\end{proof}

In Fermi coordinates, by Lemma \ref{Lemma Semiclassical SP QJ}, modulo $O(1)$ errors, we write
\begin{align*}
    K_J (x, y)&=\lambda^2 \iint e^{i\lambda[t+\varphi(t, x, \xi)-y\cdot \xi]} \widehat{\chi^2}(t) \Tilde{a}_J(t, x, y, \xi)\:d\xi\:dt,
\end{align*}
where, by Lemma \ref{Lemma xi2 derivative of varphi qJ}, $\Tilde{a}_J(t, x, y, \xi)$ is supported where $|t|\lesssim \lambda^{-\frac{1}{3}}$. Moreover, we have
\begin{align}\label{Size estimates of b lambda J}
    |\partial_t^k \partial_{\xi_1}^l \partial_{\xi_2}^m \Tilde{a}_J|\leq C_{k, l, m} (\lambda^{\frac{1}{3}})^m.
\end{align}
As before, here we used \eqref{Size estimates for a} and size estimates of $q_j$ and  $\kappa_t^* q_j$, since $|\xi|_g \approx 1$ by the support properties of $\Upsilon$. Note that $y_2=0$ in Fermi coordinates for $y=(y_1, y_2)\in \gamma$. As before, if we set
\begin{align*}
    L_\xi=\frac{1-i w_1 \partial_{\xi_1} }{1+|w_1|^2 },\quad w_1=\lambda (\partial_{\xi_1} \varphi(t, x, \xi)-y_1),
\end{align*}
then we have
\begin{align*}
    L_\xi (e^{i\lambda(t+\varphi(t, x, \xi)-y_1 \xi_1)})=e^{i\lambda(t+\varphi(t, x, \xi)-y_1 \xi_1)}.
\end{align*}
By Lemma \ref{Lemma varphi derivative}, we have
\begin{align*}
    |\partial_{\xi_1}^k \varphi|\lesssim |\xi_2|^2 |t|,\quad \text{for } k\geq 2,
\end{align*}
which in turn implies that
\begin{align}\label{xi1 deriv of w1}
    |\partial_{\xi_1}^k w_1|\lesssim \lambda |\xi_2|^2 |t|\lesssim \lambda (\lambda^{-\frac{1}{3}})^2 \lambda^{-\frac{1}{3}} \lesssim 1, \quad k\geq 1.
\end{align}
Integration by parts, as before, gives, for $x, y, \in \gamma$,
\begin{align*}
    \left| \int e^{i\lambda [t+\varphi(t, x, \xi)-y_1 \xi_1]} \Tilde{a}_J (t, x, y, \xi)\:d\xi \right|=\left|\int e^{i\lambda [t+\varphi(t, x, \xi)-y_1 \xi_1]} (L_\xi^T)^N (\Tilde{a}_J(t, x, y, \xi)) \:d\xi \right|,
\end{align*}
where $L_\xi^T$ is the transpose of $L_\xi$. As above, we write (up to signs)
\begin{align*}
    L_\xi^T \Tilde{a}_J=B_0+B_1+B_2,
\end{align*}
where
\begin{align*}
    B_0=\frac{1}{1+w_1^2} \Tilde{a}_J,\quad B_1=\frac{i w_1}{1+w_1^2} \partial_{\xi_1} \Tilde{a}_J,\quad B_2=\Tilde{a}_J \partial_{\xi_1} \left(\frac{i w_1}{1+w_1^2} \right).
\end{align*}
By \eqref{Size estimates of b lambda J}, we have
\begin{align}\label{bound of B0 and B1}
    |B_0|,\; |B_1|\lesssim \frac{1}{(1+w_1^2)^{\frac{1}{2}}}\lesssim \frac{1}{1+|w_1|}.
\end{align}
Since we have
\begin{align*}
    B_2=\Tilde{a}_J \partial_{w_1}\left(\frac{i w_1}{1+w_1^2} \right)\partial_{\xi_1} w_1=\Tilde{a}_J \frac{i(1-w_1^2)}{(1+w_1^2)^2} \partial_{\xi_1} w_1,
\end{align*}
it follows from \eqref{xi1 deriv of w1} that
\begin{align*}
    |B_2|\leq |\Tilde{a}_J|\frac{1+w_1^2}{(1+w_1^2)^2}|\partial_{\xi_1}w_1|\lesssim \frac{1}{1+w_1^2}\lesssim \frac{1}{(1+|w_1|)^2}.
\end{align*}
By this and \eqref{bound of B0 and B1}, we have
\begin{align*}
    |B_0|,\; |B_1|,\; |B_2|\lesssim \frac{1}{1+|w_1|}.
\end{align*}

Hence, integration by parts gives, for $x, y, \in \gamma$,
\begin{align*}
    |K_J (x, y)|\leq C_N \lambda^2 \iint_{|t|\lesssim \lambda^{-\frac{1}{3}}, |\xi_2|\lesssim \lambda^{-\frac{1}{3}}, |\xi|_g \approx 1 } |\widehat{\chi^2}(t)| (1+\lambda |\partial_{\xi_1}\varphi(t, x, \xi)-y_1|)^{-N}\:d\xi\:dt.
\end{align*}
In Fermi coordinates, we write $\gamma=\{(x_1, 0): |x_1|\leq \epsilon\}$ for $\epsilon>0$ small, and so, $x=(x_1, 0)$ and $y=(y_1, 0)$. We thus want to show that
\begin{align}\label{KJ Young's inequality}
    \int |K_J (x_1, 0, y_1, 0)|\:dx_1 \lesssim \lambda^{\frac{1}{3}},\quad \int |K_J (x_1, 0, y_1, 0)|\:dy_1 \lesssim \lambda^{\frac{1}{3}}.
\end{align}
Indeed, this and Young's inequality imply \eqref{QJ TT* claim} immediately.

We first focus on $\int |K_J (x_1, 0, y_1, 0)|\:dy_1$. We take $\Tilde{C}>0$ sufficiently large, and bound
\begin{align*}
    \int |K_J (x_1, 0, y_1, 0)|\:dy_1 &\leq C_N \lambda^2 \iiint_{|t|\lesssim \lambda^{-\frac{1}{3}}, |\xi_2|\lesssim\lambda^{-\frac{1}{3}}, |\xi|_g \approx 1 } |\widehat{\chi^2}(t)| (1+\lambda|\partial_{\xi_1} \varphi(t, x, \xi)-y_1|)^{-N}\:dy_1\:d\xi\:dt \\
    &\lesssim \lambda^2 \lambda^{-1} \lambda^{-\frac{1}{3}} \mathrm{Vol}(\{|\xi_2|\lesssim \lambda^{-\frac{1}{3}}, |\xi|\approx 1\})\lesssim \lambda^{\frac{1}{3}}.
\end{align*}
Here, we gained $\lambda^{-1}$ from $y_1$ integration, $\lambda^{-\frac{1}{3}}$ from $t$ integration due to $|t|\lesssim \lambda^{-\frac{1}{3}}$, and $\mathrm{Vol}(\{|\xi_2|\lesssim \lambda^{-\frac{1}{3}}, |\xi|\approx 1\})$ from $\xi_2$ integration. 

The proof of the second inequality in \eqref{KJ Young's inequality} is similar, but uses that $|\partial_{x_1 \xi_1}^2 \varphi (t, x, \xi)|\gtrsim 1$ for small $t$ as in the case $j\leq J-1$.

This completes the proof of Proposition \ref{Prop: QJ estimates}.

\subsection{Proof of Proposition \ref{Prop: I-Qj estimates}}\label{SS: I-Qj estimates}
As we promised before, we talk about Proposition \ref{Prop: I-Qj estimates} here. Let $\Tilde{Q}=I-\sum_{j\leq J} Q_j$. By the Fourier inversion formula, we write
\begin{align*}
    \Tilde{Q} f(x)=\int \Tilde{Q}(x, y) f(y)\:dy,
\end{align*}
where
\begin{align*}
    \Tilde{Q}(x, y)=\frac{1}{(2\pi)^2} \int e^{i(x-y)\cdot \xi} \Big(1-\sum_{j\leq J} q_j (x, y, \xi)\Big)\:d\xi.
\end{align*}
Setting
\begin{align*}
    \Tilde{q}(x, y, \xi)=1-\sum_{j\leq J} q_j (x, y, \xi),
\end{align*}
we write
\begin{align*}
    \Tilde{q}(x, y, \xi)&=1-\chi_0 (\rho(x, \gamma)) \Tilde{\chi}_0 (\rho (y, \gamma))\sum_{j\leq J} \Tilde{\chi}_j \left(2^j \frac{|\xi(N)|}{|\xi|_g} \right) \Upsilon(|\xi|_g/\lambda)\\
    &=1-\chi_0 (\rho(x, \gamma)) \Tilde{\chi}_0 (\rho(y, \gamma)) \Upsilon(|\xi|_g/\lambda).
\end{align*}
Let $\Tilde{Q}$ be a pseudodifferential operator whose kernel is $\Tilde{Q}(x, y)$. Since $\chi_0, \Tilde{\chi}_0$, and $\Upsilon$ are compactly supported bump functions, we have
\begin{align*}
    |\partial_{x, y, \xi}^\alpha \Tilde{q}(x, y, \lambda\xi)|\leq C_\alpha,
\end{align*}
and so, we can consider integration by parts below easily.

We write the kernel of $\Tilde{Q}\circ \chi(\lambda-P)$ as
\begin{align*}
    (\Tilde{Q}\circ \chi(\lambda-P))(x, y)=\frac{\lambda^4}{(2\pi)^3 } \iiiint e^{i\lambda \Psi(t, z, \eta, \xi)} \widehat{\chi}(t) \Tilde{q}(x, z, \lambda \eta) a(t, z, \lambda \xi) \:dt\:dz\:d\eta\:d\xi,
\end{align*}
where
\begin{align*}
    \Psi(t, z, \eta, \xi)=(x-z)\cdot \eta+ \varphi(t, z, \xi) -y\cdot \xi.
\end{align*}
We note that, on the support of $\Tilde{q}(x, z, \lambda \eta)$ in $\eta$,
\begin{align*}
    |\nabla_{t, z}\Psi(t, z, \eta, \xi)|&=|(\Psi_t', \Psi_z')|=\sqrt{|1-|\nabla_z \varphi(t, z, \xi)|_{g(z)}|^2+|\nabla_z \varphi(t, z, \xi)-\eta|^2} \\
    &\gtrsim |1-|\nabla_z \varphi(t, z, \xi)|_{g(z)}|+||\nabla_z \varphi(t, z, \xi)|_{g(z)}-|\eta||\geq |1-|\eta||\gtrsim 1+|\eta|.
\end{align*}
With this in mind, we first consider the integral
\begin{align}\label{Integral 1-Upsilon for 2.6}
    \frac{\lambda^4}{(2\pi)^3 } \iiiint e^{i\lambda \Psi(t, z, \eta, \xi)} \widehat{\chi}(t) \Tilde{q}(x, z, \lambda \eta) a(t, z, \lambda \xi) (1-\Upsilon(|\xi|)) \:dt\:dz\:d\eta\:d\xi.
\end{align}
On the support of $1-\Upsilon(\xi)$ in $\xi$, we have that
\begin{align*}
    |\nabla_{t, z} \Psi (t, z, \eta, \xi)|\gtrsim |1-|\nabla_z \varphi(t, z, \xi)|_{g(z)}|=|1-|\xi|_{g(\nabla_\xi \varphi(t, z, \xi))}|\approx |1-|\xi||\approx 1+|\xi|,
\end{align*}
when we choose $c_1>0$ small enough in \eqref{Construction of the compound symbol qj}. Integration by parts in $t$ and $z$ then gives us that the integral \eqref{Integral 1-Upsilon for 2.6} is dominated by
\begin{align*}
    C\lambda^4 \lambda^{-N} \iiiint_{t\in \mathrm{supp}(\widehat{\chi}), |z|\lesssim 1} (1+|\eta|)^{-N'} (1+|\xi|)^{-N'}\:dt\:dz\:d\eta\:d\xi \lesssim \lambda^{4-N},
\end{align*}
when we take $N, N'$ large enough. Using the generalized Young's inequality, this satisfies the estimates \eqref{I-Qj estimates}, and thus, we focus on the integral
\begin{align*}
    \frac{\lambda^4}{(2\pi)^3 } \iiiint e^{i\lambda \Psi(t, z, \eta, \xi)} \widehat{\chi}(t) \Tilde{q}(x, z, \lambda \eta) a(t, z, \lambda \xi) \Upsilon(|\xi|) \:dt\:dz\:d\eta\:d\xi.
\end{align*}
In this case, the amplitude of the integral is compactly supported in $\xi$, and so, we do not need to consider $|\Psi_t'|$ separately. Thus, integration by parts in $t$ and $z$, the integral is dominated by
\begin{align*}
    C\lambda^4 \lambda^{-N}\iiiint_{t\in \mathrm{supp}(\widehat{\chi}), |z|\lesssim 1, |\xi|\approx 1 } (1+|\eta|)^{-N}\:dt\:dz\:d\eta\:d\xi \lesssim \lambda^{4-N},
\end{align*}
when we take $N$ large enough, which proves Proposition \ref{Prop: I-Qj estimates}.

This shows \eqref{I-Qj estimates} by using the generalized Young's inequality, and thus, completes the proof of Theorem \ref{Theorem Universal Estimates}.

\section{Proof of Theorem \ref{Theorem Log Improvement}}\label{S:Prop for large j}
In this section, assuming nonpositive sectional curvatures on $M$, we want to prove Theorem \ref{Theorem Log Improvement}. Let
\begin{align}\label{T Definition}
    T=c_0 \log \lambda,
\end{align}
where $c_0>0$ is small but fixed, which will be specified later. Let $P=\sqrt{-\Delta_g}$ as before. As in Theorem \ref{Theorem Universal Estimates}, we would have Theorem \ref{Theorem Log Improvement}, if we could show that
\begin{align}\label{Thm2 chi reduction 1}
    \| \chi(T(\lambda-P)) f\|_{L^p (\gamma)}\leq C_p \frac{\lambda^{\frac{1}{3}-\frac{1}{3p}}}{T^{\frac{1}{2}} } \|f\|_{L^2 (M)},\quad 2\leq p<4,
\end{align}
where $C_p \to \infty$ as $p\to 4$. Since $\chi(0)=1$ and $\chi\in \mathcal{S}(\mathbb{R})$, by the mean value theorem, we have, for some $\Tilde{c}$ between $0$ and $t$,
\begin{align*}
    |(1-\chi(t)) \chi(Tt)|&=|(\chi(0)-\chi(t))\chi(Tt)| \\
    &=|\chi'(\Tilde{c})t\chi(Tt)|\leq C_N T^{-1} (1+T|t|)^{-N}.
\end{align*}
As in \cite{Sogge2017ImprovedCritical}, \cite{XiZhang2017improved}, and \cite{BlairSogge2019logarithmic}, etc., by this and the universal estimates in Theorem \ref{Theorem Universal Estimates}, we have that
\begin{align*}
    \| (I-\chi(\lambda-P)) \circ \chi(T(\lambda-P)) f\|_{L^2 (\gamma)}\lesssim \frac{\lambda^{\frac{1}{6}}}{T} \|f\|_{L^2 (M)}.
\end{align*}
Similarly, using \cite[Theorem 1]{BurqGerardTzvetkov2007restrictions} (see also \cite[Theorem 1.1]{Hu2009lp}) instead of Theorem \ref{Theorem Universal Estimates}, we have that
\begin{align*}
    \| (I-\chi(\lambda-P)) \circ \chi(T(\lambda-P)) f\|_{L^4 (\gamma)}\lesssim \frac{\lambda^{\frac{1}{4}}}{T} \|f\|_{L^2 (M)}.
\end{align*}
By interpolation, we have that
\begin{align*}
    \| (I-\chi(\lambda-P)) \circ \chi(T(\lambda-P)) f\|_{L^p (\gamma)}\lesssim \frac{\lambda^{\frac{1}{3}-\frac{1}{3p}}}{T} \|f\|_{L^2 (M)},\quad 2\leq p\leq 4.
\end{align*}
We would therefore have \eqref{Thm2 chi reduction 1} if we could show
\begin{align*}
    \| \chi(\lambda-P) \circ \chi(T(\lambda-P)) f\|_{L^p(\gamma)} \leq C_p \frac{\lambda^{\frac{1}{3}-\frac{1}{3p}}}{T^{\frac{1}{2}} } \|f\|_{L^2(M)},\quad 2\leq p<4.
\end{align*}
This follows from
\begin{align}\label{Qj chi chi T estimate}
    \sum_{j\leq J} \| Q_j \circ \chi (\lambda-P)\circ \chi(T(\lambda-P)) f\|_{L^p (\gamma)}\leq C_p \frac{\lambda^{\frac{1}{3}-\frac{1}{3p}}}{T^{\frac{1}{2}} } \|f\|_{L^2 (M)}, \quad 2\leq p<4,
\end{align}
and, for $N=1, 2, 3, \cdots$,
\begin{align}\label{I-Qj chi chi T estimate}
    \|(I-\sum_{j\leq J} Q_j) \circ \chi(\lambda-P)\circ \chi(T(\lambda-P)) f\|_{L^p (\gamma)}\lesssim \lambda^{-N} \|f\|_{L^2 (M)},\quad 2\leq p<4.
\end{align}

We first show \eqref{I-Qj chi chi T estimate}. Recall that
\begin{align}\label{chi T L2 to L2}
    \|\chi(T(\lambda-P)) f\|_{L^2 (M)}\lesssim \|f\|_{L^2 (M)}.
\end{align}
By Proposition \ref{Prop: I-Qj estimates} and \eqref{chi T L2 to L2}, we have, for $2\leq p<4$ and $N=1, 2, 3, \cdots$,
\begin{align*}
    \|(I-\sum_{j\leq J} Q_j)\circ \chi(\lambda-P)\circ \chi(T(\lambda-P)) f\|_{L^p (\gamma)} \leq C_N \lambda^{-N} \|\chi(T(\lambda-P)) f\|_{L^2 (M)}\lesssim \lambda^{-N} \|f\|_{L^2 (M)},
\end{align*}
which is better than \eqref{I-Qj chi chi T estimate}, and so, we are left to show \eqref{Qj chi chi T estimate}.

Before we proceed further, let us look at the $L^2 (M) \to L^4(\gamma)$ estimate of $Q_j \circ \chi(\lambda-P)$.

\begin{lemma}\label{Lemma L2 to L4}
For $j\leq J$, we have
\begin{align*}
    \| Q_j \circ \chi(\lambda-P) f\|_{L^4 (\gamma)}\leq C \lambda^{\frac{1}{4}} \|f\|_{L^2 (M)}. 
\end{align*}
\end{lemma}

It follows from \eqref{chi T L2 to L2} that
\begin{align}\label{Result of Lemma 4.1}
    \|Q_j \circ \chi(\lambda-P)\circ \chi(T(\lambda-P)) f\|_{L^4(\gamma)}\lesssim \lambda^{\frac{1}{4}} \|\chi(T(\lambda-P)) f\|_{L^2 (M)}\lesssim \lambda^{\frac{1}{4}} \|f\|_{L^2 (M)}.
\end{align}

\begin{proof}
In Fermi coordinates as above, we write, for $\epsilon>0$ small,
\begin{align*}
    \gamma=\{(r, 0): |r|\leq \epsilon \},\quad \gamma_c=\{(x_1, x_2): |x_1|\leq \epsilon,\; x_2=c \}.
\end{align*}
We first show that
\begin{align}\label{Restriction to gamma c estimate}
    \| \mathcal{R}_\gamma \circ Q_j g\|_{L^4 (\gamma)}\lesssim \sup_{|c|\leq \epsilon} \|\mathcal{R}_{\gamma_c} g \|_{L^4 (\gamma_c)},
\end{align}
where $\mathcal{R}_\gamma g$ and $\mathcal{R}_{\gamma_c} g$ are the restrictions of $g$ onto $\gamma$ and $\gamma_c$, respectively.

We can write
\begin{align*}
    (\mathcal{R}_\gamma \circ Q_j)(r, y)=\frac{1}{(2\pi)^2} \int e^{i[(r-y_1)\xi_1-y_2 \xi_2]} q_j (r, 0, \xi) \:d\xi.
\end{align*}
We may assume $|y_1|, |y_2|\leq \epsilon$ by a partition of unity if necessary. By \eqref{Symbol Q properties}, integration by parts then gives
\begin{align*}
    |(\mathcal{R}_\gamma \circ Q_j)(r, y)|&\leq C_N \lambda^2 2^{-j} (1+\lambda|r-y_1|+\lambda 2^{-j}|y_2|)^{-2N} \\
    &\leq C_N \lambda^2 2^{-j} (1+\lambda |r-y_1|)^{-N}(1+\lambda 2^{-j}|y_2|)^{-N},\quad N=1, 2, 3, \cdots.
\end{align*}
This implies that
\begin{align*}
    \int |(\mathcal{R}_\gamma\circ Q_j)(r, y_1, y_2)|\:dr,\quad \int |(\mathcal{R}_\gamma\circ Q_j)(r, y_1, y_2)|\:dy_1 \lesssim C_N \lambda 2^{-j} (1+\lambda 2^{-j} |y_2|)^{-N}.
\end{align*}
By Young's inequality, we then have that
\begin{align*}
    \| \mathcal{R}_\gamma \circ Q_j g (\cdot, y_2) \|_{L_r^4 ([-\epsilon, \epsilon])} \lesssim \lambda 2^{-j} (1+\lambda 2^{-j}|y_2|)^{-N} \|g(\cdot, y_2)\|_{L_{y_1}^4([-\epsilon, \epsilon])}.
\end{align*}
By this and Minkowski's inequality for integrals, we have that
\begin{align*}
    \|\mathcal{R}_\gamma \circ Q_j g \|_{L^4 (\gamma)}&=\left(\int \left|\int\left[\int (\mathcal{R}_\gamma \circ Q_j)(r, y_1, y_2)g(y_1, y_2)\:dy_1\right]\:dy_2 \right|^4\:dr \right)^{\frac{1}{4}} \\
    &\leq \int \| (\mathcal{R}_\gamma \circ Q_j) g(\cdot, y_2) \|_{L^4 ([-\epsilon, \epsilon])}\:dy_2 \\
    &\lesssim \lambda 2^{-j}\int (1+\lambda 2^{-j}|y_2|)^{-N} \|g(\cdot, y_2)\|_{L^4 ([-\epsilon, \epsilon])}\:dy_2 \\
    &\lesssim \sup_{|y_2|\leq \epsilon} \|g(\cdot, y_2)\|_{L^4 ([-\epsilon, \epsilon])}=\sup_{|c|\leq \epsilon} \|g(\cdot, c)\|_{L^4 ([-\epsilon, \epsilon])}=\sup_{|c|\leq \epsilon} \|\mathcal{R}_{\gamma_c} g\|_{L^4 (\gamma_c)},
\end{align*}
which proves \eqref{Restriction to gamma c estimate}.

By (the proof of) \cite[Theorem 1]{BurqGerardTzvetkov2007restrictions} and \cite[Theorem 1.1]{Hu2009lp}, we know that
\begin{align*}
    \sup_{|c|\leq \epsilon} \|\mathcal{R}_{\gamma_c} \circ \chi(\lambda-P) f\|_{L^4 (\gamma_c)} \lesssim \lambda^{\frac{1}{4}} \|f\|_{L^2 (M)}.
\end{align*}
Combining this and \eqref{Restriction to gamma c estimate} with $g=\chi(\lambda-P)f$, we obtain that
\begin{align*}
    \|\mathcal{R}_\gamma \circ Q_j \circ \chi(\lambda-P) f\|_{L^4(\gamma)} \lesssim \sup_{|c|\leq \epsilon} \|\mathcal{R}_{\gamma_c}\circ \chi(\lambda-P) f\|_{L^4 (\gamma_c)} \lesssim \lambda^{\frac{1}{4}}\|f\|_{L^2 (M)}.
\end{align*}
Here, the implicit constants are uniform, which are stable under $C^\infty$ perturbation of $\gamma$. This completes the proof.
\end{proof}

By Proposition \ref{Prop: q plus minus estimates} and Lemma \ref{Lemma L2 to L4}, we have that, for $j\leq J$,
\begin{align*}
    & \|Q_j \circ \chi(\lambda-P) f\|_{L^2 (\gamma)} \leq C 2^{\frac{j}{2}} \|f\|_{L^2 (M)}, \\
    & \|Q_j \circ \chi(\lambda-P) f\|_{L^4 (\gamma)} \leq C \lambda^{\frac{1}{4}} \|f\|_{L^2 (M)}.
\end{align*}
By interpolation, we have
\begin{align}\label{Qj chi L2 to Lp}
    \| Q_j \circ \chi(\lambda-P) f\|_{L^p (\gamma)}\leq C 2^{\frac{j}{2}(\frac{4}{p}-1)} \lambda^{\frac{1}{4}(2-\frac{4}{p})} \|f\|_{L^2 (M)}, \quad 2\leq p<4.
\end{align}
Let $\epsilon>0$ be a fixed but small number, which will be specified later. By \eqref{Qj chi L2 to Lp} and \eqref{chi T L2 to L2}, if $2\leq p<4$, then
\begin{align*}
    & \sum_{j\leq \lfloor \log_2 \lambda^{\frac{1}{3}-\epsilon}\rfloor} \|Q_j \circ \chi (\lambda-P) \circ \chi (T(\lambda-P))f \|_{L^p (\gamma)} \\
    &\hspace{100pt}\leq C \sum_{j\leq \lfloor \log_2 \lambda^{\frac{1}{3}-\epsilon}\rfloor} 2^{\frac{j}{2}(\frac{4}{p}-1)} \lambda^{\frac{1}{4}(2-\frac{4}{p})} \|\chi(T(\lambda-P)) f\|_{L^2 (M)} \\
    &\hspace{100pt}\leq \frac{2C}{1-2^{-\frac{1}{2}(\frac{4}{p}-1)}} \lambda^{\frac{1}{3}-\frac{1}{3p}-\frac{\epsilon}{2}(\frac{4}{p}-1)} \|f\|_{L^2 (M)} \\
    &\hspace{100pt}\leq \frac{2C}{1-2^{-\frac{1}{2}(\frac{4}{p}-1)}} \frac{\lambda^{\frac{1}{3}-\frac{1}{3p}}}{T^{\frac{1}{2}}}\|f\|_{L^2 (M)},
\end{align*}
which satisfies \eqref{Qj chi chi T estimate}.

\begin{remark}
We note that we cannot relax the condition $C_p \to \infty$ as $p\to 4$ in our argument. Indeed, note that
\begin{align*}
    \lim_{\lambda\to \infty} \lim_{\epsilon\to 0} \frac{\lambda^{\frac{1}{3}-\frac{1}{3p}-\frac{\epsilon}{2}(\frac{4}{p}-1)}}{\lambda^{\frac{1}{3}-\frac{1}{3p}}/T^{\frac{1}{2}} } =\lim_{\lambda\to \infty} T^{\frac{1}{2}}=\infty.
\end{align*}
Also, if we set
\begin{align*}
    C_p=\frac{2C}{1-2^{-\frac{1}{2}(\frac{4}{p}-1)}},
\end{align*}
then our argument gives
\begin{align*}
    \sum_{j\leq \lfloor \log_2 \lambda^{\frac{1}{3}-\epsilon}\rfloor} \|Q_j \circ \chi (\lambda-P) \circ \chi (T(\lambda-P))f \|_{L^p (\gamma)} \leq C_p \frac{\lambda^{\frac{1}{3}-\frac{1}{3p}}}{T^{\frac{1}{2}}} \|f\|_{L^2 (M)},
\end{align*}
but we have that $\displaystyle \lim_{p\to 4} C_p=\infty$.
\end{remark}

If we set
\begin{align}\label{mu T definition}
    \chi_T (\zeta)=\chi(\zeta/T),\quad \mu_T (\zeta)=\chi_T (\zeta) \chi(\zeta),
\end{align}
we have $\chi(\lambda-P) \chi(T(\lambda-P))=\mu_T (T(\lambda-P))$. Also, since $\widehat{\chi}_T (\zeta)=T\widehat{\chi}(T\zeta)$ and $\widehat{\mu}_T (t)=(2\pi)^{-1} \widehat{\chi}_T * \widehat{\chi} (t)$, we have, by \eqref{epsilon0 and chi},
\begin{align*}
    \mathrm{supp}(\widehat{\mu}_T) \subset \mathrm{supp}(\widehat{\chi}_T)+\mathrm{supp} (\widehat{\chi})\subset [-\frac{\epsilon_0}{T}, \frac{\epsilon_0}{T}]+[-\epsilon_0, \epsilon_0]\subset [-2\epsilon_0, 2\epsilon_0],
\end{align*}
and so,
\begin{align}\label{mu support}
    \mathrm{supp}(\widehat{\mu_T^2})\subset \mathrm{supp}(\widehat{\mu_T})+\mathrm{supp} (\widehat{\mu_T})\subset [-4\epsilon_0, 4\epsilon_0],
\end{align}
since $T=c_0 \log \lambda \gg 1$. We have shown that
\begin{align*}
    \sum_{j\leq \lfloor \log_2 \lambda^{\frac{1}{3}-\epsilon}\rfloor} \|Q_j \circ \mu_T (T(\lambda-P))f \|_{L^p (\gamma)}\leq \frac{2C}{1-2^{-\frac{1}{2}(\frac{4}{p}-1)}} \frac{\lambda^{\frac{1}{3}-\frac{1}{3p}}}{T^{\frac{1}{2}}}\|f\|_{L^2 (M)},
\end{align*}
For the rest of \eqref{Qj chi chi T estimate}, we want to show that
\begin{align*}
    \|Q_j \circ \mu_T (T(\lambda-P)) f\|_{L^p (\gamma)}\lesssim \frac{\lambda^{\frac{1}{4}}}{T^{\frac{1}{2}}} e^{C'T} (2^{-j})^{\frac{1}{p}} \|f\|_{L^2 (M)},\quad 2\leq p<4,\quad \lfloor \log_2 \lambda^{\frac{1}{3}-\epsilon}\rfloor \leq j\leq J,
\end{align*}
or
\begin{align*}
    \|Q_j \circ \mu_T (T(\lambda-P)) f\|_{L^p(\gamma)} \lesssim \frac{2^{j\left(\frac{2}{p}-\frac{1}{2}\right)} \lambda^{\frac{1}{2}-\frac{1}{p}} }{T^{\frac{1}{2}}} \|f\|_{L^2 (M)},\quad 2\leq p<4.
\end{align*}
Indeed, we have
\begin{align*}
    \sum_{\lfloor \log_2 \lambda^{\frac{1}{3}-\epsilon}\rfloor \leq j\leq J} \frac{\lambda^{\frac{1}{4}}}{T^{\frac{1}{2}}} e^{C'T} (2^{-j})^{\frac{1}{p}}\lesssim \frac{\lambda^{\frac{1}{4}-\frac{1}{3p}+\frac{\epsilon}{p}+C'c_0}}{T^{\frac{1}{2}}} \epsilon \log_2 \lambda \lesssim \frac{\lambda^{\frac{1}{3}-\frac{1}{3p}}}{T^{\frac{1}{2}}},\quad \text{when } \lambda\gg 1,
\end{align*}
and
\begin{align*}
    \sum_{\lfloor \log_2 \lambda^{\frac{1}{3}-\epsilon}\rfloor \leq j\leq J} \frac{2^{j\left(\frac{2}{p}-\frac{1}{2}\right)} \lambda^{\frac{1}{2}-\frac{1}{p}} }{T^{\frac{1}{2}}} \lesssim \frac{\lambda^{\frac{1}{3}-\frac{1}{3p}}}{T^{\frac{1}{2}}}.
\end{align*}
Here, we take $\epsilon>0$ to be sufficiently small and choose a small $c_0>0$ in \eqref{T Definition}.

By the $TT^*$ argument, we would have \eqref{Qj chi chi T estimate} if we could show either
\begin{align}\label{Log improvement TT* Claim}
    \|Q_j \circ \mu_T^2 (T(\lambda-P))\circ Q_j^* f\|_{L^p (\gamma)} \lesssim \frac{\lambda^{\frac{1}{2}}}{T} e^{CT} (2^{-j})^{\frac{2}{p}} \|f\|_{L^{p'} (\gamma)},\quad 2\leq p<4, \quad \lfloor \log_2 \lambda^{\frac{1}{3}-\epsilon}\rfloor \leq j\leq J,
\end{align}
or
\begin{align}\label{Log improvement TT* Claim 2}
    \|Q_j \circ \mu_T^2 (T(\lambda-P))\circ Q_j^* f\|_{L^p (\gamma)}\lesssim \frac{2^{j\left(\frac{4}{p}-1 \right)}\lambda^{1-\frac{2}{p}}}{T} \|f\|_{L^{p'}(\gamma) },\quad 2\leq p<4.
\end{align}

We want to lift this problem to the universal cover of $M$. Let $\Tilde{M}$ be the universal cover of $M$ with the pullback metric $\Tilde{g}$ under the covering map $p:\Tilde{M}\to M$. By the Cartan-Hadamard theorem, $\Tilde{M}$ is diffeomorphic to $\mathbb{R}^2$ with the diffeomorphism $T_{x_0} M\cong \mathbb{R}^2 \to \Tilde{M}$ for any $x_0 \in M$, so that the map $p=\mathrm{exp}_{x_0}: T_{x_0} M \to M$ is a smooth covering map. Without loss of generality, we write $p:\mathbb{R}^2\cong \Tilde{M} \to M$.

Let $D\subset \mathbb{R}^2$ be a fundamental domain of the universal covering $p$ so that every point in $\mathbb{R}^2$ is the translate of exactly one point in $D$. Without loss of generality, we may assume that $\gamma$ and other amplitudes like $q_j$ are supported in $D^\circ$, where $D^\circ$ is the interior of $D$, i.e., $\gamma\subset D^\circ$, and $\mathrm{supp}(q_j)\subset D^\circ$, etc. We write tildes over letters to express that those letters are defined in $\mathbb{R}^2\cong \Tilde{M}$. For example, for any $x\in M$, let $\Tilde{x}\in D$ be the unique point so that $p(\Tilde{x})=x$, $p(\Tilde{\gamma})=\gamma$, and the metric $\Tilde{g}$ on $\mathbb{R}^2\cong \Tilde{M}$ is the pullback metric of $g$, $\Tilde{\rho}(\Tilde{x}, \Tilde{y})$ is the Riemannian distance $d_{\Tilde{g}}(\Tilde{x}, \Tilde{y})$, and so on. Let $\Gamma$ be the group of deck transformations $\alpha$'s, which are diffeomorphisms satisfying $p\circ \alpha =p$. With this in mind, if we have a function $\Tilde{f}$ on $D$, we can extend this $\Tilde{f}$ to $\mathbb{R}^2 \cong \Tilde{M}$ by setting 
\begin{align*}
    \Tilde{f}(\Tilde{x})=\Tilde{f}(\alpha(\Tilde{x})) \quad\text{for } \Tilde{x}\in D.
\end{align*}
Here, since $p:\mathbb{R}^2 \to M$ is a local diffeomorphism, abusing notation we write
\begin{align*}
    & \Tilde{Q}_j (\Tilde{x}, \Tilde{w})=\frac{\lambda^2}{(2\pi)^2} \int e^{i\lambda (\Tilde{x}-\Tilde{w})\cdot \eta} \Tilde{q}_j (\Tilde{x}, \Tilde{w}, \lambda \eta)\:d\eta,\\
    & \Tilde{Q}_j^* (\Tilde{z}, \Tilde{y})=\overline{\Tilde{Q}_j (\Tilde{y}, \Tilde{z})}=\overline{\Tilde{Q}_j (\alpha (\Tilde{y}), \alpha (\Tilde{z}))}=\frac{\lambda^2}{(2\pi)^2} \int e^{-i\lambda (\alpha (\Tilde{y})-\alpha(\Tilde{z}))\cdot \zeta} \Tilde{q}_j (\alpha(\Tilde{y}), \alpha (\Tilde{z}), \lambda \zeta)\:d\zeta.
\end{align*}
Recall that we know from \cite{SoggeZelditch2014eigenfunction} that
\begin{align*}
    (\cos tP)(x, y)=\sum_{\alpha \in \Gamma} (\cos t \sqrt{-\Delta_{\Tilde{g}}}) (\Tilde{x}, \alpha (\Tilde{y})),\quad \Tilde{x}, \Tilde{y}\in D.
\end{align*}
Also recall that, by a counting argument and finite propagation speed as in \cite{SoggeZelditch2014eigenfunction}, there are at most $O(e^{Ct})$ many nonzero terms in the sum.

Using Euler's formula, we have, modulo $O(\lambda^{-N})$ errors,
\begin{align*}
    \chi^2 (T(\lambda-P))(x, y)&=\frac{1}{\pi T} \int e^{it\lambda} \widehat{\chi^2}(t/T) \cos(tP)(x, y)\:dt-\chi^2 (T(\lambda+P))(x, y) \\
    &=\frac{1}{\pi T}\sum_{\alpha\in \Gamma} \int e^{it\lambda} \widehat{\chi^2}(t/T) \cos(t\sqrt{-\Delta_{\Tilde{g}}})(\Tilde{x}, \alpha(\Tilde{y}))\:dt,
\end{align*}
since $\chi^2(T(\lambda+P))(x, y)=O(\lambda^{-N})$.

We want to show that the estimate for $\alpha=\mathrm{Id}$ satisfies \eqref{Log improvement TT* Claim 2}, and the estimate for $\alpha\not=\mathrm{Id}$ satisfies \eqref{Log improvement TT* Claim}.

\begin{lemma}\label{Lemma alpha=Id}
If $\alpha=\mathrm{Id}$ and $2\leq p\leq 4$, then
\begin{align*}
    \left\|\frac{1}{\pi T} \iint e^{it\lambda} \widehat{\mu_T^2}(t/T) (\Tilde{Q}_j \circ\cos(t\sqrt{-\Delta_{\Tilde{g}}})(\cdot, \alpha(\cdot)) \circ \Tilde{Q}_j^*)(\Tilde{\gamma}(\cdot), \Tilde{\gamma}(s))f(s)\:dt\:ds \right\|_{L^p (\gamma)}\lesssim \frac{2^{j\left(\frac{4}{p}-1 \right)}\lambda^{1-\frac{2}{p}}}{T} \|f\|_{L^{p'} (\gamma)},
\end{align*}
which satisfies the estimate \eqref{Log improvement TT* Claim 2}.
\end{lemma}

\begin{proof}
We choose $\beta\in C_0^\infty (\mathbb{R})$ satisfying
\begin{align}\label{beta support in Lemma alpha Id}
    \beta(t)=1 \text{ for } |t|\leq c, \text{ and } \beta(t)=0 \text{ for } |t|\geq 2c,
\end{align}
for a small $c>0$. Since $\beta(t)\widehat{\mu_T^2}(t/T)$ is compactly supported in $t$ and
\begin{align*}
    |\partial_t^k [\beta(t) \widehat{\mu_T^2}(t/T)]| \leq C_k,
\end{align*}
the term $\beta(t)\widehat{\mu_T^2}(t/T)$ plays the same role as $\widehat{\chi^2}(t)$ in \S \ref{S:Proof of universal estimates}. Thus, by the proof of Theorem \ref{Theorem Universal Estimates}, we have, for $\alpha=\mathrm{Id}$,
\begin{align*}
    \left\|\frac{1}{\pi T}\iint e^{it\lambda} \beta(t)\widehat{\mu_T^2}(t/T) (\Tilde{Q}_j \circ\cos(t\sqrt{-\Delta_{\Tilde{g}}})(\cdot, \cdot)\circ \Tilde{Q}_j^*)(\Tilde{\gamma}(\cdot), \Tilde{\gamma}(s))f(s)\:dt\:ds \right\|_{L^2 (\gamma)} \lesssim \frac{2^j}{T}\|f\|_{L^2 (\gamma)}.
\end{align*}
The difference between this and Theorem \ref{Theorem Universal Estimates} is that here we use the Hadamard parametrix about the cosine propagator $\cos(t\sqrt{-\Delta_{\Tilde{g}}})$, and we used the Lax parametrix about $e^{-itP}$.

Similarly, instead of using Theorem \ref{Theorem Universal Estimates}, by using the proof of \eqref{Result of Lemma 4.1} with a $TT^*$ argument, we can obtain, for $\alpha=\mathrm{Id}$,
\begin{align*}
    \left\|\frac{1}{\pi T}\iint e^{it\lambda} \beta(t)\widehat{\mu_T^2}(t/T) (\Tilde{Q}_j \circ \cos(t\sqrt{-\Delta_{\Tilde{g}}}) (\cdot, \cdot) \circ \Tilde{Q}_j^*)(\Tilde{\gamma}(\cdot), \Tilde{\gamma}(s))f(s)\:dt\:ds \right\|_{L^4 (\gamma)} \lesssim \frac{\lambda^{\frac{1}{2}}}{T}\|f\|_{L^{\frac{4}{3}} (\gamma)}.
\end{align*}
The desired estimate then follows from interpolation.

It then suffices to show that, for $\alpha=\mathrm{Id}$ and $N=1, 2, 3, \cdots$,
\begin{align*}
    \left\|\frac{1}{\pi T}\iint e^{it\lambda} (1-\beta(t))\widehat{\mu_T^2}(t/T) (\Tilde{Q}_j \circ\cos(t\sqrt{-\Delta_{\Tilde{g}}})(\cdot, \cdot)\circ \Tilde{Q}_j^*)(\Tilde{\gamma}(\cdot), \Tilde{\gamma}(s)) f(s)\:dt\:ds \right\|_{L^p (\gamma)}\lesssim \lambda^{-N}\|f\|_{L^{p'} (\gamma)}.
\end{align*}
We show this as in \cite[Lemma 3.1]{ChenSogge2014few}. We first consider the kernel of the integral operator inside the $L^2$ norm without $Q_j$ and $Q_j^*$ compositions
\begin{align*}
    \frac{1}{\pi T} \int e^{it\lambda} (1-\beta(t))\widehat{\mu_T^2}(t/T) \cos(t\sqrt{-\Delta_{\Tilde{g}}})(\Tilde{x}, \Tilde{y})\:dt,\quad \Tilde{x}, \Tilde{y}\in D.
\end{align*}
We recall properties of the cosine propagator (cf. \cite[(5.14)]{BlairSogge2015OnKakeyaNikodym}, etc.)
\begin{align*}
    \mathrm{sing}\: \mathrm{supp}(\cos t\sqrt{-\Delta_{\Tilde{g}}})(\cdot, \cdot)\subset \{(\Tilde{x}, \Tilde{z})\in \mathbb{R}^2\times \mathbb{R}^2: \Tilde{\rho} (\Tilde{x}, \Tilde{z})=|t|\},
\end{align*}
that is, $\cos(t\sqrt{-\Delta_{\Tilde{g}}})(\Tilde{x}, \Tilde{z})$ is smooth if $\Tilde{\rho}(\Tilde{x}, \Tilde{z})\not=|t|$. Since $1-\beta(t)=0$ for $|t|\leq c$ where $c>0$ is as in \eqref{beta support in Lemma alpha Id}, we may assume that $|t|\geq c>0$. Here, we choose a sufficiently small $c>0$, compared to the injectivitiy radius of $M$. For $\alpha=\mathrm{Id}$, by a partition of unity if necessary, we may assume that $\Tilde{\rho}(\Tilde{x}, \alpha(\Tilde{y}))\leq c/2$ for $\Tilde{x}, \Tilde{y}\in D$, and thus,
\begin{align*}
    |t|\geq c>\Tilde{\rho}(\Tilde{x}, \alpha(\Tilde{y})),\; \alpha=\mathrm{Id}, \; \Tilde{x}, \Tilde{y}\in D,\quad \text{that is, } \Tilde{\rho}(\Tilde{x}, \alpha(\Tilde{y}))\not=|t|.
\end{align*}
This implies that $\cos(t\sqrt{-\Delta_{\Tilde{g}}})(\Tilde{x}, \Tilde{z})$ is smooth for $\Tilde{x}, \Tilde{z}\in \mathbb{R}^2$, and thus, integration by parts in $t$ implies that
\begin{align*}
    \frac{1}{\pi T}\int e^{it\lambda} (1-\beta(t))\widehat{\mu_T^2}(t/T) \cos(t\sqrt{-\Delta_{\Tilde{g}}})(\Tilde{x}, \alpha(\Tilde{y}))\:dt=O(\lambda^{-N}),\quad \alpha=\mathrm{Id},\; \Tilde{x}, \Tilde{y}\in D.
\end{align*}
For the contribution after compositions of $Q_j$ and $Q_j^*$,  by \eqref{Symbol Q properties}, we note that
\begin{align*}
    & \left|\frac{1}{\pi T} \int e^{it\lambda} (1-\beta(t))\widehat{\mu_T^2}(t/T) (\Tilde{Q}_j \circ \cos(t\sqrt{-\Delta_{\Tilde{g}}})(\cdot, \alpha(\cdot))\circ \Tilde{Q}_j^*)(\Tilde{\gamma}(r), \Tilde{\gamma}(s))\:dt\right| \\
    & \lesssim \frac{1}{T}\left| \iint \Tilde{Q}_j (\Tilde{\gamma}(r), z)\left(\int e^{it\lambda} (1-\beta(t))\widehat{\mu_T^2}(t/T) \cos (t\sqrt{-\Delta_{\Tilde{g}}}) (z, w)\:dt \right) \Tilde{Q}_j^* (w, \Tilde{\gamma}(s))\:dz\:dw \right| \\
    & \lesssim \sup_{z, w}\left(\frac{1}{T}\int e^{it\lambda} (1-\beta(t))\widehat{\mu_T^2}(t/T) \cos (t\sqrt{-\Delta_{\Tilde{g}}}) (z, w)\:dt \right) \iint |\Tilde{Q}_j (\Tilde{\gamma}(r), z)| |Q_j^* (w, \Tilde{\gamma}(s))|\:dz\:dw \\
    & \lesssim \lambda^{-N} \sup_{\Tilde{\gamma}(r)}\int|\Tilde{Q}_j (\Tilde{\gamma}(r), z)|\:dz\; \sup_{\Tilde{\gamma}(s)}\int |\Tilde{Q}_j^* (w, \gamma(s))|\:dw \lesssim \lambda^{-N}. 
\end{align*}
This completes the proof.
\end{proof}

By Lemma \ref{Lemma alpha=Id}, we can ignore the contribution of $\alpha=\mathrm{Id}$. Using Euler's formula, we know
\begin{align*}
    \mu_T^2 (T(\lambda-P))(x, y)=\frac{1}{\pi T} \int e^{it\lambda} \widehat{\mu_T^2}(t/T) (\cos tP)(x, y)\:dt -\mu_T^2 (T(\lambda+P))(x, y),
\end{align*}
and also know that $\mu_T^2 (T(\lambda+P))(x, y)=O(\lambda^{-N})$. As in \S \ref{S:Proof of universal estimates}, if we set
\begin{align*}
    K_j (x, y)=\frac{1}{2\pi T} \int e^{it\lambda} \widehat{\mu_T^2}(t/T) (Q_j \circ e^{-itP}\circ Q_j^*) (x, y)\:dt,
\end{align*}
then, by Euler's formula, modulo $O(\lambda^{-N})$ errors, we have
\begin{align*}
    K_j (x, y)=\frac{1}{\pi T} \int e^{it\lambda} \widehat{\mu_T^2}(t/T) (Q_j \circ \cos tP\circ Q_j^*)(x, y)\:dt,
\end{align*}
which is the kernel of $Q_j \circ \mu_T^2 (T(\lambda-P))\circ Q_j^*$ modulo $O(\lambda^{-N})$ errors. From now on, we focus on $\lfloor \log_2 \lambda^{\frac{1}{3}-\epsilon}\rfloor \leq j< J$. Similar arguments will also work for $j=J$.

By a version of Egorov's theorem in \cite{BouzouinaRobert2002Duke} and the subsequent observation in \cite[Theorem 4.2.4]{Anantharaman2008Entropy}, we have
\begin{align}\label{Long time Egorov thm}
    e^{-it\sqrt{-\Delta_{\Tilde{g}}}}\circ Q_j^*=\Tilde{B}_{t, j} \circ e^{-it\sqrt{-\Delta_{\Tilde{g}}} },
\end{align}
where $\Tilde{B}_{t, j}$ has a symbol
\begin{align*}
    \Tilde{b}_{t, j}=\kappa_t^* \Tilde{q}_j^* +b',
\end{align*}
with the Hamiltonian flow $\kappa_t$, and $| b'|=O(\lambda^{-1+\frac{2}{3}+2\Lambda c_0})=O(\lambda^{-\frac{1}{3}+2\Lambda c_0})$ for some fixed $\Lambda>0$. Since $M$ is compact, we have $|b'|=O(\lambda^{-\frac{1}{3}+2\Lambda c_0})=O(\lambda^{-\frac{1}{3}+\epsilon'})$ for some small $\epsilon'>0$ when taking $c_0>0$ to be sufficiently small in \eqref{T Definition} for a uniform constant $\Lambda$. As before, we will ignore the contribution of the remainder $b'$, and write $b_{t, j}=\kappa_t^* q_j$. Using Euler's formula again, we can replace $e^{-it \sqrt{-\Delta_{\Tilde{g}}}}$ by $\cos t \sqrt{-\Delta_{\Tilde{g}}}$ modulo $O(\lambda^{-N})$ errors in \eqref{Long time Egorov thm}.

With this in mind, modulo $O(\lambda^{-N})$ errors, we have that
\begin{align}\label{Kj mod errors}
    K_j (x, y)=\frac{1}{\pi T} \int e^{it\lambda}\widehat{\mu_T^2}(t/T) (Q_j\circ B_{t, j}\circ \cos tP)(x, y)\:dt,
\end{align}
where $\Tilde{B}_{t, j}$ is the lift of $B_{t, j}$. Since $p$ is a local isometry, we may assume $|\det p|=1$ in Riemannian measure. Using $w=p(\Tilde{w})$ and $z=p(\Tilde{z})$, we write
\begin{align*}
    (Q_j \circ B_{t, j} \circ \cos t\sqrt{-\Delta_g})(x, y)&=\sum_\alpha \iint_{D^2} \Tilde{Q}_j (\Tilde{x}, \Tilde{w}) \Tilde{B}_{t, j}(\Tilde{w}, \Tilde{z}) (\cos t\sqrt{-\Delta_{\Tilde{g}}})(\Tilde{z}, \alpha(\Tilde{y})) |\det p|^2 \:d\Tilde{w}\:d\Tilde{z} \\
    &=\sum_\alpha \iint_{D^2} \Tilde{Q}_j (\Tilde{x}, \Tilde{w}) \Tilde{B}_{t, j} (\Tilde{w}, \Tilde{z}) (\cos t\sqrt{-\Delta_{\Tilde{g}}}) (\Tilde{z}, \alpha (\Tilde{y}))\:d\Tilde{w}\:d\Tilde{z}.
\end{align*}

By the Hadamard parametrix (cf. \cite{Berard1977onthewaveequation}, \cite{Sogge2014hangzhou}, \cite{SoggeZelditch2014eigenfunction}, etc.), we write, for $\Tilde{x}\in D$ and $\Tilde{w}\in \alpha(D)$,
\begin{align*}
    (\cos t\sqrt{-\Delta_{\Tilde{g}}})(\Tilde{x}, \Tilde{w})=\Tilde{K}_N (t, \Tilde{x}; \Tilde{w})+ \Tilde{R}_N (t, \Tilde{x}; \Tilde{w}),
\end{align*}
where
\[
\Tilde{K}_N (t, \Tilde{x};\Tilde{w})=\begin{cases}
\sum_{\nu=0}^N u_\nu (\Tilde{x}, \Tilde{w}) \partial_t E_\nu (t, \Tilde{\rho} (\Tilde{x}, \Tilde{w})), & t\geq 0, \\
-\sum_{\nu=0}^N u_\nu (\Tilde{x}, \Tilde{w}) \partial_t E_\nu (-t, \Tilde{\rho} (\Tilde{x}, \Tilde{w})), & t< 0.
\end{cases}
\]
We explain the $u_\nu$ and $E_\nu$ below.

For simplicity, from now on, we focus on $t\geq 0$. Similar arguments work for $t\leq 0$. Here, the $C^\infty$ functions $u_\nu$ are as in \cite[\S 2 in B and (10)]{Berard1977onthewaveequation} and \cite[p.35]{Sogge2014hangzhou}:
\begin{align*}
    \begin{split}
         & u_0 (\Tilde{x}, \Tilde{w})=\Theta^{-\frac{1}{2}} (\Tilde{x}, \Tilde{w}), \\
        & u_\nu (\Tilde{x}, \Tilde{w})=\Theta^{-\frac{1}{2}} (\Tilde{x}, \Tilde{w})\int_0^1 s^{\nu-1} \Theta^{1/2} (\Tilde{x}, \Tilde{x}_s) \Delta_{\Tilde{g}, \Tilde{w}} u_{\nu-1}(\Tilde{x}, \Tilde{x}_s)\:dx,\quad \nu=1, 2, 3, \cdots, \\
        &\Theta (\Tilde{x}, \Tilde{w})=|\det D_{\exp_{\Tilde{x}}^{-1} (\Tilde{w})} \exp_{\Tilde{x}}|,
    \end{split}
\end{align*}
where $\Tilde{x}_s$ is the minimizing geodesic from $\Tilde{x}$ to $\alpha(\Tilde{w})$ parametrized by arc length and
\begin{align*}
    \Theta=(\det (\Tilde{g}_{jk}))^{\frac{1}{2}}.
\end{align*}

As in \cite[Chapter 1]{Sogge2014hangzhou}, the distributions $E_\nu$ are, in $\mathbb{R}^n$,
\begin{align*}
    E_\nu (t, x)=\lim_{\epsilon\to 0+} \nu! (2\pi)^{-n-1} \iint_{\mathbb{R}^{1+n}} e^{ix\cdot \xi+it\tau} (|\xi|^2-(\tau-i\epsilon)^2)^{-\nu-1} \:d\xi\:d\tau,\quad \nu=0, 1, 2, \cdots,
\end{align*}
and
\begin{align*}
    & E_0 (t, x)=H(t)\times (2\pi)^{-n} \int_{\mathbb{R}^n} e^{ix\cdot \xi} \frac{\sin t|\xi|}{|\xi|}\:d\xi,\\
    & \Box E_\nu=\nu E_{\nu-1},\quad -2\frac{\partial E_\nu}{\partial x}=xE_{\nu-1},\quad 2\frac{\partial E_\nu}{\partial t}=tE_{\nu-1},\quad \nu=1, 2, 3, \cdots,
\end{align*}
where $H(t)$ is the Heaviside function
\[
H(t)=\begin{cases}
1, & t\geq 0,\\
0, & t<0.
\end{cases}
\]
We have $n=2$ in our work. Here, $E_0 (t, x)$ is interpreted as
\begin{align*}
    E_0 (t, x)=\begin{cases}
    (2\pi)^{-2} \int_{\mathbb{R}^2} e^{ix\cdot \xi} \frac{\sin t |\xi|}{|\xi|}\:d\xi, & t\geq 0,\\
    0, & t<0,
\end{cases}
\end{align*}
and
\begin{align*}
    \langle E_0 (t, \cdot), f \rangle=(2\pi)^{-2}H(t)\int_{\mathbb{R}^2} \frac{\sin t|\xi|}{|\xi|}\hat{f}(\xi)\:d\xi,
\end{align*}
that is, the Fourier transform of $E_0 (t, \cdot)$ is $\frac{\sin t|\xi|}{|\xi|}$. Also, since the $E_\nu$ are radial in $x$, we may abuse notation, for example, $E_\nu (t, x)=E_\nu (t, |x|)$.

We will ignore the contribution of $\Tilde{R}_N$. We first recall a result in \cite{Berard1977onthewaveequation}, \cite[Theorem 3.1.5]{Sogge2014hangzhou}, and \cite[Proposition 3.1]{Keeler2019TwoPointWeylLaw}, adapted to our settings.

\begin{lemma}[\cite{Berard1977onthewaveequation}, \cite{Sogge2014hangzhou}, \cite{Keeler2019TwoPointWeylLaw}]\label{Lemma: remainder small}
For $|t|\leq T$, we have $\Tilde{R}_N \in C^{N-5}([-T, T]\times D\times D)$ and
\begin{align*}
    |\partial_{t, x, y}^\beta \Tilde{R}_N (t, \Tilde{x}; \Tilde{w})|\lesssim e^{C_\beta T},\quad \text{if } |\beta|\ll N.
\end{align*}
\end{lemma}

Let $\Tilde{R}_N$ be the operator whose kernel is
\begin{align*}
    \frac{1}{\pi T}\int e^{it\lambda} \widehat{\mu_T^2} (t/T) \Tilde{R}_N (t, \Tilde{x};\Tilde{w})\:dt.
\end{align*}
By Lemma \ref{Lemma: remainder small}, integration by parts in $t$ gives
\begin{align*}
    \frac{1}{\pi T}\int e^{it\lambda} \widehat{\mu_T^2} (t/T) \Tilde{R}_N (t, \Tilde{x};\Tilde{w})\:dt=O(\pi^{-1} T^{-1} (2T) \lambda^{-N'} e^{C_{N'} T})=O(\lambda^{-N}),\quad N=1, 2, 3, \cdots.
\end{align*}
By \eqref{Symbol Q properties} again as in the proof of Lemma \ref{Lemma alpha=Id}, we can obtain
\begin{align*}
    \left(\Tilde{Q}_j \circ \left(\frac{1}{\pi T} \int e^{it\lambda} \widehat{\mu_T^2}(t/T)\Tilde{R}_N (t, \cdot;\cdot)\:dt \right)\circ \Tilde{Q}_j^*\right) (\Tilde{\gamma}(r), \Tilde{\gamma}(s))=O(\lambda^{-N}),\quad N=1, 2, 3, \cdots,
\end{align*}
and thus, by Young's inequality, we ignore the contribution of $\Tilde{R}_N$, when we take $N\gg 1$.

We can also ignore the contribution of $E_\nu$ for $\nu\geq 1$.

\begin{lemma}[Theorem 3.4 in \cite{Chen2015improvement}]\label{Lemma: nu positive small}
We have, for $\Tilde{x}\in D$ and $\Tilde{w}\in \alpha(D)$,
\begin{align*}
    \int e^{it\lambda} \widehat{\mu_T^2}(t/T) \partial_t E_\nu (t, \Tilde{\rho} (\Tilde{x}, \Tilde{w}))\:dt=O(\lambda^{1-2\nu}),\quad \nu=0, 1, 2, \cdots.
\end{align*}
\end{lemma}

By the same arguments as in $\Tilde{R}_\lambda$, Lemma \ref{Lemma: nu positive small} gives us that
\begin{align*}
    \left(\Tilde{Q}_j \circ \left(\frac{1}{\pi T} \int e^{it\lambda} \widehat{\mu_T^2}(t/T) \partial_t E_\nu (t, \Tilde{\rho}(\cdot, \cdot))\:dt \right)\circ \Tilde{Q}_j^*\right) (\Tilde{\gamma}(r), \Tilde{\gamma}(s))=O(\lambda^{1-2\nu}),\quad \nu=0, 1, 2, \cdots.
\end{align*}
By Young's inequality, the contribution of this is better than \eqref{Log improvement TT* Claim} when $\nu\geq 1$, and so, we only need to consider $\nu=0$. With this in mind, we may write, modulo $O(\lambda^{-1})$ errors,
\begin{align}\label{Cosine propagator without errors}
    \begin{split}
        (\cos t\sqrt{-\Delta_{\Tilde{g}}}) (\Tilde{z}, \alpha(\Tilde{y}))&=u_0 (\Tilde{z}, \alpha(\Tilde{y}))\partial_t E_0 (t, \Tilde{\rho}(\Tilde{z}, \alpha(\Tilde{y}))) \\
        &=\frac{1}{(2\pi)^2} u_0 (\Tilde{z}, \alpha(\Tilde{y})) \int e^{i\Phi(\Tilde{z}, \alpha(\Tilde{y}))\cdot \xi} \cos (t|\xi|)\:d\xi,
    \end{split}
\end{align}
where $|\Phi(\Tilde{z}, \alpha(\Tilde{y}))|=\Tilde{\rho} (\Tilde{z}, \alpha(\Tilde{y}))$ (cf. \cite[p.4026]{Chen2015improvement}, \cite{Sogge2014hangzhou}, etc.). Using (orthogonal) coordinate changes if necessary, we may assume that
\begin{align*}
    \Phi (\Tilde{z}, \alpha(\Tilde{y}))\cdot \xi= \Tilde{z}\cdot \xi, \quad \text{in normal coordinates at } \alpha(\Tilde{y}).
\end{align*}

Modulo $O(\lambda^{-1})$ errors, it follows from \eqref{Cosine propagator without errors} that
\begin{align*}
    & (Q_j \circ B_{t, j}\circ \cos (t\sqrt{-\Delta_g}) )(x, y) \\
    &=(2\pi)^{-2} \sum_\alpha \iint_{D^2} \int \Tilde{Q}_j (\Tilde{x}, \Tilde{w}) \Tilde{B}_{t, j} (\Tilde{w}, \Tilde{z}) u_0 (\Tilde{z}, \alpha(\Tilde{y})) e^{i\Phi(\Tilde{z}, \alpha(\Tilde{y}))\cdot \xi} \cos t|\xi|\:d\xi\:d\Tilde{w}\:d\Tilde{z} \\
    &=\frac{1}{2(2\pi)^2}\sum_\alpha \sum_\pm \iiint \Tilde{Q}_j (\Tilde{x}, \Tilde{w}) \Tilde{B}_{t, j}(\Tilde{w}, \Tilde{z}) u_0 (\Tilde{z}, \alpha(\Tilde{y})) e^{i\Phi (\Tilde{z}, \alpha(\Tilde{y}))\cdot \xi \pm it|\xi|}\:d\xi\:d\Tilde{w}\:d\Tilde{z}.
\end{align*}
We now write
\begin{align*}
    & (Q_j \circ B_{t, j}\circ \cos (t\sqrt{-\Delta_g})) (x, y) \\
    &=\frac{\lambda^6}{2(2\pi)^6} \sum_{\alpha, \pm}\int e^{i\lambda[(\Tilde{x}-\Tilde{w})\cdot \eta+(\Tilde{w}-\Tilde{z})\cdot \zeta+\Phi(\Tilde{z}, \alpha(\Tilde{y}))\cdot \xi\pm t|\xi|]} \Tilde{q}_j (\Tilde{x}, \Tilde{w}, \lambda \eta) \Tilde{b}_{t, j} (\Tilde{w}, \Tilde{z}, \lambda \zeta) u_0 (\Tilde{z}, \alpha(\Tilde{y})) \:d\Tilde{w}\:d\eta\:d\Tilde{z}\:d\zeta\:d\xi.
\end{align*}
By \eqref{Kj mod errors}, modulo $O(\lambda^{-1})$ errors, we write
\begin{align*}
    K_j (x, y)=\sum_{\alpha, \pm} U_{\alpha, j, \pm} (\Tilde{x}, \Tilde{y}),
\end{align*}
where
\begin{align*}
    & U_{\alpha, j, \pm} (\Tilde{x}, \Tilde{y})=\frac{\lambda^6}{(2\pi)^7 T} \int e^{i\lambda \Tilde{\psi}_{\alpha, \pm} (t, \xi, \Tilde{w}, \eta, \Tilde{z}, \zeta)} a_j (t, \Tilde{w}, \eta, \Tilde{z}, \zeta)\:d\Tilde{w}\:d\eta\:d\Tilde{z}\:d\zeta\:d\xi\:dt, \\
    & a_j (t, \Tilde{w}, \eta, \Tilde{z}, \zeta)=a_j (\Tilde{x}, \Tilde{y}; t, \Tilde{w}, \eta, \Tilde{z}, \zeta)=\widehat{\mu_T^2}(t/T) \Tilde{q}_j (\Tilde{x}, \Tilde{w}, \lambda \eta) \Tilde{b}_{t, j} (\Tilde{w}, \Tilde{z}, \lambda \zeta) u_0 (\Tilde{z}, \alpha (\Tilde{y})), \\
    & \Tilde{\psi}_{\alpha, \pm} (t, \xi, \Tilde{w}, \eta, \Tilde{z}, \zeta)=t+(\Tilde{x}-\Tilde{w})\cdot \eta+(\Tilde{w}-\Tilde{z})\cdot \zeta+\Phi(\Tilde{z}, \alpha(\Tilde{y}))\cdot \xi\pm t|\xi|.
\end{align*}

In geodesic normal coordinates centered at $\alpha(\Tilde{y})$, using suitable orthogonal coordinate changes, we have
\begin{align*}
    \Tilde{\psi}_{\alpha, \pm} (t, \xi, \Tilde{w}, \eta, \Tilde{z}, \zeta)=t+(\Tilde{x}-\Tilde{w})\cdot \eta+(\Tilde{w}-\Tilde{z})\cdot \zeta+\Tilde{z}\cdot \xi\pm t|\xi|.
\end{align*}
By Lemma \ref{Lemma alpha=Id}, we can focus only on $\alpha\not=\mathrm{Id}$. We would then have \eqref{Log improvement TT* Claim}, if we could show that
\begin{align}\label{Log imp TT^* U reduction}
    \left\|\int U_{\alpha, j, \pm} (\Tilde{\gamma}(\cdot), \Tilde{\gamma}(s)) f(s)\:ds \right\|_p \lesssim \frac{\lambda^{\frac{1}{2}}}{T} e^{CT} (2^{-j})^{\frac{2}{p}} \|f\|_{p'},\quad \alpha\not=\mathrm{Id},\quad 2\leq p<4,\quad \lfloor \log_2 \lambda^{\frac{1}{3}-\epsilon} \rfloor \leq j\leq J.
\end{align}
We have the following analysis for $U_{\alpha, j, \pm}$.

\begin{proposition}\label{Prop: SP results in Log Improvement}
For $\alpha \not=\mathrm{Id}$ fixed, we have, modulo $O(\lambda^{-1})$ errors, that
\[
U_{\alpha, j, \pm}(\Tilde{\gamma}(r), \Tilde{\gamma}(s))=\begin{cases}
\frac{\lambda^{\frac{1}{2}}}{T}e^{i\lambda \Tilde{\rho}(\Tilde{\gamma}(r), \alpha(\Tilde{\gamma}(s)))} \Tilde{a}_{\alpha, j}(r, s), & \text{if } |d_{\Tilde{x}} \Tilde{\rho}(\Tilde{\gamma}(r), \alpha(\Tilde{\gamma}(s)))(\Tilde{N})|\approx 2^{-j} \\
&\hspace{10pt} \text{and } |d_{\Tilde{y}}\Tilde{\rho}(\Tilde{\gamma}(r), \alpha(\Tilde{\gamma}(s)))(\alpha_*(\Tilde{N}))|\approx 2^{-j}, \\
O(\lambda^{-N}), & \text{otherwise},
\end{cases}
\]
where $\lfloor \log_2 \lambda^{\frac{1}{3}-\epsilon} \rfloor \leq j\leq J$, $\Tilde{\rho}=d_{\Tilde{g}}$, and $|\Tilde{a}_{\alpha, j} (r, s)|\lesssim e^{CT}$.
\end{proposition}
\begin{proof}
In normal coordinates at $\alpha(\Tilde{y})$, we write $U_{\alpha, j, \pm} (\Tilde{x}, \Tilde{y})$ as $U_{\alpha, j, \pm} (\Tilde{x})$
\begin{align*}
    U_{\alpha, j, \pm}(\Tilde{x})=\frac{\lambda^6}{(2\pi)^7 T}\int e^{i\lambda \Tilde{\psi}_{\alpha, \pm}(t, \xi, \Tilde{w}, \eta, \Tilde{z}, \zeta)} \Tilde{a}_j (t, \Tilde{w}, \eta, \Tilde{z}, \zeta)\:d\Tilde{w}\:d\eta\:d\Tilde{z}\:d\zeta\:d\xi\:dt,
\end{align*}
where $\Tilde{a}_j (t, \Tilde{w}, \eta, \Tilde{z}, \zeta)$ is the coordinate expression of $a_j (t, \Tilde{w}, \eta, \Tilde{z}, \zeta)$ in normal coordinates at $\alpha(\Tilde{y})$. Let $\Tilde{\beta}\in C_0^\infty (\mathbb{R}^2)$ be such that $\mathrm{supp}(\Tilde{\beta}) \subset \{\xi:c_2\leq |\xi|\leq c_2^{-1}\}$ for a small fixed $c_2>0$. We write
\begin{align*}
    U_{\alpha, j, \pm} (\Tilde{x})=U_{\alpha, j, \pm}^1 (\Tilde{x})+U_{\alpha, j, \pm}^2 (\Tilde{x}),
\end{align*}
where
\begin{align*}
    & U_{\alpha, j, \pm}^1 (\Tilde{x})=\frac{\lambda^6}{(2\pi)^7 T} \int e^{i\lambda \Tilde{\psi}_{\alpha, \pm} (t, \xi, \Tilde{w}, \eta, \Tilde{z}, \zeta)} \Tilde{\beta}(\xi)\Tilde{a}_j (t, \Tilde{w}, \eta, \Tilde{z}, \zeta)\:d\Tilde{w}\:d\eta\: d\Tilde{z}\:d\zeta\:d\xi\:dt, \\
    & U_{\alpha, j, \pm}^2 (\Tilde{x})=\frac{\lambda^6}{(2\pi)^7 T} \int e^{i\lambda \Tilde{\psi}_{\alpha, \pm} (t, \xi, \Tilde{w}, \eta, \Tilde{z}, \zeta)} (1-\Tilde{\beta}(\xi))\Tilde{a}_j (t, \Tilde{w}, \eta, \Tilde{z}, \zeta)\:d\Tilde{w}\:d\eta\: d\Tilde{z}\:d\zeta\:d\xi\:dt.
\end{align*}
We note that, choosing $c_2>0$ small in the support of $\Tilde{\beta}$, we have $|\partial_t \Tilde{\psi}_\alpha|=|1\pm |\xi||\gtrsim 1+|\xi|$. Thus, integrating by parts in $t$ as in the proof of Lemma \ref{Lemma alpha=Id}, we can write $U_{\alpha, j, \pm} (\Tilde{x}, \Tilde{y})$ as $U_{\alpha, j, \pm}^1 (\Tilde{x})$ in normal coordinates at $\alpha(\Tilde{y})$, and we focus on $U_{\alpha, j, \pm}^1 (\Tilde{x})$.

We will focus on the minus sign in the phase function. Indeed, if we choose the plus sign, then we have $\partial_t \Tilde{\psi}_{\alpha, +} =1+|\xi|>0$, and thus, there is no critical point of the phase function. Hence, integration by parts in $t$ again, we have $U_{\alpha, j, \pm}^1 (\Tilde{x})=O(\lambda^{-N})$. Set $\rho_0 =\Tilde{\rho} (\Tilde{x}, \alpha(\Tilde{y}))$. Since $\alpha\not=\mathrm{Id}$, we know $\rho_0>0$, and thus, can consider the following change of variables.
\begin{align}\label{Change of variables bar}
    \begin{split}
        \Bar{t}=\frac{t}{\sqrt{\rho_0}},\quad \Bar{w}=\frac{\Tilde{w}}{\sqrt{\rho_0}},\quad \Bar{\xi}=\sqrt{\rho_0}\xi,\quad \Bar{\eta}=\sqrt{\rho_0}\eta,\quad \Bar{z}=\frac{\Tilde{z}}{\sqrt{\rho_0}}, \quad \Bar{x}=\frac{\Tilde{x}}{\sqrt{\rho_0}},\quad \Bar{\zeta}=\sqrt{\rho_0} \zeta.
    \end{split}
\end{align}
This implies that
\begin{align*}
    d\Tilde{w}\:d\eta\:d\Tilde{z}\:d\zeta\:d\xi\:dt=\frac{1}{\sqrt{\rho_0}} d\Bar{w}\:d\Bar{\eta}\:d\Bar{z}\:d\Bar{\zeta}\:d\Bar{\xi}\:d\Bar{t}.
\end{align*}
Since we choose the minus sign in the phase function, we set
\begin{align*}
    \Tilde{\psi}_{\alpha, -} (t, \xi, \Tilde{w}, \eta, \Tilde{z}, \zeta)=\sqrt{\rho_0} \Bar{t}-\Bar{t}|\Bar{\xi}|+(\Bar{x}-\Bar{w})\cdot \Bar{\eta}+(\Bar{w}-\Bar{z})\cdot \Bar{\zeta}+\Bar{z}\cdot \Bar{\xi}=:\Bar{\psi}(\Bar{t}, \Bar{\xi}, \Bar{w}, \Bar{\eta}, \Bar{z}, \Bar{\zeta}).
\end{align*}
Note that
\begin{align*}
    \nabla \Bar{\psi}&=(\partial_{\Bar{t}}\Bar{\psi}, \partial_{\Bar{\xi}} \Bar{\psi}, \partial_{\Bar{w}} \Bar{\psi}, \partial_\eta \Bar{\psi}, \partial_{\Bar{z}} \Bar{\psi}, \partial_{\Bar{\zeta}} \Bar{\psi} ) \\
    &=(\sqrt{\rho_0}- |\Bar{\xi}|, \Bar{z}- \Bar{t} \frac{\Bar{\xi}}{|\Bar{\xi}|}, \Bar{\zeta}-\Bar{\eta}, \Bar{x}-\Bar{w}, -\Bar{\zeta}+\Bar{\xi}, \Bar{w}-\Bar{z} ),
\end{align*}
and thus, the critical point satisfies
\begin{align}\label{Stationary point condition}
    \begin{split}
        \sqrt{\rho_0} = |\Bar{\xi}|,\quad \Bar{z}=\Bar{t}\frac{\Bar{\xi}}{|\Bar{\xi}|},\quad \Bar{\zeta}=\Bar{\eta}, \quad \Bar{x}=\Bar{w},\quad \Bar{\zeta}=\Bar{\xi}, \quad \Bar{w}=\Bar{z}.
    \end{split}
\end{align}
The Hessian $\partial^2 \Bar{\psi}$ is
\[
    \partial^2 \Bar{\psi} =\begin{pmatrix}
    O_{1\times 1} & \left(-\frac{\Bar{\xi}^T}{|\Bar{\xi}|} \right)_{1\times 2} & O_{1\times 2} & O_{1\times 2} & O_{1\times 2} & O_{1\times 2} \\
    \left(-\frac{\Bar{\xi}}{|\Bar{\xi}|} \right)_{2\times 1} & A_{2\times 2} & O_{2\times 2} & O_{2\times 2} & I_{2\times 2} & O_{2\times 2} \\
    O_{2\times 1} & O_{2\times 2} & O_{2\times 2} & -I_{2\times 2} & O_{2\times 2} & I_{2\times 2} \\
    O_{2\times 1} & O_{2\times 2} & -I_{2\times 2} & O_{2\times 2} & O_{2\times 2} & O_{2\times 2} \\
    O_{2\times 1} & I_{2\times 2} & O_{2\times 2} & O_{2\times 2} & O_{2\times 2} & -I_{2\times 2} \\
    O_{2\times 1} & O_{2\times 2} & I_{2\times 2} & O_{2\times 2} & -I_{2\times 2} & O_{2\times 2}
    \end{pmatrix}=:\begin{pmatrix}
    B_{7\times 7} & C_{7\times 4} \\
    (C^T)_{4\times 7} & D_{4\times 4}
    \end{pmatrix},
  \]
where $A_{2\times 2}=\Bar{\psi}_{\Bar{\xi}\Bar{\xi}}''$. By properties of determinants for block matrices (cf. \cite{Powell2011DetOfBlockMatrices}, \cite{Silvester2000Determinants}, etc.), we have, at the critical point,
\begin{align*}
    \det(\partial^2\Bar{\psi})=\det(B-CD^{-1}C^T)\det D,
\end{align*}
provided the matrix $D$ is invertible. Since $\det D=1$, by properties of block matrix determinants again, we have, at the critical point,
\begin{align*}
    \det(\partial^2 \Bar{\psi})=\det (B-CD^{-1}C^T) &=\det \begin{pmatrix}
    O_{1\times 1} & \left(-\frac{\Bar{\xi}^T}{\sqrt{\rho_0}} \right)_{1\times 2} & O_{1\times 2} & O_{1\times 2} \\
    \left(-\frac{\Bar{\xi}}{\sqrt{\rho_0}} \right)_{2\times 1} & A_{2\times 2} & I_{2\times 2} & O_{2\times 2} \\
    O_{2\times 1} & I_{2\times 2} & O_{2\times 2} & -I_{2\times 2} \\
    O_{2\times 1} & O_{2\times 2} & -I_{2\times 2} & O_{2\times 2}
    \end{pmatrix}\\
    &=\det \begin{pmatrix}
    O_{1\times 1} & \left(-\frac{\Bar{\xi}^T}{\sqrt{\rho_0}} \right)_{1\times 2} & O_{1\times 2} & O_{1\times 2} \\
    \left(-\frac{\Bar{\xi}}{\sqrt{\rho_0}} \right)_{2\times 1} & A_{2\times 2} & O_{2\times 2} & O_{2\times 2} \\
    O_{2\times 1} & O_{2\times 2} & O_{2\times 2} & I_{2\times 2} \\
    O_{2\times 1} & O_{2\times 2} & I_{2\times 2} & O_{2\times 2}
    \end{pmatrix} \\
    &=\det\begin{pmatrix}
    0 & -\frac{\Bar{\xi}_1}{\sqrt{\rho_0}} & -\frac{\Bar{\xi}_2}{\sqrt{\rho_0}} \\
    -\frac{\Bar{\xi}_1}{\sqrt{\rho}_0} & -\frac{\Bar{t}}{\sqrt{\rho_0}^3} \Bar{\xi_2}^2 & \frac{\Bar{t}}{\sqrt{\rho_0}^3} \Bar{\xi}_1 \Bar{\xi}_2 \\
    -\frac{\Bar{\xi}_2}{\sqrt{\rho_0}} & \frac{\Bar{t}}{\sqrt{\rho_0}^3} \Bar{\xi}_1 \Bar{\xi}_2 & -\frac{\Bar{t}}{\sqrt{\rho_0}^3} \Bar{\xi}_1^2
    \end{pmatrix}=\frac{\Bar{t} |\Bar{\xi}|^4 }{\rho_0^{\frac{5}{2}}} = 1.
\end{align*}
In the last equality, we used $\Bar{t}, |\Bar{\xi}|= \sqrt{\rho_0}$ at the critical point, since, by \eqref{Change of variables bar} and \eqref{Stationary point condition}, we have that, for $\Bar{t}>0$,
\begin{align*}
    |\Bar{\xi}|= \sqrt{\rho_0}, \quad \Bar{t}=|\Bar{t}|=|\Bar{z}|=|\Bar{w}|=|\Bar{x}|=\frac{1}{\sqrt{\rho_0}}|\Tilde{x}|=\sqrt{\rho_0}.
\end{align*}
This gives us that $|\det \partial^2 \Bar{\psi}|=1$ at the critical point.

\begin{remark}
Since we have shown $\det (\partial^2 \Bar{\psi})=1$, we have
\begin{align*}
    (\partial^2 \Bar{\psi})^{-1}=\frac{1}{\det (\partial^2 \Bar{\psi})}\mathrm{adj}(\partial^2 \Bar{\psi})=\mathrm{adj}(\partial^2 \Bar{\psi}).
\end{align*}
Each entry of the adjugate $\mathrm{adj}(\partial^2 \Bar{\psi})$ is a finite linear combination of multiplications of terms of the form
\begin{align*}
    1,\; \frac{\Bar{\xi}_1}{\sqrt{\rho_0}},\; \frac{\Bar{\xi}_2}{\sqrt{\rho_0}},\; \frac{\Bar{t}}{\sqrt{\rho_0}^3}\Bar{\xi}_1^2,\; \frac{\Bar{t}}{\sqrt{\rho_0}^3}\Bar{\xi}_2^2,\; \frac{\Bar{t}}{\sqrt{\rho_0}^3} \Bar{\xi}_1 \Bar{\xi}_2.
\end{align*}
These are all $O(1)$ near the critical point, since we have $|\Bar{t}|, |\Bar{\xi}|\approx \sqrt{\rho_0}$. This implies that the matrix norm of $\partial^2 \Bar{\psi}$ is $O(1)$, and thus, we can use the method of stationary phase below easily.
\end{remark}

Continuing with our proof, in the normal coordinates at $\alpha(\Tilde{y})$, by the stationary phase argument, we have, modulo $O(\lambda^{-1})$ errors, at the critical point,
\begin{align*}
    U_{\alpha, j, -}^1 (\Tilde{x})=\frac{\lambda^6}{(2\pi)^7 \sqrt{\rho_0} T } \bigg[ \left(\frac{\lambda}{2\pi} \right)^{-\frac{11}{2}} e^{i\lambda \sqrt{\rho_0}|\Bar{x}|} 
    e^{\frac{i\pi}{4}\mathrm{sgn}(\partial^2 \Bar{\psi})} \sum_{l<l_0} \lambda^{-l} L_l a_0 + O\left( \lambda^{-l_0} \sum_{|\beta|\leq 2 l_0} \sup|D^\beta a_0| \right) \bigg],
\end{align*}
where $a_0$ is defined by
\begin{align*}
    a_0 (\Bar{t}, \Bar{\xi}, \Bar{w}, \Bar{\eta}, \Bar{z}, \Bar{\zeta})=\Tilde{\beta}\left(\frac{\Bar{\xi}}{\sqrt{\rho_0}} \right)\widehat{\mu_T^2}\left(\frac{\sqrt{\rho_0}\Bar{t}}{T} \right) \Tilde{q}_j \left(\sqrt{\rho_0}\Bar{x}, \sqrt{\rho_0}\Bar{w}, \lambda \frac{\Bar{\eta}}{\sqrt{\rho_0}}\right) \Tilde{b}_{\sqrt{\rho_0}\Bar{t}, j}\left(\sqrt{\rho_0}\Bar{w}, \sqrt{\rho_0} \Bar{x}, \lambda \frac{\Bar{\eta}}{\sqrt{\rho_0}}\right) u_0 (\sqrt{\rho_0}\Bar{z}),
\end{align*}
$u_0 (\sqrt{\rho_0}\Bar{z})$ is the coordinate expression of $u_0 (\sqrt{\rho_0}\Bar{z}, \alpha(\Tilde{y}))$ in normal coordinates at $\alpha(\Tilde{y})$, and the $L_l$ are the differential operators of order at most $2l$ acting on $a_0$ at the critical point. Recall that we can easily control the size estimates of $\Tilde{q}_j$ by construction, and the size estimates of $u_0$ by \cite[Lemma B.1]{Keeler2019TwoPointWeylLaw}. Also, by \cite{BouzouinaRobert2002Duke} and/or \cite[Lemma 11.11]{Zworski2012Semiclassical}, the size estimtaes for $\kappa_t^* q_j^*$ are the same as those for $q_j$, up to $e^{CT}$. Thus, the remainder is
\begin{align*}
    O\left( \lambda^{-l_0} \sum_{|\beta|\leq 2 l_0} \sup|D^\beta a_0| \right)=O(\lambda^{-l_0} (\lambda^{\frac{1}{3}})^{2 l_0} e^{CT})=O(\lambda^{-\frac{l_0}{3}} e^{CT}).
\end{align*}
Taking $l_0$ large enough, we can ignore the contribution of the remainder.

As before, by \eqref{Change of variables bar} and \eqref{Stationary point condition}, at the critical point, we have that
\begin{align*}
    \Bar{t}=|\Bar{z}|=|\Bar{w}|=|\Bar{x}|=\frac{1}{\sqrt{\rho_0}}|\Tilde{x}|,\quad |\Bar{\xi}|=\sqrt{\rho_0},\quad \Bar{\xi}=\frac{|\Bar{\xi}|}{\Bar{t}}\Bar{z},\quad \Bar{z}=\Bar{w}=\Bar{x}.
\end{align*}
This gives us that, in the geodesic normal coordinates,
\begin{align*}
    \Bar{\xi}=\frac{|\Bar{\xi}|}{\Bar{t}} \Bar{z}=\frac{\sqrt{\rho_0}}{|\Bar{x}|} \Bar{x}=\sqrt{\rho_0} \frac{\Tilde{x}}{|\Tilde{x}|}=\sqrt{\rho_0}\frac{d_{\Tilde{x}} \Tilde{\rho} (\Tilde{x}, \alpha(\Tilde{y}))}{|d_{\Tilde{x}} \Tilde{\rho} (\Tilde{x}, \alpha(\Tilde{y}))|_{\Tilde{g}} }.
\end{align*}

We then have, modulo $O(\lambda^{-1})$ errors, that
\begin{align*}
    U_{\alpha, j, -}^1 (\Tilde{x}, \Tilde{y})&=\frac{\lambda^{\frac{1}{2}}}{2\pi\sqrt{2\pi}T} \sum_{\alpha \not=\mathrm{Id}} \frac{1}{\sqrt{\rho_0}} e^{i\lambda \Tilde{\rho} (\Tilde{x}, \alpha(\Tilde{y}))} e^{\frac{i\pi}{4}\mathrm{sgn}(\partial^2\Bar{\psi})} \\
    & \hspace{5pt} \times \sum_{l<l_0} \lambda^{-l} L_l a_0 \bigg(\frac{|\Tilde{x}|}{\sqrt{\rho_0}}, \sqrt{\rho_0} d_{\Tilde{x}} \Tilde{\rho} (\Tilde{x}, \alpha(\Tilde{y})), \frac{\Tilde{x}}{\sqrt{\rho_0}}, \sqrt{\rho_0} d_{\Tilde{x}} \Tilde{\rho} (\Tilde{x}, \alpha(\Tilde{y})), \frac{\Tilde{x}}{\sqrt{\rho_0}}, \sqrt{\rho_0} d_{\Tilde{x}} \Tilde{\rho} (\Tilde{x}, \alpha(\Tilde{y})) \bigg) \\
    &=\frac{\lambda^{\frac{1}{2}}}{2\pi\sqrt{2\pi}T} \sum_{\alpha \not=\mathrm{Id}} \frac{1}{\sqrt{\rho_0}} e^{i\lambda \Tilde{\rho} (\Tilde{x}, \alpha(\Tilde{y}))} e^{\frac{i\pi}{4}\mathrm{sgn}(\partial^2\Bar{\psi})} \\
    & \hspace{5pt} \times \sum_{l<l_0} \lambda^{-l} L_l \bigg( \widehat{\mu_T^2}\left(|\Tilde{x}|/T \right) u_0 (\Tilde{x}, \alpha(\Tilde{y})) \Tilde{b}_{|\Tilde{x}|, j}(\Tilde{x}, \Tilde{x}, \lambda d_{\Tilde{x}} \Tilde{\rho} (\Tilde{x}, \alpha(\Tilde{y})) ) \Tilde{q}_j (\Tilde{x}, \Tilde{x}, \lambda d_{\Tilde{x}} \Tilde{\rho} (\Tilde{x}, \alpha(\Tilde{y})) ) \bigg).
\end{align*}
By the discussion in \S \ref{SS: Notation for PDOs} and the properties of the geodesic flow, we can write
\begin{align*}
    \Tilde{b}_{|\Tilde{x}|, j}(\Tilde{x}, \Tilde{x}, \lambda d_{\Tilde{x}} \Tilde{\rho} (\Tilde{x}, \alpha(\Tilde{y})) )&=\Tilde{q}_j (\kappa_{|\Tilde{x}|} (\Tilde{x}, \lambda d_{\Tilde{x}} \Tilde{\rho} (\Tilde{x}, \alpha(\Tilde{y})) )) \\
    &=\Tilde{q}_j (\alpha(\Tilde{y}), \alpha(\Tilde{y}), -\lambda d_{\Tilde{y}} \Tilde{\rho} (\Tilde{x}, \alpha(\Tilde{y})) ).
\end{align*}

Hence, modulo $O(\lambda^{-1})$ errors, for $\alpha\not=\mathrm{Id}$,
\begin{align*}
    U_{\alpha, j, -}^1 (\Tilde{x}, \Tilde{y})&=\frac{\lambda^{\frac{1}{2}}}{2\pi\sqrt{2\pi}T} e^{i\lambda \Tilde{\rho} (\Tilde{x}, \alpha(\Tilde{y})) } a_j (\Tilde{x}, \alpha(\Tilde{y})),
\end{align*}
where
\begin{align*}
    a_j (\Tilde{x}, \alpha(\Tilde{y}))=\sum_{l<l_0} \frac{1}{(\Tilde{\rho}(\Tilde{x}, \alpha(\Tilde{y})))^{\frac{1}{2}}} \lambda^{-l} L_l \bigg( \widehat{\mu_T^2}(\Tilde{\rho} (\Tilde{x}, \alpha(\Tilde{y}))/T) u_0 (\Tilde{x}, \alpha(\Tilde{y})) \Tilde{q}_j (\alpha(\Tilde{y}), \alpha(\Tilde{y}), -\lambda d_{\Tilde{y}} \Tilde{\rho}(\Tilde{x}, \alpha(\Tilde{y})))\\
    \times \Tilde{q}_j (\Tilde{x}, \Tilde{x}, \lambda d_{\Tilde{x}}\Tilde{\rho} (\Tilde{x}, \alpha(\Tilde{y})) ) \bigg).
\end{align*}
Since we have $\Tilde{\rho} (\Tilde{x}, \alpha(\Tilde{y})) \gtrsim 1$ for $\alpha\not=\mathrm{Id}$, we have, by construction,
\begin{align*}
    |a_j (\Tilde{x}, \alpha(\Tilde{y}))| \leq e^{CT}.
\end{align*}

Recall that the $\xi$-support of $q_j (x, y, \xi)$ is $\frac{|\xi(N)|}{|\xi|_g} \approx 2^{-j}$ and $\xi (N)=\langle \xi^{\#}, N \rangle_{\Tilde{g}}$. Note that, in geodesic normal coordinates centered at $\alpha(\Tilde{y})$, we have, for $\Tilde{\rho}(\Tilde{x}, \Tilde{z})$,
\begin{align*}
    |d_{\Tilde{x}} \Tilde{\rho}(\Tilde{x}, \alpha(\Tilde{y}))|_{\Tilde{g}}= 1 = |d_{\Tilde{z}} \Tilde{\rho}(\Tilde{x}, \alpha(\Tilde{y}))|_{\Tilde{g}}.
\end{align*}
By the support properties of $\Tilde{q}_j (\Tilde{x}, \Tilde{x}, \lambda d_{\Tilde{x}} \Tilde{\rho}(\Tilde{x}, \alpha(\Tilde{y})))$, $a_j (\Tilde{x}, \alpha(\Tilde{y}))$ is supported where
\begin{align*}
    |d_{\Tilde{x}}\Tilde{\rho}(\Tilde{x}, \alpha(\Tilde{y}))(\Tilde{N})| \approx 2^{-j}.
\end{align*}
Here, $\Tilde{N}$ is a unit normal vector to $\Tilde{\gamma}$. We also observe that $\alpha_* (\Tilde{N})$ is normal to $\alpha \circ \Tilde{\gamma}$. By the support properties of
\begin{align*}
    \Tilde{q}_j (\alpha(\Tilde{y}), \alpha(\Tilde{y}), -\lambda d_{\Tilde{y}} \Tilde{\rho}(\Tilde{x}, \alpha(\Tilde{y}))),
\end{align*}
for each $\alpha \not=\mathrm{Id}$, if $\Tilde{y}=\alpha\circ \Tilde{\gamma}(s)$ for $|s|\ll 1$, then $a_j (\Tilde{x}, \alpha(\Tilde{y}))$ is supported where
\begin{align*}
    |d_{\Tilde{y}}\Tilde{\rho}(\Tilde{x}, \alpha(\Tilde{y}))(\alpha_*(\Tilde{N}))|=|d_{\Tilde{y}} \Tilde{\rho}(\Tilde{x}, \alpha \circ \Tilde{\gamma}(s)) (\alpha_* (\Tilde{N}))| \approx 2^{-j}.
\end{align*}
This completes the proof.
\end{proof}

We next consider the support properties of the amplitude of $U_{\alpha, j, \pm}$. Let $\Tilde{\rho}_\alpha (\Tilde{x}, \Tilde{y})=\Tilde{\rho}(\Tilde{x}, \alpha(\Tilde{y}))$ for $\alpha\not=\mathrm{Id}$. Fix $r_0, s_0 \in [0, 1]$ so that
\begin{align*}
    |d_{\Tilde{x}} \Tilde{\rho}_\alpha (\Tilde{\gamma}(r_0), \Tilde{\gamma}(s_0))(\Tilde{N})|\approx 2^{-j} \quad \text{and} \quad |d_{\Tilde{y}}\Tilde{\rho}_\alpha(\Tilde{\gamma}(r_0), \Tilde{\gamma}(s_0))(\alpha_* (\Tilde{N}))|\approx 2^{-j}.
\end{align*}
We can assume such $r_0$ and $s_0$ exist, or otherwise, by the above proposition, we have $U_{\alpha, j, \pm}=O(\lambda^{-N})$ for any $N=1, 2, 3, \cdots$. Using a partition of unity, we may assume that
\begin{align*}
    |d_{\Tilde{x}} \Tilde{\rho}_\alpha (\Tilde{\gamma}(r), \Tilde{\gamma}(s))(\Tilde{N})|\approx 2^{-j} \quad \text{and} \quad |d_{\Tilde{y}}\Tilde{\rho}_\alpha(\Tilde{\gamma}(r), \Tilde{\gamma}(s))(\alpha_* (\Tilde{N}))|\approx 2^{-j}.
\end{align*}
happens only near $(r_0, s_0)$. By Proposition \ref{Prop: SP results in Log Improvement}, we may assume that $U_{\alpha, j, \pm} (\Tilde{\gamma}(r), \Tilde{\gamma}(s))$ is supported where
\begin{align*}
    |d_{\Tilde{x}} \Tilde{\rho}_\alpha (\Tilde{\gamma}(r), \Tilde{\gamma}(s))(\Tilde{N})|,\quad |d_{\Tilde{y}} \Tilde{\rho}_\alpha (\Tilde{\gamma}(r), \Tilde{\gamma}(s))(\alpha_*(\Tilde{N}))| \in [2^{-j-1}, 2^{-j+1}].
\end{align*}
Suppose $r, s\in [0, \epsilon_1]=I$ for some small $\epsilon_1>0$, and write $I=\cup_k I_k$, where $\{I_k\}_k$ is a collection of almost disjoint intervals with $|I_k|\approx e^{-CT}$ for some large $C>0$. Let $r_0$ and $s_0$ be fixed points in a subinterval $I_k$ and $I_{k'}$, respectively. We want to show the following.

\begin{lemma}\label{lemma Hessian operator}
Suppose $|d_{\Tilde{x}} \Tilde{\rho}_\alpha (\Tilde{\gamma}(r_0), \Tilde{\gamma}(s_0))(\Tilde{N})|\in [2^{-j-1}, 2^{-j+1}]$. Then, choosing $C>0$ sufficiently large with $|I_k|\approx e^{-CT}$, there exists a uniform constant $\Tilde{C}>0$ such that, for $r$ and $r_0$ in a same subinterval $I_k$,
\begin{align*}
    |d_{\Tilde{x}} \Tilde{\rho}_\alpha (\Tilde{\gamma}(r), \Tilde{\gamma}(s_0))(\Tilde{N})|\not\in [2^{-j-1}, 2^{-j+1}],\quad \text{whenever } |r-r_0|\geq \Tilde{C} 2^{-j}.
\end{align*}

Similarly, if $|d_{\Tilde{y}} \Tilde{\rho}_\alpha (\Tilde{\gamma}(r_0), \Tilde{\gamma}(s_0))(\alpha_*(\Tilde{N}))|\in [2^{-j-1}, 2^{-j+1}]$, then, choosing $C>0$ sufficiently large with $|I_k|\approx e^{-CT}$, there exists a uniform constant $\Tilde{C}>0$ such that, for $s$ and $s_0$ in a same subinterval $I_k$,
\begin{align*}
    |d_{\Tilde{y}} \Tilde{\rho}_\alpha(\Tilde{\gamma}(r_0), \Tilde{\gamma}(s))(\alpha_*(\Tilde{N}))|\not\in [2^{-j-1}, 2^{-j+1}],\quad \text{whenever } |s-s_0|\geq \Tilde{C} 2^{-j}.
\end{align*}
\end{lemma}

Before we prove this lemma, we review some basic properties of the Hessian operator $\mathcal{H}_r$ in \cite{Lee2018secondEd}: Suppose $(M, g)$ is an $n$-dimensional Riemannian manifold, $U$ is a normal neighborhood of a point $p\in M$, and $r:U\to \mathbb{R}$ is the radial distance function from the point $p$ defined by
\begin{align}\label{radial dist fcn for Hessian opr}
    r(x)=\sqrt{(x_1)^2+\cdots + (x_n)^2},
\end{align}
where $(x_i)$ are normal coordinates on $U$ centered at $p$. We also define the radial vector field on $U\setminus \{p\}$, denoted by $\partial_r$, as
\begin{align*}
    \partial_r=\sum_{i=1}^n \frac{x_i}{r(x)} \frac{\partial}{\partial x_i}=\mathrm{grad}\:r
\end{align*}
(cf. \cite[Corollary 6.10]{Lee2018secondEd}), where $\mathrm{grad} f=(d_{\Tilde{x}} f)^\#$ is the Riemannian gradient of $f$ and $\#$ is the musical isomorphism sharp. Then, the $(1, 1)$-tensor field $\mathcal{H}_r=\nabla (\partial_r)$, defined by
\begin{align*}
    \mathcal{H}_r (w)=\nabla_w \partial_r, \quad \text{for all } w\in TM|_{U\setminus \{p\}},
\end{align*}
is called the Hessian operator of $r$, where $\nabla$ is the Levi-Civita connenction. By \cite[Lemma 11.1]{Lee2018secondEd}, $\mathcal{H}_r$ is self-adjoint, $\mathcal{H}_r (\partial_r)\equiv 0$, and the restriction of $\mathcal{H}_r$ to vectors tangents to a level set of $r$ is equal to the shape operator of the level set associated with the normal vector field $-\partial_r$.

\begin{proof}[Proof of Lemma \ref{lemma Hessian operator}]
We prove the second case in this lemma. Similar arguments will work on the first one. In this proof, $\nabla$ denotes the Levi-Civita connection. We write $\Tilde{\rho}_0 (\Tilde{y})=\Tilde{\rho}_{0, \alpha} (\Tilde{y})=\Tilde{\rho}_0 (\Tilde{\gamma}(r_0), \alpha(\Tilde{y}))$ so that $\Tilde{\rho}_0$ is the distance function as in \eqref{radial dist fcn for Hessian opr}, since $r_0$ is fixed. We also write the radial vector field as $\partial_{\Tilde{\rho}_0}=\frac{\partial}{\partial \Tilde{\rho}_0}$. Set
\begin{align*}
    h(s)=\langle \mathrm{grad}_{\Tilde{y}} \Tilde{\rho}_0 (\Tilde{\gamma}(s)), \alpha_*(\Tilde{N}) \rangle_{\Tilde{g}},
\end{align*}
where $\mathrm{grad}_{\Tilde{y}}$ is the gradient for $\Tilde{y}$. By assumption, we have $|h(s_0)|\in [2^{-j-1}, 2^{-j+1}]$. We will work in geodesic normal coordinates centered at $\Tilde{\gamma}(r_0)$.

We want to show that $|h'(s_0)|\approx 1$. Let $\Tilde{\eta}=\Tilde{\eta}_\alpha=\alpha\circ \Tilde{\gamma}$. We then have that
\begin{align}\label{h(s) derivative}
    \frac{d}{ds}(h(s))=\langle \nabla_{\dot{\Tilde{\eta}}(s)} \mathrm{grad}_{\Tilde{y}} \Tilde{\rho}_0 (\Tilde{\gamma}(s)), \alpha_*(\Tilde{N}) \rangle_{\Tilde{g}} +\langle \mathrm{grad}_{\Tilde{y}} \Tilde{\rho}_0 (\Tilde{\gamma}(s)), \nabla_{\dot{\Tilde{\eta}}(s)} \alpha_*(\Tilde{N}) \rangle_{\Tilde{g}}.
\end{align}
For the first term in the right hand side, note that 
\begin{align*}
    \langle \nabla_{\dot{\Tilde{\eta}}(s_0)} \mathrm{grad}_{\Tilde{y}} \Tilde{\rho}_0 (\Tilde{\gamma}(s_0)), \alpha_*(\Tilde{N}) \rangle_{\Tilde{g}}=\langle \mathcal{H}_{\Tilde{\rho}_0} (\dot{\Tilde{\eta}}(s_0)), \alpha_*(\Tilde{N})\rangle_{\Tilde{g}},
\end{align*}
where $\mathcal{H}_{\Tilde{\rho}_0}$ is the Hessian operator of $\Tilde{\rho}_0$. Since $\mathcal{H}_{\Tilde{\rho_0}} (\partial_{\Tilde{\rho}_0})\equiv 0$, we have
\begin{align}\label{Hessian at rho0}
    |\mathcal{H}_{\Tilde{\rho}_0} (\dot{\Tilde{\eta}}(s_0))|_{\Tilde{g}}=\left|\mathcal{H}_{\Tilde{\rho}_0} \left(\dot{\Tilde{\eta}}(s_0)-\frac{\partial}{\partial \Tilde{\rho}_0} \right) \right|_{\Tilde{g}}\leq \|\mathcal{H}_{\Tilde{\rho}_0}\|\left| \dot{\Tilde{\eta}}(s_0)-\frac{\partial}{\partial \Tilde{\rho}_0} \right|_{\Tilde{g}}\approx 2^{-j} \|\mathcal{H}_{\Tilde{\rho}_0}\|,
\end{align}
where $\|\mathcal{H}_{\Tilde{\rho}_0} \|$ denotes the operator norm of $\mathcal{H}_{\Tilde{\rho}_0}$ and $\frac{\partial}{\partial \Tilde{\rho}_0}$ is a radial vector at $\alpha \circ \Tilde{\gamma}(s_0)$. Here we used the fact that $\left| \dot{\Tilde{\eta}}(s_0)-\frac{\partial}{\partial \Tilde{\rho}_0} \right|_{\Tilde{g}}\approx 2^{-j}$, which follows from the assumption $|h(s_0)|\approx 2^{-j}$. For the first term in \eqref{h(s) derivative} on the right hand side, we continue to show that $\|\mathcal{H}_{\Tilde{\rho}_0}\|\lesssim 1$. Let
\[
s_c (t)=\begin{cases}
t, & \text{if } c=0, \\
R\sin \frac{t}{R}, & \text{if } c=\frac{1}{R^2}>0, \\
R\sinh \frac{t}{R}, & \text{if } c=-\frac{1}{R^2}.
\end{cases}
\]
If we call the curvature of our manifold $\kappa$, then we can assume $-1\leq \kappa \leq 0$ by scaling our metric if necessary. By the Hessian comparison (cf. Theorem 11.7 in \cite{Lee2018secondEd}), we have that
\begin{align}\label{Hess comparison result}
    \frac{1}{\Tilde{\rho}_0}\pi_{\Tilde{\rho}_0} =\frac{s_0'(\rho)}{s_0(\rho)} \pi_{\Tilde{\rho}_0} \leq \mathcal{H}_{\Tilde{\rho}_0} \leq \frac{s_{-1}'(\rho)}{s_{-1}(\rho)}\pi_{\Tilde{\rho}_0}=\coth (\Tilde{\rho}_0)\pi_{\Tilde{\rho}_0},
\end{align}
where $\pi_{\Tilde{\rho}_0}$ is the orthogonal projection onto the tangent space of the level set of $\Tilde{\rho}_0$ as in \cite{Lee2018secondEd}. Here, $A\leq B$ means $\langle Av, v \rangle_{\Tilde{g}} \leq \langle Bv, v \rangle_{\Tilde{g}}$ for all vectors $v$. From the second inequality in \eqref{Hess comparison result}, we have
\begin{align*}
    \langle \mathcal{H}_{\Tilde{\rho}_0} v, v\rangle_{\Tilde{g}} \leq \coth(\Tilde{\rho}_0)\langle \pi_{\Tilde{\rho}_0} v, v \rangle_{\Tilde{g}} \lesssim \langle \pi_{\Tilde{\rho}_0} v, \pi_{\Tilde{\rho}_0} v+(v-\pi_{\Tilde{\rho}_0} v)\rangle_{\Tilde{g}}=|\pi_{\Tilde{\rho}_0} v|_{\Tilde{g}}^2 \leq |v|_{\Tilde{g}}^2.
\end{align*}
Here, we used the fact that $\coth(\rho)=\frac{e^\rho +e^{-\rho}}{e^{\rho}-e^{-\rho}}$ with $1\lesssim\Tilde{\rho}_0\leq T$. We can make the same argument for the first inequality, and so, in summary, we have
\begin{align*}
    0\leq \frac{1}{\Tilde{\rho}_0}|\pi_{\Tilde{\rho}_0} v|_{\Tilde{g}}^2 \leq \langle \mathcal{H}_{\Tilde{\rho}_0} v, v \rangle_{\Tilde{g}} \lesssim |v|_{\Tilde{g}}^2,
\end{align*}
from which it follows that $0\leq |\langle \mathcal{H}_{\Tilde{\rho}_0} v, v \rangle| \lesssim |v|_{\Tilde{g}}^2$. Since $\mathcal{H}_{\Tilde{\rho}_0}$ is self-adjoint (cf. Lemma 11.1 in \cite{Lee2018secondEd}), what we have shown is
\begin{align*}
    \|\mathcal{H}_{\Tilde{\rho}_0}\|=\sup_{|v|_{\Tilde{g}}=1} |\langle \mathcal{H}_{\Tilde{\rho}_0} v, v \rangle_{\Tilde{g}}|\lesssim \sup_{|v|_{\Tilde{g}}=1} |v|_{\Tilde{g}}^2=1.
\end{align*}
Combining this with \eqref{Hessian at rho0}, \eqref{h(s) derivative} is translated into
\begin{align}\label{h(s0) derivative}
    \left.\frac{d}{ds}(h(s))\right.\bigg|_{s=s_0}=O(2^{-j})+\langle \mathrm{grad}_{\Tilde{y}} \Tilde{\rho}_0 (\Tilde{\gamma}(s_0)), \nabla_{\dot{\Tilde{\eta}}(s_0)} \alpha_*(\Tilde{N}) \rangle_{\Tilde{g}}=O(2^{-j})+\left\langle \frac{\partial}{\partial \Tilde{\rho}_0}, \nabla_{\dot{\Tilde{\eta}}(s_0)} \alpha_*(\Tilde{N}) \right\rangle_{\Tilde{g}}.
\end{align}

For the second term in \eqref{h(s) derivative}, we first note that
\begin{align*}
    \nabla_{\dot{\Tilde{\eta}}(s)} \dot{\Tilde{\eta}}(s)=\langle \nabla_{\dot{\Tilde{\eta}}(s)} \dot{\Tilde{\eta}} (s),  \alpha_* (\Tilde{N})\rangle_{\Tilde{g}} \alpha_* (\Tilde{N}).
\end{align*}
Indeed, since $\Tilde{\eta}$ can be parametrized by arc length, we have
\begin{align*}
\langle \nabla_{\dot{\Tilde{\eta}} (s)} \dot{\Tilde{\eta}}(s), \dot{\Tilde{\eta}}(s) \rangle_{\Tilde{g}}=\frac{1}{2}\nabla_{\dot{\Tilde{\eta}}(s)} (\langle \dot{\Tilde{\eta}}(s), \dot{\Tilde{\eta}}(s) \rangle_{\Tilde{g}})=\frac{1}{2}\nabla_{\dot{\Tilde{\eta}}(s)} 1=0,
\end{align*}
which in turn implies that
\begin{align*}
    \nabla_{\dot{\Tilde{\eta}}(s)} \dot{\Tilde{\eta}}(s)=\langle \nabla_{\dot{\Tilde{\eta}}(s)} \dot{\Tilde{\eta}}(s), \dot{\Tilde{\eta}}(s) \rangle_{\Tilde{g}} \dot{\Tilde{\eta}}(s)+\langle \nabla_{\dot{\Tilde{\eta}}(s)} \dot{\Tilde{\eta}} (s), \alpha_*(\Tilde{N}) \rangle_{\Tilde{g}} \alpha_* (\Tilde{N})=\langle \nabla_{\dot{\Tilde{\eta}}(s)} \dot{\Tilde{\eta}} (s), \alpha_*(\Tilde{N}) \rangle_{\Tilde{g}} \alpha_* (\Tilde{N}).
\end{align*}
Since $\alpha_*(\Tilde{N})$ and $\dot{\Tilde{\eta}}(s)$ are orthogonal (at $\Tilde{\eta}(s_0)$), we have
\begin{align*}
    |\nabla_{\dot{\Tilde{\eta}}(s_0)} \dot{\Tilde{\eta}}(s_0)|_{\Tilde{g}}=|\langle \nabla_{\dot{\Tilde{\eta}}(s_0)} \dot{\Tilde{\eta}}(s_0), \alpha_*(\Tilde{N}) \rangle_{\Tilde{g}}|&=|\langle \mathrm{\RomanNumeralCaps{2}} (\dot{\Tilde{\eta}}(s_0), \dot{\Tilde{\eta}}(s_0)), \alpha_*(\Tilde{N}) \rangle_{\Tilde{g}}| \\
    &=|\langle \dot{\Tilde{\eta}}(s_0), W_{\alpha_*(\Tilde{N})}(\dot{\Tilde{\eta}}(s_0)) \rangle_{\Tilde{g}}|\\
    &=|\langle \dot{\Tilde{\eta}}(s_0), \nabla_{\dot{\Tilde{\eta}}(s_0)} \alpha_*(\Tilde{N})\rangle_{\Tilde{g}}|\approx \left|\left\langle \frac{\partial}{\partial \Tilde{\rho}_0 }, \nabla_{\dot{\Tilde{\eta}}(s_0)} \alpha_*(\Tilde{N}) \right\rangle_{\Tilde{g}} \right|,
\end{align*}
where the map $W_N$ is the Weingarten map in the direction of $N$ and $\mathrm{\RomanNumeralCaps{2}}$ is the second fundamental form of $\alpha(\gamma)$ in the universal cover $(\mathbb{R}^2, \Tilde{g})$. In the last approximation, we used the fact $\left|\frac{\partial}{\partial \Tilde{\rho}_0}-\dot{\Tilde{\eta}}(s_0) \right|_{\Tilde{g}}\approx 2^{-j}$, where $j$ is large enough.

Since we know $|\nabla_{\dot{\Tilde{\eta}}(s)} \dot{\Tilde{\eta}}(s)|_{\Tilde{g}}\approx 1$ by the assumption on the curvature of the given curve $\gamma$, we have
\begin{align*}
    \left|\left\langle \frac{\partial}{\partial \Tilde{\rho}_0}, \nabla_{\dot{\Tilde{\eta}}(s_0)} \alpha_*(N) \right\rangle_{\Tilde{g}}\right| \approx |\nabla_{\dot{\Tilde{\eta}}(s_0)} \dot{\Tilde{\eta}}(s_0)|_{\Tilde{g}} \approx 1.
\end{align*}
Combining this with \eqref{h(s0) derivative}, we have that $|h'(s_0)|\approx 1$.

By Taylor's formula,
\begin{align*}
    h(s)=h(s_0)+h'(s_0)(s-s_0)+O(|h''|(s-s_0)^2).
\end{align*}
As a consequence of \cite[Lemma B.2]{Keeler2019TwoPointWeylLaw}, there exists $C'>0$ such that
\begin{align*}
    h(s)=h(s_0)+h'(s_0)(s-s_0)+O(e^{C'T} (s-s_0)^2).
\end{align*}
Since we are assuming $|s-s_0|\approx e^{-CT}$, for a sufficiently large $C>0$, we have
\begin{align*}
    h(s)=h(s_0)+(h'(s_0)+O(e^{(C'-C)T}))(s-s_0)\approx h(s_0)\pm|h'(s_0)|(s-s_0).
\end{align*}
Since we have shown $|h'(s_0)|\approx 1$, there exists $C_2>0$ such that if $|s-s_0|\geq C_2 2^{-j}$, then we have $|h(s)|\not\in [2^{-j-1}, 2^{-j+1}]$, which proves the lemma.
\end{proof}

By Lemma \ref{lemma Hessian operator}, we have, modulo $O(\lambda^{-1})$ errors, for $r\in I_k, s\in I_{k'}$,
\begin{align}\label{U alpha j after Hessian comp}
    \begin{split}
        U_{\alpha, j, \pm}(\Tilde{\gamma}(r), \Tilde{\gamma}(s))=
            \begin{cases}
                \frac{\lambda^{\frac{1}{2}}}{T}e^{i\lambda \Tilde{\rho}(\Tilde{\gamma}(r), \alpha(\Tilde{\gamma}(s)))} \Tilde{a}_{\alpha, j}(r, s), & \text{if } |r-r_0|\lesssim 2^{-j} \text{ and } |s-s_0|\lesssim 2^{-j}, \\
            O(\lambda^{-N}), & \text{otherwise},
            \end{cases}
    \end{split}
\end{align}
where $|\Tilde{a}_{\alpha, j}(r, s)|\leq C e^{CT}$. Here, there is at most one cube of sidelength $C 2^{-j}$ in $(r, s)\in I_k \times I_{k'}\subset I\times I=[0, \epsilon_1]^2$ for small $\epsilon_1>0$ such that the amplitude $\Tilde{a}_{\alpha, j}(r, s)$ is nonzero, and $(r_0, s_0)$ is the center of the cube $I_k \times I_{k'}$.

\begin{remark}
We observe that the way to find support properties of $U_{\alpha, j}$ here is similar to that of $K_j, +$ or $K_J$ in the previous section. We used the assumption of nonvanishing geodesic curvatures on $\gamma$ in both cases. We also used the properties of the Hessian operator and the Taylor expansion here, whereas used the properties of the solution to the eikonal equation $\varphi$ and the mean value theorem there.
\end{remark}

It follows from \eqref{U alpha j after Hessian comp} that
\begin{align*}
    \int |U_{\alpha, j, \pm}(\Tilde{\gamma}(r), \Tilde{\gamma}(s))|\:dr=\sum_k \int_{I_k} |U_{\alpha, j, \pm} (\Tilde{\gamma}(r), \Tilde{\gamma}(s))|\:dr \lesssim e^{C'T} \frac{\lambda^{\frac{1}{2}}}{T} 2^{-j}.
\end{align*}
Here, $e^{C'T}$ comes from the fact $|\Tilde{a}_{\alpha, j}(\Tilde{\gamma}(r), \Tilde{\gamma}(s))|\leq e^{C''T}$ and the fact that the number of $\{I_k\}$ is $e^{CT}$ up to some constant, and $2^{-j}$ comes from the support property $|r-r_0|\lesssim 2^{-j}$ in $I_k$ for some $k$. Similarly, we also have $\int |U_{\alpha, j, \pm}(\Tilde{\gamma}(r), \Tilde{\gamma}(s))|\:ds\lesssim e^{C'T} \frac{\lambda^{\frac{1}{2}}}{T} 2^{-j}$. By Young's inequality, we have
\begin{align*}
    \left\|\int U_{\alpha, j, \pm} (\Tilde{\gamma}(\cdot), \Tilde{\gamma}(s))f(s)\:ds \right\|_2 \lesssim \frac{\lambda^{\frac{1}{2}}}{T} e^{CT} 2^{-j} \|f\|_2.
\end{align*}
By \eqref{U alpha j after Hessian comp}, we also have that
\begin{align*}
    \left\|\int U_{\alpha, j, \pm} (\Tilde{\gamma}(\cdot), \Tilde{\gamma}(s))f(s)\:ds \right\|_\infty \lesssim \frac{\lambda^{\frac{1}{2}}}{T} e^{CT} \|f\|_1.
\end{align*}
By interpolation, we obtain
\begin{align*}
    \left\|\int U_{\alpha, j, \pm} (\Tilde{\gamma}(\cdot), \Tilde{\gamma}(s))f(s)\:ds \right\|_p \lesssim \frac{\lambda^{\frac{1}{2}}}{T} e^{CT} (2^{-j})^{\frac{2}{p}} \|f\|_{p'},\quad 2\leq p\leq \infty,
\end{align*}
which proves \eqref{Log imp TT^* U reduction}. This completes the proof.

\section{Proof of Corollary \ref{Cor: at p=4}}\label{S: Cor: at p=4}

Let $P=\sqrt{-\Delta_g}$, $\chi\in \mathcal{S}(\mathbb{R})$, and $\gamma$ be as above. In this section, we heavily borrow arguments from Xi and Zhang \cite{XiZhang2017improved}, which was also motivated by Bourgain \cite{Bourgain1991Besicovitch} and Sogge \cite{Sogge2017ImprovedCritical}. We first have an analogue of \cite[Lemma 1]{XiZhang2017improved}.
\begin{lemma}[Lemma 1 in \cite{XiZhang2017improved}]\label{Lemma l subsegment}
We set $\lambda^{-1}\leq l\leq 1$. Let $\gamma_l$ be a fixed subsegment of $\gamma$ with length $l$. We then have that
\begin{align*}
    \|\chi(\lambda-P) f\|_{L^2 (\gamma_l)} \lesssim \lambda^{\frac{1}{4}} l^{\frac{1}{4}} \|f\|_{L^2 (M)}.
\end{align*}
\end{lemma}

\begin{remark}\label{Remark for subseg of gamma}
\begin{enumerate}
    \item In fact, \cite[Lemma 1]{XiZhang2017improved} focuses on the case where $\gamma$ is a geodesic segment, but the argument there can also be considered for our $\gamma$, by using $\rho (\gamma(r), \gamma(s))\approx |r-s|$, which comes from $|r-s|\ll 1$ by a partition of unity if necessary, and \cite[Lemma 4.5]{BurqGerardTzvetkov2007restrictions}.
    \item A similar argument also works for a general curve $\gamma$, considering that $\rho (\gamma(r), \gamma(s))\approx |r-s|$ holds for any curve $\gamma$ by using a $C^\infty$ topology argument in \cite[Section 6]{BurqGerardTzvetkov2007restrictions}.
    \item As observed in \cite[Remark 1]{XiZhang2017improved}, a similar argument gives the same estimate for $\chi(T_0(\lambda-P))$ if $T_0 \geq 1$.
\end{enumerate}
\end{remark}

Let $T$ be as in \eqref{T Definition}. We show a weak $L^4$ estimate.

\begin{proposition}\label{Prop: a weak L4 estimate}
Suppose $(M, g)$ is a $2$-dimensional compact Riemannian manifold with nonpositive curvatures. Then, for $\lambda\gg 1$, we have
\begin{align*}
    \| \chi(T(\lambda-P)) \|_{L^2 (M) \to L^{4, \infty}(\gamma) }\lesssim \frac{\lambda^{\frac{1}{4}}}{(\log \lambda)^{\frac{1}{4}} }.
\end{align*}
\end{proposition}

To show this, we will need a result from B\'erard \cite{Berard1977onthewaveequation}.

\begin{lemma}[\cite{Berard1977onthewaveequation}]\label{Lemma: a bound of Berard}
Let $(M, g)$ be as above. Then there exists a constant $C=C(M, g)$ so that, for $T_0 \geq 1$ and $\lambda \gg 1$, we have that
\begin{align*}
    |\chi^2 (T_0 (\lambda-P)) (x, y)|\leq C \left[T_0^{-1}\left(\frac{\lambda}{\rho (x, y)} \right)^{\frac{1}{2}} +\lambda^{\frac{1}{2}} e^{CT} \right].
\end{align*}
\end{lemma}

We now show Proposition \ref{Prop: a weak L4 estimate}.
\begin{proof}[Proof of Proposition \ref{Prop: a weak L4 estimate}]
Assuming $\|f\|_{L^2 (M)}=1$, it suffices to show that
\begin{align*}
    |\{x\in \gamma: |\chi(T(\lambda-P)) f(x)|>\alpha\}|\leq C\alpha^{-4} \lambda (\log \lambda)^{-1}.
\end{align*}
By the Chebyshev inequality and Theorem \ref{Theorem Log Improvement}, we have
\begin{align*}
    |\{x\in \gamma: |\chi(T(\lambda-P)) f(x)|>\alpha\}|\leq \alpha^{-2}\int_\gamma |\chi(T(\lambda-P))f|^2\:ds \leq \alpha^{-2} \lambda^{\frac{1}{3}} (\log \lambda)^{-1}.
\end{align*}
Note that, for large $\lambda$,
\begin{align*}
    \alpha^{-2} \lambda^{\frac{1}{3}} (\log \lambda)^{-1} \ll \alpha^{-4} \lambda (\log \lambda)^{-1},\quad \text{if } \alpha^2 \leq \lambda^{\frac{2}{3}},\quad \text{i.e., } \alpha\leq \lambda^{\frac{1}{3}}.
\end{align*}

We are left to show that
\begin{align}\label{Weak L4 claim}
    |\{x\in \gamma: |\chi(T(\lambda-P)) f(x)|>\alpha\}|\leq C \alpha^{-4} \lambda (\log \lambda)^{-1},\quad \text{when } \alpha\geq \lambda^{\frac{1}{3}}, \text{ and } \|f\|_{L^2 (M)}=1.
\end{align}
We set
\begin{align*}
    A=A_\alpha=\{x\in \gamma: |\chi (T(\lambda-P)) f(x)|>\alpha\},\quad \text{and} \quad r=\lambda \alpha^{-4} (\log \lambda)^{-2}.
\end{align*}
We consider a disjoint union $A=\cup_j A_j$, where $|A_j|\approx r$. Replacing $A$ by a set of proportional measure, we may assume that $\mathrm{dist}(A_j, A_k)>C_1 r$, when $j\not= k$ for some $C_1>0$, which will be specified later.

Let $T_\lambda: \chi (T(\lambda-P)):L^2 (M) \to L^2 (\gamma)$, and let
\[
\psi_\lambda (x)=\begin{cases}
    \frac{T_\lambda f(x)}{|T_\lambda f(x)|}, & \text{if } T_\lambda f(x)\not=0,\\
    1, & \text{otherwise.}
\end{cases}
\]
We also write
\begin{align*}
    S_\lambda= T_\lambda T_\lambda^*,\quad \text{and } a_j =\overline{\psi_\lambda \mathds{1}_{A_j}}.
\end{align*}
By the Chebyshev inequality and Cauchy-Schwarz inequality, we have
\begin{align*}
    \alpha |A|\leq \left|\int_\gamma T_\lambda f \overline{\psi_\lambda \mathds{1}_A}\:ds \right|\leq \left|\int_\gamma \sum_j T_\lambda f a_j\:ds \right|=\left|\int_M \sum_j T_\lambda^* a_j f\:dV_g \right|\leq \left(\int_M |\sum_j T_\lambda^* a_j|^2\:dV_g \right)^{\frac{1}{2}}.
\end{align*}
We can then write
\begin{align*}
    \alpha^2 |A|^2\leq \RomanNumeralCaps{1}+\RomanNumeralCaps{2},
\end{align*}
where
\begin{align*}
    \RomanNumeralCaps{1}=\sum_j \int_M |T_\lambda^* a_j|^2 \:dV_g,\quad \RomanNumeralCaps{2}=\sum_{j\not=k} \int_\gamma S_\lambda a_j \overline{a_k}\:ds.
\end{align*}
By duality and Remark \ref{Remark for subseg of gamma}, we have that
\begin{align*}
    \RomanNumeralCaps{1}\leq C r^{\frac{1}{2}}\lambda^{\frac{1}{2}} \sum_j \int_\gamma |a_j|^2\:ds=C r^{\frac{1}{2}} \lambda^{\frac{1}{2}} |A|=C \lambda \alpha^{-2} (\log \lambda)^{-1} |A|.
\end{align*}
For \RomanNumeralCaps{2}, by Lemma \ref{Lemma: a bound of Berard}, we note that the kernel $K_\lambda (s, s')$ of $S_\lambda$ satisfies
\begin{align*}
    |K_\lambda (s, s')|\leq C\left[\frac{1}{T}\left(\frac{\lambda}{|s-s'|} \right)^{\frac{1}{2}}+\lambda^{\frac{1}{2}} e^{CT} \right]=C\left[\frac{1}{c_0 \log \lambda}\left(\frac{\lambda}{|s-s'|} \right)^{\frac{1}{2}} +\lambda^{\frac{1}{2}+Cc_0} \right],
\end{align*}
which in turn implies that
\begin{align*}
    \RomanNumeralCaps{2}\leq C\left[\frac{1}{c_0 \log \lambda} \left(\frac{\lambda}{C_1 r} \right)^{\frac{1}{2}}+\lambda^{\frac{1}{2}+Cc_0} \right]\sum_{j\not=k} \|a_j\|_{L^1} \|a_k\|_{L^1} \leq \left[\frac{C}{c_0 C_1^{\frac{1}{2}}} \alpha^2+C\lambda^{\frac{1}{2}+Cc_0} \right] |A|^2.
\end{align*}
We now take $c_0$ to be sufficiently small so that $C\lambda^{\frac{1}{2}+Cc_0}\leq \frac{1}{4}\lambda^{\frac{2}{3}} \leq \frac{1}{4}\alpha^2$, since $\lambda\gg 1$ and $\alpha\geq \lambda^{\frac{1}{3}}$. Given the small $c_0>0$, we take $C_1\gg 1$ so that $\frac{C}{c_0 C_1^{\frac{1}{2}}}\leq \frac{1}{4}$. It then follows that
\begin{align*}
    \RomanNumeralCaps{2}\leq \frac{1}{2}\alpha^2 |A|^2.
\end{align*}

Putting these all together, we have that
\begin{align*}
    \alpha^2 |A|^2 \leq \RomanNumeralCaps{1}+\RomanNumeralCaps{2}\leq C\lambda \alpha^{-2} (\log \lambda)^{-1}|A|+\frac{1}{2}\alpha^2 |A|^2,
\end{align*}
and thus,
\begin{align*}
    |A|\leq C\lambda\alpha^{-4}(\log \lambda)^{-1},\quad \text{if } \alpha\geq \lambda^{\frac{1}{3}},
\end{align*}
which proves \eqref{Weak L4 claim}. This completes the proof of Proposition \ref{Prop: a weak L4 estimate}.
\end{proof}

We are now ready to prove Corollary \ref{Cor: at p=4}. We first recall a special case of a result in Bak and Seeger \cite{BakSeeger2011Extensions}.

\begin{lemma}[\cite{BakSeeger2011Extensions}]\label{Lemma: Bak and Seeger}
Suppose $(M, g)$ is any $2$-dimensional Riemannian manifold. If $\gamma\subset M$ is a curve segment in $M$, then
\begin{align*}
    \| \mathds{1}_{[\lambda, \lambda+1]} (P) f\|_{L^{4, 2}(\gamma)}\lesssim \lambda^{\frac{1}{4}} \|f\|_{L^2 (M)},\quad \lambda \geq 1.
\end{align*}
\end{lemma}

We recall some properties of the Lorentz space $L^{p, q}(\gamma)$ (see also Grafakos \cite{Grafakos2014Classical}, etc.). First, for a function $u$ on $M$, the corresponding distribution function $d_u (\alpha)$ with respect to $\gamma$ is defined by
\begin{align*}
    d_u (\alpha)=|\{x\in \gamma: |u(x)|>\alpha\}|,\quad \alpha>0.
\end{align*}
The function $u^*$ is the nondecreasing rearrangement of $u$ on $\gamma$, defined by
\begin{align*}
    u^* (t)=\inf \{\alpha: d_u (\alpha)\leq t\},\quad t\geq 0.
\end{align*}
For $1\leq p\leq \infty$ and $1\leq q\leq \infty$, the Lorentz space $L^{p, q}(\gamma)$ is then
\begin{align*}
    L^{p, q} (\gamma)=\left\{u: \|u\|_{L^{p, q}(\gamma)}:=\left(\frac{q}{p}\int_0^\infty [t^{\frac{1}{p}} u^* (t)]^q\:\frac{dt}{t} \right)^{\frac{1}{q}} <\infty \right\}.
\end{align*}
It is also known that
\begin{align*}
    \| \cdot \|_{L^{p, p}(\gamma)}=\| \cdot \|_{L^p(\gamma)},\quad \text{and } \sup_{t>0} t^{\frac{1}{p}} u^* (t)=\sup_{\alpha >0} [d_u (\alpha)]^{\frac{1}{p}}.
\end{align*}

We now take $u=\mathds{1}_{[\lambda, \lambda+(\log \lambda)^{-1}]} (P)f$ with $\|f\|_{L^2 (M)}=1$. By Proposition \ref{Prop: a weak L4 estimate}, we have that
\begin{align}\label{A weak L4 and sup of u}
    \sup_{t>0} t^{\frac{1}{4}} u^* (t)\lesssim \|u\|_{L^{4, \infty}} \lesssim \frac{\lambda^{\frac{1}{4}}}{(\log \lambda)^{\frac{1}{4}} }.
\end{align}
Since $\mathds{1}_{[\lambda, \lambda+1]} (P)u=u$, by Lemma \ref{Lemma: Bak and Seeger}, we have that
\begin{align}\label{Est from Bak Seeger}
    \|u\|_{L^{4, 2}(\gamma)}\lesssim \lambda^{\frac{1}{4}} \|u\|_{L^2 (M)}\lesssim \lambda^{\frac{1}{4}}.
\end{align}
By \eqref{A weak L4 and sup of u} and \eqref{Est from Bak Seeger}, we have
\begin{align*}
    \|u\|_{L^4(\gamma)}=\left(\int_0^\infty [t^{\frac{1}{4}} u^* (t)]^4\:\frac{dt}{t} \right)^{\frac{1}{4}}\lesssim (\sup_{t>0} t^{\frac{1}{4}} u^* (t))^{\frac{1}{2}} \|u\|_{L^{4, 2}(\gamma)}^{\frac{1}{2}} \lesssim \left(\frac{\lambda^{\frac{1}{4}}}{(\log \lambda)^{\frac{1}{4}}} \right)^{\frac{1}{2}} \lambda^{\frac{1}{8}}=\frac{\lambda^{\frac{1}{4}}}{(\log \lambda)^{\frac{1}{8}}}.
\end{align*}

This completes the proof.

\bibliographystyle{amsalpha}
\bibliography{references}

\end{document}